\documentclass[11]{amsart}
\usepackage{m9macros}
\usepackage{mymacros}
\usepackage{tikz}
\usepackage{pgfplots}
\usetikzlibrary{spy,intersections, pgfplots.fillbetween}
\usepackage{tkz-euclide}
\usepackage{graphicx}
\usepackage{float}
\usepackage{enumerate}
\usepackage[english]{babel}
\usepackage{longtable}
\usepackage{tabularx}
\newcolumntype{C}[1]{>{\centering\arraybackslash}p{#1}}
\newcolumntype{L}[1]{>{\raggedright\arraybackslash}p{#1}}
\graphicspath{ {./images/} }
\usepackage[ruled,vlined]{algorithm2e}

\usepackage[utf8]{inputenc}
\usepackage{csquotes}
\usepackage[backend=biber,style=numeric,firstinits=true]{biblatex}
\usepackage{cleveref}

\newtheorem{theorem}{Theorem}[section]
\newtheorem{lemma}[theorem]{Lemma}
\newtheorem{proposition}[theorem]{Proposition}
\newtheorem{corollary}[theorem]{Corollary}
\newtheorem{example}[theorem]{Example}
\newtheorem{conjecture}[theorem]{Conjecture}
\theoremstyle{definition}
\newtheorem{definition}[theorem]{Definition}

\newtheorem{remark}[theorem]{Remark}
\newtheorem*{remark*}{Remark}

\title[Jung-type Inequalities \& Blaschke-Santaló Diagrams for Diameter Variants]{Jung-type Inequalities and Blaschke-Santaló Diagrams for Different Diameter Variants}
\author[R. Brandenberg]{Ren\'e Brandenberg}
\author[M. Runge]{Mia Runge}

\begin{document}

\maketitle

\section*{Abstract}
We study geometric inequalities for the circumradius and diameter with respect to general gauges, partly also involving the inradius and the Minkowski asymmetry. There are a number of options for defining the diameter of a convex body that fall apart when we consider non-symmetric gauges. These definitions correspond to different symmetrizations of the gauge, i.e.~means of the gauge $C$ and its origin reflection $-C$.

\section{Preliminaries and Introduction}
A compact convex set is called a \cemph{red}{convex body}. The family of convex bodies in $\R^n$ is denoted by $\CC^n$, the convex bodies in $\R^n$ excluding single points by $\bar{\CC}^n$, and those, which contain $0$ in their interior, by $\CC^n_0$.  The \cemph{red}{support function} of $C \in \CC^n$ is denoted by $h_C(\cdot): \R^n \to \R $, $h_C(a):=\max_{x\in C}a^Ts$, while we write $\cnorm{\cdot}{C}: \R^n \to [0, \infty], x \mapsto \cnorm{x}{C} = \min\set{\lambda \geq 0 : x \in \lambda C}$ for the corresponding \cemph{red}{gauge function}. Hyperplanes are denoted by $H_{(a,\beta)}:=\set{x\in\R^n : a^Tx= \beta}$ and halfspaces by  $H^{\leq}_{(a,\beta)}:=\set{x\in\R^n : a^Tx\leq \beta}$. We say that $H_{(a,\beta)}$ or $H^{\leq}_{(a,\beta)}$ \cemph{red}{support} $C\in\CC^n$ in $p\in C$, if $p^Ta=\beta$ and $C \subset H^{\leq}_{(a,\beta)}$. In that case $a$ is called an outer
normal vector of $C$ in $p$. The \cemph{red}{polar} of $C\in\CC^n$ is defined as $C^{\circ}:=\{a\in\R^n:h_C(a)\leq1\}$.
For any $X\subset \R^n$, the \cemph{red}{positive, linear, affine} and \cemph{red}{convex hull} are denoted by $\pos(X)$, $\lin(X)$, $\aff(X)$ and $\conv(X)$, respectively.
We call the convex hull of two points $x$ and $y$ a \cemph{red}{segment} and abbreviate it by $[x,y]$.  The boundary of $X$ is described by $\bd(X)$ and the interior by $\inte(X)$. For any $X,Y\subset\R^n$ and $\rho\in\R$, let $X+Y:=\{x+y: x\in X, y\in Y\}$ be the \cemph{red}{Minkowski sum} of $X$ and $Y$ and $\rho X:=\{\rho x: x\in X\}$ be the $\rho$-\cemph{red}{dilatation} of $X$. We abbreviate $\set{x}+Y=:x+Y$ and $(-1)X=:-X$.
If $X=-X$, the set $X$ is called \cemph{red}{0-symmetric}. If there exists $t\in\R^n$ such that $-(t+X)=t+X$, we say that $X$ is \cemph{red}{symmetric}.

Inequalities between geometric functionals, such as the inradius, circumradius, and diameter, form a central area of convex geometry. They play an important role in many classical works such as \cite{bonnesenfenchel,burago2013geometric, DGK,eggleston1966convexity,hardyIneqs} and are still of interest today \cite{jungpairmink,merino2013ratio, henk1992generalization, cifre2003complete}. These geometric inequalities have proven to be useful for many results in convexity and also have many applications such as providing bounds for approximation algorithms, for example as bounds on the size of core-sets for containment under homothety \cite{NoDimIndep}.
One of the first such inequalities has been given by Jung \cite{Jung1901}. It provides a bound for the diameter-circumradius ratio in euclidean spaces. Variants and natural extensions, given e.g.~in \cite{bohnenblust1938convex, sharpening, Dolnikov}, are typically subsumed under the term \cemph{red}{Jung-type}.

Even when the gauge is not symmetric, the in- and circumradius have a unified definition: The \cemph{red}{circumradius} of $K \in \CC^n$ with respect to $C\in \CC^n$ is defined as
\[
  R(K,C):=\inf\{\rho\geq 0: \exists t \in \R^n \text{ such that } K\subset t+\rho C\}
\]
and the \cemph{red}{inradius} as
\[
  r(K,C):=\sup\{\rho\geq 0: \exists t \in \R^n \text{ such that } t+ \rho C\subset K\}.
\]
One should recognize that the inradius can be expressed as a circumradius: $r(K,C)=R(C,K)^{-1}$.

However, the diameter does not have such a unified definition. Maybe the first diameter definition for non-symmetric gauges has been given by Leichtweiss \cite{leichtweiss}. It differs from the one which later has been studied most \cite{ CompleteSimpl, sharpening, DGK, gruenbaum1959some}. In the following we present four diameter definitions, including the two mentioned above, all being identical if the gauge is 0-symmetric. We show that each belongs to a different symmetrization of the gauge.

This work expands upon Mia Runge's master thesis \cite{MRmaster}.

As part of this investigation, we consider the following question: given a gauge $C\in\CC^n_0$ and values $(r,R,D)$, is there a convex body $K\in\CC^n$ such that its inradius \wrt~$C$ is $r$, its circumradius is $R$, and its diameter is $D$?  This kind of question can be answered by giving a system of inequalities such that for every triple fulfilling these inequalities there exists such a convex body $K$. Such systems have been considered first by Blaschke for the volume, surface area and mean width in (euclidean) 3-space \cite{blaschke} and by
Santaló for some triples of
functionals out of area, perimeter, circumradius, inradius, diameter and width for the euclidean planar case \cite{santalo}. Later, many of the missing triples from Santaló's list have been solved \cite{cifre03,delyon2021,cifre00Dwr,cifre02}
and also other functionals such as the first Dirichlet eigenvalue \cite{ftouhi} or the Cheeger constant \cite{ftouhi2022complete} or even four of these functionals \cite{3dimBSdiagram, ting2005extremal} have been taken into account. As for the well known Blaschke-diagram \cite{blaschke, sangwine1989missing}, no complete descriptions of such diagrams of bodies in three or higher dimensional spaces are known so far, not even for the triple of circumradius, inradius, and diameter. Partial results can be derived from  \cite{Dekster1985,Jung1901,vrecica1981note} (see \cite{BGMtalk} for an overview on the known material).

In \cite{ngons} a complete system of inequalities for the $(r,R,D)$-diagram with triangular gauges and the most common diameter is given. We do the same for the three previously mentioned diameter variants.
Studying these diagrams for triangular gauge bodies is interesting because they also provide us with some generally valid inequalities (for all possible gauge bodies), as we expected by the result in \cite{ngons}.

Moreover, we investigate completeness and completion aspects for these diameters.
A convex body is \cemph{red}{complete} if we cannot add points without increasing the diameter. A \cemph{red}{completion} of a convex body $K$ is a complete body of the same diameter as $K$ containing $K$. Such sets often provide us with extreme cases of geometric inequalities. While one might expect the gauge to be complete, as it is the case with symmetric gauges or the common diameter, this is not always true for non-symmetric gauges and the alternative diameter definitions we consider. We show that the question, about the shape of the completions of the gauge, is closely related to their symmetrizations.

From now on, let us allways assume that $n \geq 2$.
When considering the containment between different convex bodies, such as the mentioned symmetrizations, we use the following notation. For $K,C\in \CC^n$ we say that $K$ is \cemph{red}{optimally contained} in $C$ if $K\subset C$ and $R(K,C)=1$, which is abbreviated by $K\optc C$. The next proposition from \cite{NoDimIndep} characterizes optimal containment.
    \begin{proposition}
    \label{opt}
    Let $K,C\in \CC^n$ and $C$ fulldimensional. Then $K\optc C$ if and only if
    \begin{enumerate}
        \item[i)] $K\subset C$ and
        \item[ii)] for some $k\in\{2,...,n+1\}$, there exist $p^1,...,p^k\in \ext(K)\cap \bd(C)$ and halfspaces $H^{\leq}_{(a^{i},h_C(a^i))}$ supporting $C$ at $p^{i}$ with $a^{i}\in \R^n \setminus\{0\}$, $i\in[k]$, affinely independent, such that $0\in \conv(\{a^1,...,a^k\})$.
    \end{enumerate}
    \end{proposition}
    The symmetrizations which will later correspond to the diameter definitions are the \cemph{red}{minimum} $C_{\MIN}:=C\cap -C$, the \cemph{red}{harmonic mean} $C_{\HM}:=\left(\frac{C \pol - C\pol}{2}\right)\pol$, the \cemph{red}{arithmetic mean} $C_{\AM}:=\frac{C-C}{2}$, and the \cemph{red}{maximum} $C_{\MAX}:=\conv(C \cup -C)$. These symmetrizations and their relations are studied for example in \cite{firey,GeoMean} and later also in \cite{BDG, reversing}.
   The latter is motivated by the objective of achieving a better understanding of these diameters, in particular by relating them to each other through geometric inequalities.

    Except $C_{\AM}$, these symmetrizations depend decisively on the position of $C$ and thus the diameters would do, too. However, as said above, all symmetrization should coincide if $C$ is symmetric, which is achieved if and only if in those cases $C$ is 0-symmetric, i.e.~centered in 0. Thus it seems natural to consider in general only \enquote{somehow} centered gauges and the exact definition of the \enquote{somehow} should fit the theory.

    The \cemph{red}{Minkowski asymmetry} of $C\in \CC^n$ is defined as $s(C):=R(C,-C)$ and
    we say that $C$ is \cemph{red}{Minkowski-centered} if $C\optc -s(C)C$.
    For $C\in \bar{\CC}^n$, the range of the Minkowski asymmetry is $[1, n]$,
    where $s(C) = 1$ if and only if $C$ is symmetric and  $s(C) = n$ if and only if C is an $n$-dimensional
    simplex \cite{gruenbaum}.
    The symmetrizations can be ordered and the first and third containments are always optimal \cite{ BDG, firey}.
    \begin{proposition} \label{prop:firey-chain}
        \label{contsym}
        Let $C\in\CC^n_0$. Then,
        \begin{equation*}
            C_{\MIN}\optc C_{\HM} \subset C_{\AM} \optc C_{\MAX}.
        \end{equation*}
      \end{proposition}

    \begin{figure}[ht]
        \centering
    \includegraphics[scale=0.1]{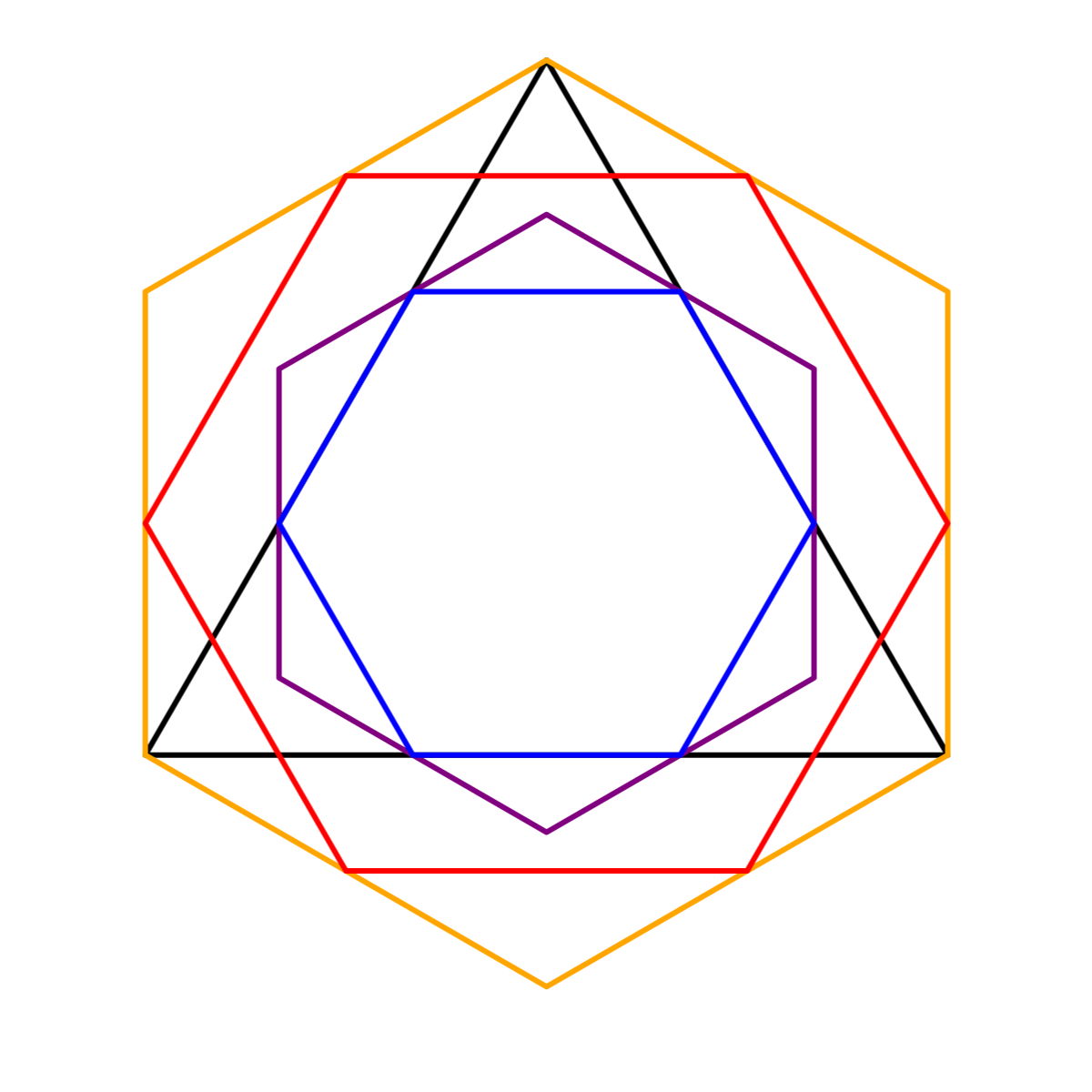}
    \caption{The equilateral triangle (black) and its symmetrizations: minimum (blue), harmonic mean (purple), arithmetic mean (red), maximum (orange) (\cf~\cite{BDG}).
    }
    \label{Symmetrizations}
  \end{figure}

\begin{remark*}
  From now on, unless otherwise specified, we always assume that $K\in \CC^n$, that the gauge $C \in \CC^n_0$ is a Minkowski-centered fulldimensional convex body and $s\in\R^n\setminus\set{0}$.
\end{remark*}

\section{Diameter Definitions} 

There are several ways to interpret the diameter in the symmetric case and we extend these ideas to the non-symmetric case. We will see that these correspond to different symmetrizations of the gauge. First, one can try to define the diameter of $K\in\CC^n$ by finding a "maximal" segment. But how should this "maximality" be defined? We could measure the distance between two points $x,y\in K$ using the gauge function $\cnorm{x-y}{C}$ and define the diameter as the maximal such distance $\max_{x,y \in K} \cnorm{x-y}{C}$. Or, we define it as the maximal circumradius of segments in $K$: $\max_{x,y \in K} 2R([x,y],C)$. However, the diameter could also be defined as the maximal distance between two parallel supporting hyperplanes of $K$.
\begin{equation*}
       \max_{s \in \bd(C^\circ)}   h_K(s)+h_{K}(-s) = \max_{s\in \R^n\setminus\set{0}} \frac{h_K(s)+h_{K}(-s)}{h_C(s)}.
\end{equation*}
As already mentioned, if the gauge $C$ is symmetric, all these definitions lead to the same diameter.
\begin{equation*}\max_{x,y\in K}\cnorm{x-y}{C}=\max_{x,y\in K}2R([x,y],C)
    =\max_{s\in \R^n\setminus\set{0}} \frac{h_K(s)+h_{K}(-s)}{h_C(s)}.
\end{equation*}
The most common diameter corresponds to the segement-radius definition and is therefore equal to two times the first core-radius of the set $K$ \cite{NoDimIndep}.
We call it the \cemph{red}{arithmetic diameter} (or standard diameter).
\begin{definition}
    \begin{enumerate}[i)]
        \item The $s$-length $l_{s,\AM}$ is defined as
                \begin{equation*}
                    l_{s,\AM}(K,C):= \max_{x-y\in (K-K) \cap \lin(s)}2R([x,y],C).
                \end{equation*}
        \item The $s$-breadth $b_{s,\AM}$ is defined as
                \begin{equation*}
                    b_{s,\AM}(K,C):= 2\cdot \frac{h_K(s)+h_{K}(-s)}{h_C(s)+h_C(-s)}.
                \end{equation*}
        \item The arithmetic diameter is defined as the maximal $s$-length:
                \begin{equation*}
                    D_{\AM}(K,C):=\max_{s\in\R^n\setminus \set{0}} l_{s,\AM}(K,C).
                \end{equation*}
    \end{enumerate}
\end{definition}
In \cite{sharpening} the following properties of $D_{\AM}$, which are well known for symmetric gauges (\cf~\cite{GritzmannKlee}), are proven.
\begin{proposition}
    \label{DamProperties}
    \begin{enumerate}[i)]
        \item \begin{equation*}
            D_{\AM}(K,C)=\max_{s\in\R^n\setminus \set{0}} b_{s,\AM}(K,C)
        \end{equation*}
        \item \begin{equation*}
            \begin{split}
                l_{s,\AM}(K,C)&=l_{s,\AM}\left(\frac{K-K}{2},C\right)=l_{s,\AM}\left(K,\frac{C-C}{2}\right)=l_{s,\AM}\left(\frac{K-K}{2},\frac{C-C}{2}\right)
            \end{split}
        \end{equation*}
        \item \begin{equation*}
            \begin{split}
                b_{s,\AM}(K,C)&=b_{s,\AM}\left(\frac{K-K}{2},C\right)=b_{s,\AM}\left(K,\frac{C-C}{2}\right)=b_{s,\AM}\left(\frac{K-K}{2},\frac{C-C}{2}\right)
            \end{split}
        \end{equation*}
        \item \begin{equation*}
            \begin{split}
                D_{\AM}(K,C)&=D_{\AM}\left(\frac{K-K}{2},C\right)=D_{\AM}\left(K,\frac{C-C}{2}\right)\\
                &=D_{\AM}\left(\frac{K-K}{2},\frac{C-C}{2}\right)=2 R\left(\frac{K-K}{2},\frac{C-C}{2}\right)
            \end{split}
        \end{equation*}
        \item \begin{equation*}
            \min_{s\in\R^n\setminus \set{0}} l_{s,\AM}(K,C)=\min_{s\in\R^n\setminus \set{0}} b_{s,\AM}(K,C) = 2r\left(\frac{K-K}{2},\frac{C-C}{2}\right)
        \end{equation*}
    \end{enumerate}
\end{proposition}
The fact that we can replace $C$ by its symmetrization $\frac{C-C}{2}$ is the reason why this diameter is called $D_{\AM}$. \\

Arguably the most intuitive way to measure a diameter is using the gauge function $\cnorm{x-y}{C}$. This diameter has been studied by Leichtweiss in \cite{leichtweiss} for non-symmetric gauges.
\begin{definition}
    \begin{enumerate}[i)]
        \item The asymmetric \cemph{red}{$s$-length} $l'_{s,\MIN}$ is defined as
        \begin{equation*}
            l'_{s,\MIN}(K,C):=\max_{x-y\in (K-K) \cap \pos(s)}\cnorm{x-y}{C}.
        \end{equation*}
        \item The symmetric \cemph{red}{$s$-length} $l_{s,\MIN}$ is defined as
        \begin{equation*}
            l_{s,\MIN}(K,C):=\max_{x-y\in (K-K) \cap \lin(s)}\cnorm{x-y}{C}.
        \end{equation*}
        \item The \cemph{red}{minimum diameter } is defined as the maximal symmetric $s$-length:
        \begin{equation*}
            D_{\MIN}(K,C):=\max_{s\in\R^n\setminus \set{0}} l_{s,\MIN}(K,C).
        \end{equation*}
    \end{enumerate}
\end{definition}
As we maximize over the $s$-lengths to obtain the diameter, both definitions of the $s$-length lead to the same diameter: $D_{\MIN}(K,C) = \max_{s\in\R^n\setminus \set{0}} l_{s,\MIN}(K,C) = \max_{s\in\R^n\setminus \set{0}} l'_{s,\MIN}(K,C)$.

\begin{lemma}\label{DminProperties}
  \begin{enumerate}[i)]
  \item
    \begin{align*}
      l_{s,\MIN}(K,C) &= \max_{x-y\in (K-K) \cap \pos(s)}\max(\cnorm{x-y}{C},\cnorm{x-y}{-C})\\
      &= l_{s,\MIN}(K,C\cap(-C))
      =l_{s,\AM}(K,C\cap(-C))\\
      &= l_{s,\MIN}\left(\frac{K-K}{2},C \cap (-C)\right)=l_{s,\MIN}\left(\frac{K-K}{2},C \right)
    \end{align*}
  \item
    \begin{equation*}
        \begin{split}
        D_{\MIN}(K,C)&=D_{\MIN}(K,C \cap (-C))=D_{\MIN}\left(\frac{K-K}{2},C \cap (-C)\right)\\
        &=D_{\MIN}\left(\frac{K-K}{2},C\right)=D_{\AM}(K,C \cap (-C))
    \end{split}
  \end{equation*}
  \end{enumerate}
\end{lemma}
\begin{proof}
  Since $\cnorm{v}{C\cap -C}=\max(\cnorm{v}{C}, \cnorm{-v}{C})$ for every $v\in\R^n$, we obtain
  \begin{equation*}
    \begin{split}
      \max_{x-y\in (K-K) \cap \pos(s)}\max(\cnorm{x-y}{C}, \cnorm{y-x}{C})&=\max_{x-y\in (K-K) \cap \lin(s)}\cnorm{x-y}{C}\\
      =\max_{x-y\in (K-K) \cap \lin(s)}\|x-y\|_{C\cap-C}&=\max_{x-y\in (K-K) \cap \lin(s)}2R([x,y],C\cap-C).
    \end{split}
  \end{equation*}
    This proves i) and ii) follows obviously.
\end{proof}

The next two diameters have been first introduced in \cite[Appendix]{reversing}.
For the first, instead of taking the maximum of $\cnorm{x-y}{C}$ and $\cnorm{x-y}{-C}$ one takes the arithmetic mean of these two values in the definition of the $s$-length. This diameter definition corresponds to the harmonic mean of the gauge.
\begin{definition}
    \begin{enumerate}[i)]
        \item The \cemph{red}{$s$-length} $l_{s,\HM}$ is defined as
        \begin{equation*}
            l_{s,\HM}(K,C):= \max_{x-y\in (K-K) \cap \lin(s)}\frac{1}{2}(\cnorm{x-y}{C}+\cnorm{x-y}{-C}).
        \end{equation*}
        \item The \cemph{red}{harmonic diameter}  is defined as the maximal $s$-length:
        \begin{equation*}
            D_{\HM}(K,C):=\max_{s\in\R^n\setminus \set{0}} l_{s,\HM}(K,C).
        \end{equation*}
    \end{enumerate}
\end{definition}
\begin{lemma} \label{DhmProperties}
    \begin{enumerate}[i)]
        \item  \begin{equation*}
        \begin{split}
            l_{s,\HM}(K,C)&= \max_{x-y\in (K-K) \cap \lin(s)} \cnorm{x-y}{\left(\frac{C\pol -C\pol}{2}\right)\pol} \\
            &= l_{s,\HM}\left(K,\left(\frac{C\pol -C\pol}{2}\right)\pol\right) = l_{s,\AM}\left(K,\left(\frac{C\pol -C\pol}{2}\right)\pol\right) \\
            &=l_{s,\HM}\left(\frac{K-K}{2},\left(\frac{C\pol -C\pol}{2}\right)\pol\right)= l_{s,\HM}\left(\frac{K-K}{2},C\right)
        \end{split}
            \end{equation*}
        \item \begin{equation*}
            \begin{split}
                D_{\HM}(K,C)&=D_{\HM}\left(K, \left(\frac{C\pol -C\pol}{2}\right)\pol\right)=D_{\HM}\left(\frac{K-K}{2}, \left(\frac{C\pol -C\pol}{2}\right)\pol\right)\\
                &=D_{\HM}\left(\frac{K-K}{2},C\right)=D_{\AM}\left(K, \left(\frac{C\pol -C\pol}{2}\right)\pol\right)
            \end{split}
        \end{equation*}
    \end{enumerate}
\end{lemma}
\begin{proof}
  We can use the fact that $h_{C\pol}(a)=\cnorm{a}{C}$ for any $a\in\R^n$ and $C\in\CC^n_0$ to obtain
  \begin{equation*}
    \begin{split}
      \|a\|_{\left(\frac{C\pol-C\pol}{2}\right)\pol}&=h_{\frac{C\pol-C\pol}{2}}(a)=\frac{1}{2}(h_{C\pol}(a)+h_{-C\pol}(a))\\
      &=\frac{1}{2}(\cnorm{a}{C}+\cnorm{a}{-C})=\frac{1}{2}(\cnorm{a}{C}+\cnorm{-a}{C}).
    \end{split}
  \end{equation*}
  Since $C_{\HM}$ is symmetric, the first part follows. The second part follows again directly from the first.
\end{proof}
Instead of dividing by the mean $\frac{h_C(s)+h_C(-s)}{2}$ in the definition of the $s$-breadth, one may also divide by the maximum of $h_C(s)$ and $h_C(-s)$. With this idea, we obtain our last diameter, the maximum diameter.
\begin{definition}
    \begin{enumerate}[i)]
        \item The \cemph{red}{$s$-breadth} $b_{s,\MAX}$ is defined as
        \begin{equation*}
            b_{s,\MAX}(K,C):= \frac{h_K(s)+h_{K}(-s)}{\max(h_{C}(s),h_C(-s))}.
        \end{equation*}
        \item The \cemph{red}{maximum diameter} is defined as the maximal $s$-breadth:
        \begin{equation*}
            D_{\MAX}(K,C):=\max_{s\in\R^n\setminus \set{0}} b_{s,\MAX}(K,C).
        \end{equation*}
    \end{enumerate}
\end{definition}
\begin{lemma} \label{DmaxProperties}
  \begin{enumerate}[i)]
  \item
    \begin{align*}
      b_{s,\MAX}(K,C)&= \frac{h_K(s)+h_{K}(-s)}{h_{\conv(C\cup (-C))}(s)} \\
      &=b_{s,\MAX}(K,\conv(C\cup (-C))) = b_{s,\AM}(K,\conv(C\cup (-C))) \\
      &=  b_{s,\MAX}\left(\frac{K-K}{2},\conv(C\cup (-C))\right) = b_{s,\MAX}\left(\frac{K-K}{2},C\right)
    \end{align*}
  \item
    \begin{equation*}
        \begin{split}
            D_{\MAX}(K,C)&=D_{\MAX}\left(K, \conv(C\cup (-C))\right)=D_{\MAX}\left(\frac{K-K}{2}, \conv(C\cup (-C))\right)\\
            &=D_{\MAX}\left(\frac{K-K}{2},C\right)=D_{\AM}(K, \conv(C\cup (-C))
        \end{split}
    \end{equation*}
  \end{enumerate}
\end{lemma}
\begin{proof}
 The first part follows directly from the fact that $\max(h_{C}(s),h_C(-s))=h_{C_{\MAX}}(s)$ and the second part again directly from the first.
\end{proof}

Since all definitions are equivalent for 0-symmetric gauges, we can always use that $D_{\m}(K,C)=D_{\m}(K,C_{\m})=D_{\AM}(K,C_{\m})$ for $\m \in \set{\MIN, \HM, \AM, \MAX}$ and results known about the arithmetic diameter. If we consider a symmetric gauge we omit the index and denote the diameter by $D(K,C)$.
Moreover, for all diameter definitions we say that $x,y\in K$ is a \cemph{red}{diametral pair} if
$ D_{\m}(K,C)=D_{\m}([x,y],C)$.

\begin{remark}
  \label{remarkwidth}
  Using the different definitons of the $s$-lengths and -breadths, width-definitions can be done analogously to the diameters. For $g\in\set{l,b}$ such that $g_{s,\m}$ is defined, the \cemph{red}{width} is defined as
  \begin{equation*}
    w_{\m}(K,C):= \min_{s\in\R^n\setminus\set{0}} g_{s,\m}(K,C).
  \end{equation*}
  In the standard case $\m=\AM$ it does not make a difference whether we minimize over the $s$-length or -breadth. By lemmas \ref{DminProperties}, \ref{DhmProperties}, and \ref{DmaxProperties} we can symmetrize the arguments of the width as well.
\end{remark}

\section{Properties of the Diameters and Blaschke-Santaló Diagrams}
In the following, we study properties of the diameters and how concepts such as completeness translate when using other definitions. Furthermore, we compare the diameter to other functionals such as the circumradius and inradius and introduce some theory on Blaschke-Santaló diagrams.

Let $\m\in\{\MIN,\HM,\AM,\MAX\}$ be one of the symmetrizations.

\begin{remark}
\label{contandfactor}
The inradius, circumradius and diameter are increasing and homogenious of degree $+1$ in the first argument and decreasing and homogenious of degree $-1$ in the second argument.
\end{remark}

\begin{lemma}[Linearity of $r,R,D_{\m}$]
\label{linearityrrRD}
Let $K_1,K_2\in \CC^n$ and $\lambda\in[0,1]$.
\begin{enumerate}
    \item[i)] If $r_1C\subset^{\opt}K_1$ and $r_2C\subset^{\opt}K_2$, then
      \begin{equation*} r(\lambda K_1 + (1-\lambda)K_2,C)\geq\lambda r_1 +(1-\lambda)r_2
      \end{equation*}
      and equality holds
      if we can choose the same outer normals of $K_1$ and $K_2$ in \Cref{opt}.
    \item[ii)] If $K_1\subset^{\opt}R_1C$ and $K_2\subset^{\opt}R_2C$, then
    \begin{equation*}
        R(\lambda K_1 + (1-\lambda)K_2,C)\leq\lambda R_1 +(1-\lambda)R_2.
    \end{equation*} If we can choose the same (up to dilatation) touching points $p^i$ in the boundary of $C$ as in in \Cref{opt}, equality is obtained.
    \item[iii)] If $D_1=D_{\m}(K_1,C)$ and $D_2=D_{\m}(K_2,C)$, then
    \begin{equation*}
        D_{\m}(\lambda K_1 + (1-\lambda)K_2,C)\leq\lambda D_1 +(1-\lambda)D_2.
    \end{equation*}
     If the diameters are defined by the same $s$-breadth or $s$-length, we have equality.
\end{enumerate}
\end{lemma}
\begin{proof}

\begin{enumerate}[i)]
    \item Obviously, $(\lambda r_1 +(1-\lambda)r_2) C \subset \lambda K_1 + (1-\lambda)K_2 $.  Let $r_1p_i$, $r_2p_i$, with $p_i \in \bd(C)$, be the touching points and $a_i$ the corresponding outer normals as in \Cref{opt}. Then, $h_{\lambda K_1 + (1-\lambda)K_2}(a_i)=\lambda r_1p_i^Ta_i+(1-\lambda)r_2p_i^Ta_i=(\lambda r_1 +(1-\lambda)r_2)p_i^Ta_i=h_{(\lambda r_1 +(1-\lambda)r_2) C}(a_i)$.  Thus, $p_i\in \bd((\lambda r_1 +(1-\lambda)r_2)C)\cap \bd(\lambda K_1 + (1-\lambda)K_2)$ and 0 is in the convex hull of the $a_i$, which shows that we have optimal containment.
    \item  The statement for the circumradius follows analogously. Here, the outer body is $C$ in both cases, so we automatically have the same supporting hyperplanes.
    \item We know $D_{\m}(K,C)=2R(K_{\AM},C_{\m})$. Thus, the inequality follows from part $ii)$. If the diameters are attained by the same $s$-length we have the same touching points in the containments $\frac{K_1-K_1}2 \optc \frac{D_{\m}(K_1,C)}{2}C_{\m}$ and $\frac{K_2-K_2}2 \optc \frac{D_{\m}(K_2,C)}{(2)}C_{\m}$ and the equality follows from part $ii)$. If the diameter is attained by the same $s$-breadth, we have
    \begin{equation*}
        \begin{split}
          \lambda D_{\m}(K_1,C) + (1-\lambda)D_{\m}(K_2,C)
          &=\lambda b_{s,\m}(K_1,C)+   (1-\lambda)b_{s,\m}(K_2,C) \\
          &= \lambda  \frac{h_{K_1-K_1}(s)}{h_{C_{\m}}(s)}+(1-\lambda)\frac{h_{K_2-K_2}(s)}{h_{C_{\m}}(s)} \\
          &= \frac{h_{(\lambda K_1 + (1-\lambda)K_2) - (\lambda K_1 + (1-\lambda)K_2)}(s)}{h_{C_{\m}}(s)} \\
          &= b_{s,\m}(\lambda K_1 + (1-\lambda)K_2,C) \\
          &\le D_{\m}(\lambda K_1 + (1-\lambda)K_2)
        \end{split}
      \end{equation*}
    and equality follows.
\end{enumerate}
\end{proof}

\begin{lemma}[Invariance under transformations]
\label{affineRrD} Let $A:\R^n\to\R^n$ be a non-singular affine transformation and $L_A$ its corresponding linear transformation. Then
\begin{equation*}
\begin{split}
    &D_{\m}(A(K),L_A(C))=D_{\m}(K,C)\\
    &R(A(K),A(C))=R(K,C)\\
    &r(A(K),A(C))=r(K,C).
\end{split}
\end{equation*}
\end{lemma}
\begin{proof}
    It follows from \Cref{opt}  that the in- and circumradius are invariant under affine transformations. All symmetrizations interchange with linear transformations: $(L_A(C))_{\m}=L_A(C_{\m})$ (see \cite{BDG}, Lemma 4). We must confine the transformation in the second argument of the diameter to be linear,
    since the symmetrizations (besides the arithmetic) are not invariant under translations.
    Because we can interprete the diameter as a circumradius with $D_{\m}(K,C)=2R(\frac{K-K}{2},C_{\m})$, the invariance of the diameter follows. The position of $K$ does not change the diameter and therefore we can apply a corresponding affine transformation to $K$.
  \end{proof}

To analyse properties such as constant width and completeness we extend their definitions to the different diameters.
\begin{definition}
Let $K,K^{*}\in \CC^n$, $K^{*} \supset K$, $C\in \CC^n_0$.
\begin{enumerate}[i)]
    \item $K$ is of \cemph{red}{constant width} if $w_{\m}(K,C)=D_{\m}(K,C)$.
    \item $K$ is \cemph{red}{complete}  if  $D_{\m}(K',C)>D_{\m}(K,C)$ for all $K'\in \CC^n$ such that $K' \supsetneq K$.
    \item $K^{*}$ is a \cemph{red}{completion} of $K$ if it is complete and $D_{\m}(K^{*},C)=D_{\m}(K,C)$.
\end{enumerate}
\end{definition}

\begin{remark}
  \label{remarkcwcompl}
Since all diameters can be expressed by the arithmetic diameter $\wrt$ $C_{\m}$, we know the following \cite{Eggleston}:
\begin{enumerate}[i)]
    \item $K$ has constant width if and only if $K_{\AM}=\lambda C_{\m}$ for some $\lambda\in \R$.
     \item If $K$ is of constant width, it is complete.
    \item In the planar case, $K$ is complete if and only if it has constant width.
\end{enumerate}
\end{remark}

\begin{lemma} \label{CCompletion}
Let $K\in \bar{\CC}^n$. If $K^{*}$ is a completion of $K$, then $K\subset^{\opt}K^{*}$.
\end{lemma}
\begin{proof}
  By definition $K\subset K^{*}$ and $D_{\m}(K^{*},C)=D_{\m}(K,C)$. Now, assuming there exist $c\in \R^n$ and  $0 \le \rho<1$ such that $K\optc c+\rho K^{*}$ implies $D_{\m}(K,C)\leq D_{\m}(c+\rho K^{*},C)<D_{\m}(K^{*},C)$, a contradiction.
\end{proof}

\begin{remark}
  Let us observe two things:
  \begin{enumerate}[i)]
  \item Whenever the gauge is symmetric, the only (up to translation and dilatation) complete and symmetric set is the gauge itself. Thus, when considering $D_{\m}$ with respect to a possibly non-symmetric gauge $C$, the only complete and symmetric set is always $C_{\m}$.
  \item  In the case that $\m=\AM$, the gauge itself is always complete. This is not always the case with other diameter definitions. Therefore, in the following sections, we will characterize when the gauge is complete and what the completion looks like.
  \end{enumerate}
\end{remark}

In the following the containment factors between $C_{\AM}$ and $C_{\m}$ will prove helpful to analyse the diameter $D_{\m}$.

\begin{notation}
By $\delta_{\m}:=\delta_{\m}(C)$ and $\rho_{\m}:=\rho_{\m}(C)$ we denote the dilatation factors needed, such that
\begin{equation*}
\label{deltarho}
    \rho_{\m}C_{\m}\optc C_{\AM} \optc \delta_{\m}C_{\m}.
\end{equation*}
These factors always exist since all symmetrizations are 0-symmetric and fulldimensional.
For better readability, we omit the argument $C$, the gauge body, whenever it is clear from the context.

Segments $L$ optimally contained in $C$ with $D(L,C)=2\rho_{\m}$ are denoted by $L_w$ and in case of $D(L,C)=2\delta_{\m}$ by $L_D$. 
\end{notation}

\begin{lemma} \label{CDiam}
  \begin{enumerate}[i)]
  \item For any segment $L \subset^{\opt}C$:
    \begin{equation*}
      2\rho_{\m}\leq D_{\m}(L,C)\leq2\delta_{\m}
    \end{equation*}
    with equality on the right side iff $L$ is diametral and equality on the left iff $L$ is a width chord of $C$.
    All values in between are attained.
  \item The diameter of $C$ with respect to itself is $D_{\m}(C,C)=2\delta_{\m}$, and the width of $C$ with respect to itself is $w_{\m}(C,C)=2\rho_{\m}$.
  \end{enumerate}
\end{lemma}

\begin{proof}
  We begin by showing part $ii)$:
  By the definition of $\delta_M$  as well as the diameter properties collected in Proposition \ref{DamProperties} and Lemmas \ref{DminProperties}, \ref{DhmProperties}, and \ref{DmaxProperties}, we have
  \begin{align*}
    D_{\m}(C,C) &= D_{\AM}(C_{\AM},C_{\m})=2R(C_{\AM},C_{\m})=2\delta_{\m}
                  \intertext{as well as}
               w_{\m}(C,C)&=w_{\AM}(C_{\AM},C_{\m})= 2r(C_{\AM},C_{\m}) = 2\rho_{\m}.
  \end{align*}
  Now, for part $i)$, let $L\subset C$ be a segment. It follows that $D_{\m}(L,C)\leq2\delta_{\m}$.
By definition, $L$ is the convex hull of a diametral pair if and only if $D_{\m}(L,C)=D_{\m}(C,C)=2\delta_{\m}$. \\
Segments with circumradius $1$ have diameter $2$ when considering the arithmetic mean. Hence, $2=D_{\AM}(L,C)\leq \frac{1}{\rho_{\m} }D_{\m}(L,C)$.  If $L$ provides us with the minimal $s$-length or -breadth, we obtain by part $ii)$: $D_{\m}(L,C)=\min_{s\in\R^n\setminus\{0\}}l_{s,\m}(C,C_{\m})=2\rho_{\m}$ or the analogue for the $s$-breadth.
  All values in between are attained since the $s$-length and $s$-breadth are continuous as a function of $s$ on $\R^n\setminus\{0\}$.
\end{proof}

\begin{definition}
  The set
  \begin{equation*}
    K^{\sup}=\bigcap_{x\in K} x+ D_{\m}(K,C)C_{\m}.
  \end{equation*}
  is called the \cemph{red}{supercompletion} of $K$.
\end{definition}

Let us remark that Moreno and Schneider \cite{MorenoSchneider} call $K^{\sup}$ the \emph{wide spherical hull}. It is shown in \cite{Eggleston} for arbitrary Minkowski spaces (i.e.~for $0$-symmetric $C$) that a set $K$ is complete \wrt~a symmetric gauge $C$ if and only if $K^{\sup}=K$ and in \cite{Moreno2011} that $K^{\sup}$ is the union of all completions of $K$. All the above were previously only defined for symmetric $C$,
but it is obvious that these properties stay true in the general case.

\begin{definition}
  \begin{enumerate}[i)]
  \item A \cemph{red}{supporting slab} of $K$ is the intersection of two antipodal parallel supporting halfspaces of $K$.
  \item A boundary point of $K$ is called smooth if the supporting hyperplane of $C$ in this point is unique.
  \item A supporting slab is \cemph{red}{regular} if at least one of the bounding hyperplanes contains a smooth boundary point of $K$.
  \item We say that $s\in\R^n\setminus\set{0}$ defines a regular slab if there exists a supporting slab such that the defining halfspaces have outer normals $\pm s$.
  \end{enumerate}
\end{definition}

It is easy to argue that a subdimensional convex body is never complete. On the the other hand, every fulldimensional convex body is the intersection of its regular slabs and completeness can be characterized by using these slabs \cite[Theorem 1]{MorenoSchneider}.

\begin{proposition}
  \label{MorSch}
  Let $K$ be fulldimensional. Then the following are equivalent:
  \begin{enumerate}[i)]
  \item $K$ is complete.
  \item For every outer normal $s$ defining a regular supporting slab of $K$ we have
  $\frac{h_K(s)+h_{-K}(s)}{h_{C_{\m}}(s)} = D_{\m}(K,C)$.
  \end{enumerate}
\end{proposition}

\begin{remark} \label{remarkXcompletion}
  As mentioned after the definition of the supercompletion, $K^* = K^{\sup}$ implies uniqueness for the completion $K^*$ of $K$. Thus, defining $K_X:=\bigcap_{x\in X} x+D_{\m}(K,C) C_{\m}$ for some $X \subset K$ we obviously have $K^* \subset K^{\sup} \subset K_X$. This means that describing properties for such subsets $X$ in the following that imply $K^* = K_X$ implicitely guarantee uniqueness of the completion $K^*$.

\end{remark}
\begin{lemma}
  \label{completion1}
  Let $X$ be a closed subset of $K$.
  Then the following are equivalent:
  \begin{enumerate}[i)]
  \item $K_X:=\bigcap_{x\in X} x+D_{\m}(K,C) C_{\m}$ is a completion of $K$.
  \item For every $s\in \R^n\setminus \set{0}$ that defines a regular slab of $C_{\m}$ there exist $\tilde s \in \set{s,-s}$ and $p\in X$ such that $p^T(-\tilde{s})=h_{K_X}(-\tilde{s})$ and $h_{K_X}(\tilde{s})=h_{p+D_{\m}(K,C)C_{\m}}(\tilde{s})$.
  \end{enumerate}
\end{lemma}

\begin{proof}
  Let us abbreviate $D:=D_{\m}(K,C)$ for the proof.
  $ii)\Rightarrow i)$: In \cite{MorenoSchneider} it is shown that the diameter is the supremum of the breadthes $b_{s}(K_X,C_{\m})$ where $s$ defines a regular slab of $C_{\m}$. For any such $s$ and $p \in X$ as defined in $ii)$ we have
  \begin{equation*}
    \begin{split}
      b_{s}(K_X,C_{\m})&=\frac{h_{K_X}(s)+h_{K_X}(-s)}{h_{C_{\m}}(s)}
                       =\frac{h_{p+D C_{\m}}(\tilde{s})+p^T(-\tilde{s})}{h_{C_{\m}}(\tilde{s})} \\
                     &=\frac{p^T\tilde{s}+D h_{C_{\m}}(\tilde{s})+p^T(-\tilde{s})}{h_{C_{\m}}(\tilde{s})}
                       =D.
    \end{split}
  \end{equation*}
  Thus $D_{\m}(K_X,C)=D$. 
  To show completeness of $K_X$ using \Cref{MorSch}, we need that all the regular slabs of $K_X$ are of diametral breadth. However, by the construction of $K_X$, every $s$ which defines a regular slab of $K_X$ also defines a regular slab of $C_{\m}$.

$i)\Rightarrow ii)$: Assume $K_X$ is a completion of $K$ and there exists some $s\in \R^n\setminus\set{0}$ that defines a regular slab such that there is no $p$ as defined in $ii)$. By the construction of $K_X$ there exist $p^1,p^2\in X$ such that $h_{K_X}(s)=(p^1)^T s+h_{D C_{\m}}(s)$ and $h_{K_X}(-s)=(p^2)^T(-s)+h_{D C_{\m}}(-s)$. By our assumption $(p^2)^Ts< (p^1)^Ts + h_{D C_{\m}}(s)$, otherwise we could choose $p=p^2$. Then,
\begin{equation*}
  \begin{split}
    b_{s}(K_X,C_{\m})&=\frac{h_{K_X}(s)+h_{K_X}(-s)}{h_{C_{\m}}(s)}
                       = \frac{(p^1)^T s + h_{D C_{\m}}(s)+(p^2)^T(-s)+h_{D C_{\m}}(-s)}{h_{C_{\m}}(s)}\\
                     &> \frac{(p^2)^T s+(p^2)^T(-s)+h_{D C_{\m}}(-s)}{h_{C_{\m}}(s)}
                       =\frac{h_{D C_{\m}}(-s)}{h_{C_{\m}}(s)}
                       =D = D_{\m}(K,C).
    \end{split}
\end{equation*}
which implies that $K_X$ is not a completion of $K$.
\end{proof}

Now, we consider the special case where $X$ is a simplex.

\begin{definition}
  We say that a subset $X\subset K \in \CC^n$ is a \cemph{red}{diametric simplex} of $K$ if
  \begin{enumerate}[i)]
  \item $X$ is a simplex, and
  \item $D_{\m}([x,y],K)=D_{\m}(K,C)$ for all pairs of vertices $x,y$ of $X$.
  \end{enumerate}
\end{definition}

\begin{lemma}
\label{diamtriagle2dim}
Let $X$ be a diametric triangle of $K\in \CC^2$. Then, $K_X$ is the unique completion of $K$. As a consequence, any triangle $T$ for which $X=T$ is diametric has a unique completion.
\end{lemma}

\begin{proof}
  We show that property $ii)$ of \Cref{completion1} is fulfilled. Assume \Wlog that $D_{\m}(X,C)=D_{\m}(K,C)=1$ and let $X=\conv \left(\set{p^1,p^2,p^3}\right)$.
  Then, the translations $-p^i+X$ with  $i\in\set{1,2,3}$ are subsets of $C_{\m}$, all  with one vertex in the origin and the other two on the boundary of $C_{\m}$ (\cf~\Cref{diamtrianglefig}).
  Since $p_i-p_j$ and $p_j-p_i$ are each other's negative, we have three pairs of points in the boundary of $C_{\m}$. For each, we choose an outer normal $a_k$, $k\in\set{1,2,3}$ ordered as given in \Cref{diamtrianglefig}. Now, the boundary of $K_X=K_{\ext(X)}$ consists of three parts which are built by parts of the boundary of $C_{\m}$ (colored in blue in the left part of \Cref{diamtrianglefig}).
Then, if $s\in \pos(\set{a_i,-a_j})$, $i\neq j$, property $ii)$ of \Cref{completion1} holds for $K_X$ with $p=p_k$, $k\neq i,j$. Hence, for all $s\in \R^n\setminus\set{0}$, property $ii)$ is fulfilled and it follows that $K_X$ is a completion.
Using \Cref{remarkXcompletion} we obtain the uniqueness.
\end{proof}

\begin{figure}[ht]
    \centering
    \begin{minipage}[c]{.4\linewidth} 
        \begin{tikzpicture}[scale=2.3]
            \tkzDefPoint(0,1){v1}
            \tkzDefPoint(-0.86602,-0.5){v2}
            \tkzDefPoint(0.86602,-0.5){v3}
            \tkzDefPoint(0,-1){v1m}
            \tkzDefPoint(0.86602,0.5){v2m}
            \tkzDefPoint(-0.86602,0.5){v3m}
             \tkzDrawPolygon[black, thick](v1,v2m,v3,v1m,v2,v3m);
             \tkzDefPoint(0.43301,0.75){w1}
             \tkzDrawPoint[radius=0.4pt](w1)
             \tkzDefPoint(0.86602,0){w2}
             \tkzDrawPoint[radius=0.4pt](w2)
             \tkzDefPoint(0.43301,-0.75){w3}
             \tkzDrawPoint[radius=0.4pt](w3)
             \tkzDefPoint(-0.43301,-0.75){w4}
             \tkzDrawPoint[radius=0.4pt](w4)
             \tkzDefPoint(-0.43301,0.75){w6}
             \tkzDrawPoint[radius=0.4pt](w6)
             \tkzDefPoint(-0.86602,0){w5}
             \tkzDrawPoint[radius=0.4pt](w5)
             \tkzDefPoint(0,0){z}
             \tkzDrawPolygon[black, thick](z,w1,w2);
             \tkzDrawPolygon[black, thick](z,w3,w4);
             \tkzDrawPolygon[black, thick](z,w5,w6);
             \tkzDrawPolySeg[blue,thick](w1,v2m,w2)
             \tkzDrawPolySeg[blue,thick](w3,v1m,w4)
             \tkzDrawPolySeg[blue,thick](w5,v3m,w6)
             \tkzDefPoints{1/0/a1,-0.5/-0.86602/a2,-0.5/0.86602/a3,-1/0/a1m,0.5/0.86602/a2m,0.5/-0.86602/a3m}
             \tkzDrawSegments[vector style](w1,a2m w2,a1 w3,a3m w4,a2 w5,a1m w6,a3)
             \tkzLabelSegment[right,pos=1.0](w2,a1){$a^1$}
             \tkzLabelSegment[left,pos=1.0](w5,a1m){$-a^1$}
             \tkzLabelSegment[above right,pos=1.0](w1,a2m){$-a^3$}
             \tkzLabelSegment[below left,pos=1.0](w4,a2){$a^3$}
             \tkzLabelSegment[below right,pos=1.0](w3,a3m){$-a^2$}
             \tkzLabelSegment[above left,pos=1.0](w6,a3){$a^2$}
             \tkzLabelSegment[below right, pos=0.3](w6,w1){$p^1-p^2$}
             \tkzLabelSegment[above right,pos=0.25](z,w2){$p^3-p^2$}
             \tkzLabelSegment[above left, pos=0.3](w3,w4){$p^2-p^1$}
             \tkzLabelSegment[above left,pos=0.7](w5,z){$p^2-p^3$}
        \end{tikzpicture}
    \end{minipage}
    \hspace{.1\linewidth}
    \begin{minipage}[c]{.4\linewidth} 
        \begin{tikzpicture}[scale=2]
        \tkzDefPoint(0,1){v1}
        \tkzDefPoint(-0.86602,-0.5){v2}
        \tkzDefPoint(0.86602,-0.5){v3}
        \tkzDefPoint(0,-1){v1m}
            \tkzDefPoint(0.86602,0.5){v2m}
            \tkzDefPoint(-0.86602,0.5){v3m}
            \tkzLabelPoint[above](v1){$p^1$}
            \tkzLabelPoint[below left](v2){$p^2$}
            \tkzLabelPoint[below right](v3){$p^3$}
            \tkzLabelPoint[below](v1m){$p^1+ DC_{\m}$}
            \tkzLabelPoint[above right](v2m){$p^2+DC_{\m}$}
            \tkzLabelPoint[above left](v3m){$p^3+DC_{\m}$}
        \tkzDrawPolygon[black, thick](v1,v2,v3);
        \tkzDrawPolygon[blue, thick](v1,v2m,v3,v1m,v2,v3m);
    \end{tikzpicture}
    \end{minipage}
    \caption{If $K$ contains a diametric triangle, its completion is constructed similar to the Reuleaux triangle in the euclidean case since it suffices to consider the extreme points, \ie~the vertices of a diametric triangle. }
    \label{diamtrianglefig}
 \end{figure}
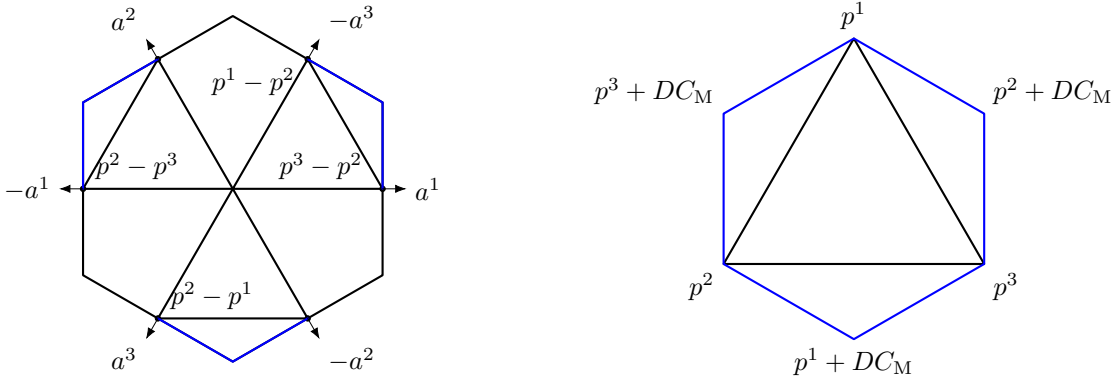

\label{subsecineqs}
From the containment chain in \Cref{prop:firey-chain} we know
\begin{equation*}
  D_{\MAX}(K,C)\leq D_{\AM}(K,C)\leq D_{\HM}(K,C)\leq D_{\MIN}(K,C)
\end{equation*}

The containment factors between the symmetrizations of the gauge can be used to improve this chain and to formulate new inequalities.

\begin{lemma} \label{lem:basicineq}
  \begin{enumerate}[i)]
  \item 
    $\displaystyle \rho_{\m}D_{\AM}(K,C)\leq D_{\m}(K,C)\leq\delta_{\m}D_{\AM}(K,C)$,
  \item  
    $\displaystyle \delta_{\m} r(K,C) \leq \frac{D_{\m}(K,C)}{2}$,
  \item
    $\displaystyle \frac{D_{\m}(K,C)}{2} \leq \delta_{\m} R(K,C)$,
  \item
    $\displaystyle \rho_{\m}(s(C)r(K,C)+R(K,C)) \leq (s(C)+1)\frac{D_{\m}(K,C)}{2}$, and
    \item
    $\displaystyle r(K,C)+R(K,C) \leq R(C_{\m},C) D_{\m}(K,C)$.
  \end{enumerate}
\end{lemma}

\begin{proof}
  \begin{enumerate}[i)]
  \item Follows directly from \Cref{contandfactor} and \Cref{contsym}.
  \item Since $r(K,C)C$ is contained in a translate of $K$, we obtain
    \begin{equation*}
      \delta_{\m} r(K,C) = \frac12 D_{\m}(C,C) r(K,C) = \frac 12 D_{\m}(r(K,C)C,C) \leq \frac{D_{\m}(K,C)}{2}.
    \end{equation*}
    \item Follows directly from \Cref{CDiam} and \Cref{contandfactor}.
  \item By \cite[Theorem 1.1]{CompleteSimpl} we have
    \[s(C)r(K,C)+R(K,C) \leq \frac{s(C)+1}{2}D_{\AM}(K,C)\]
    and \[(s(C)+1)\frac{D_{\AM}(K,C)}{2} \leq (s(C)+1)\frac{D_{\m}(K,C)}{2\rho_{\m}}\]
    follows directly from part $i)$.
    \item For the symmetrization $C_{\m}$ it holds $r(K,C_{\m})+R(K,C_{\m}) \leq D(K,C_{\m})$. Thus,
    \begin{equation*}
      D_{\m}(K,C)\geq r(K,C_{\M})+ R(K,C_{\m}) \geq \frac{1}{R(C_{\m},C)} \left(r(K,C)+R(K,C)\right).
    \end{equation*}
  \end{enumerate}
\end{proof}

We would like to describe the values the inradius, circumradius, and diameter of sets $K\in\CC^n$ may have, when we consider a fixed, Minkowski-centered $C\in \CC^n_0$. To do so, we study the following Blaschke-Santaló diagrams.
\begin{definition}
Let $f_{\m}$ be the following mapping.
\begin{equation}
    f_{\m}: \bar{\CC}^n\times\CC^n_0 \to \R^2, \: f_{\m}(K,C)=\left(\frac{r(K,C)}{R(K,C)}, \frac{D_{\m}(K,C)}{2R(K,C)}\right)
\end{equation}
The set $f_{\m}(\bar{\CC}^n,C)$ is called the \cemph{red}{Blaschke-Santaló diagram} for the inradius, circumradius,
and diameter (depending on the respective definitions) with regard to the gauge $C$ -- the ($r,R,D_{\m}$)-diagram.
\end{definition}
As for the diameter we only write $f$ if the gauge is symmetric.
In \cite{ngons}  $f_{\AM}(\bar{\CC}^2,S)$ is described and it is shown that this diagram is equal to the union of the diagrams over all possible gauges.

\begin{proposition}
  \label{prop:diagramAM}
  For every triangle $S\in\CC^2$, the diagram $f_{\AM}(\bar{\CC}^2,S)$ is fully described by the  inequalities
  \begin{align*}
    D_{\AM}(K,S)&\leq 2R(K,C)\\
    4r(K,S)+2R(K,S)&\leq 3 D_{\AM}(K,S)\\
    \frac{D_{\AM}(K,C)}{2R(K,C)}\left(1-\frac{D_{\AM}(K,C)}{2R(K,C)}\right) &\leq \frac{r(K,C)}{R(K,C)}.
  \end{align*}
  Moreover,  $f_{\AM}(\bar{\CC}^2,S)=f_{\AM}(\bar{\CC}^2,\CC^2_0)$.
\end{proposition}

The diagrams $f(\bar{\CC}^2,\B^2_2)$ \cite{santalo} (name giving) 
and $f_{\AM}(\bar{\CC}^2,S)$ \cite{ngons} can be seen in \Cref{FigBSSantalo}.

\begin{figure}[ht]
  \centering
\noindent\begin{minipage}[c]{.4\linewidth}
    \begin{tikzpicture}[scale=4]
        \draw[->] (-0.1,0) -- (1.1,0) node[right] {$\frac{r}{R}$};
        \draw[->] (0,-0.1) -- (0,1.2) node[above] {$\frac{D}{2R}$} ;
        \draw[ domain=0:1, smooth, variable=\x] plot ({\x}, {1}) ;
        \draw[ domain=0.73205:1, smooth, variable=\x] plot ({\x}, {0.5*(\x)+0.5});
        \draw[ domain=0.5:0.73205, smooth, variable=\x] plot ({\x}, {0.86602});
        \draw[ domain=0.86602:1, smooth, variable=\y] plot ({((2*(\y)^2)*(sqrt(1-(\y)^2)))/(1+sqrt(1-(\y)^2))},{\y} );
        \draw[ domain=0:0.5, dotted, variable=\x] plot ({\x}, {0.86602}) ;
         \node[label=below left:$\frac{\sqrt{3}}{2}$] (jung) at (0,0.86602) {.};
        \draw[ domain=0:1, dotted, variable=\y] plot ({1}, {\y}) ;
         \node[label=below:$1$] (one) at (1,0) {.};
        \draw[ domain=0:0.86602, dotted, variable=\y] plot ({0.73205}, {\y});
        \node[label=below:$\sqrt{3}-1$] (minus) at (0.73205,0) {.};
         \draw[ domain=0:0.86602, dotted, variable=\y] plot ({0.5}, {\y})
        ;
         \node[label=below left:$\frac{1}{2}$] (max) at (0.5,0) {.};
         \node[label=left:$1$] (delta) at (0,1) {.};
        \node[label=above:$\B^2_2$] (b2) at (1,1) {.};
        \draw[fill=black] (1,1) circle[radius=0.4pt];
        \node[label=above right:$L$] ($L$) at (0,1) {.};
        \draw[fill=black] (0,1) circle[radius=0.4pt];
        \node[label=below left:$T$] (T) at (0.5,0.86602) {.};
        \draw[fill=black] (0.5,0.86602) circle[radius=0.4pt];
         \node[label=below right:$\rt$] (Trot) at (0.73205,0.86602) {.};
        \draw[fill=black] (0.73205,0.86602) circle[radius=0.4pt];
    \end{tikzpicture}
\end{minipage}
\noindent\begin{minipage}[c]{.4\linewidth}
    \begin{tikzpicture}[scale=4]
        \draw[->] (-0.1,0) -- (1.1,0) node[right] {$\frac{r}{R}$};
        \draw[->] (0,-0.1) -- (0,1.2) node[above] {$\frac{D_{\AM}}{2R}$} ;
        \draw[ domain=0:1, smooth, variable=\x] plot ({\x}, {1}) ;
        \draw[ domain=0.25:1, smooth, variable=\x] plot ({\x}, {0.6666*(\x)+0.3333});
        \draw[ domain=0.5:1, smooth, variable=\y] plot ({(\y)*(1-\y)},{\y} );
        \draw[ domain=0:0.25, dotted, variable=\x] plot ({\x}, {0.5}) ;
         \node[label=left:$\frac{1}{2}$] (jung) at (0,0.5) {.};
        \draw[ domain=0:1, dotted, variable=\y] plot ({1}, {\y}) ;
         \node[label=below:$1$] (one) at (1,0) {.};
        \draw[ domain=0:0.5, dotted, variable=\y] plot ({0.25}, {\y});
        \node[label=below:$\frac{1}{4}$] (minus) at (0.25,0) {.};
         \node[label=left:$1$] (delta) at (0,1) {.};
        \node[label=above:$S$] (b2) at (1,1) {.};
        \draw[fill=black] (1,1) circle[radius=0.4pt];
        \node[label=above right:$L$] ($L$) at (0,1) {.};
        \draw[fill=black] (0,1) circle[radius=0.4pt];
        \node[label=below right:$-S$] (sm) at (0.25,0.5) {.};
        \draw[fill=black] (0.25,0.5) circle[radius=0.4pt];
    \end{tikzpicture}

\end{minipage}
    \caption{The $(r,R,D)$-diagram \wrt~$\B^2_2$ (left) and the $(r,R,D_{\AM})$-diagram \wrt~a triangle $S$ (right).}
    \label{FigBSSantalo}
\end{figure}
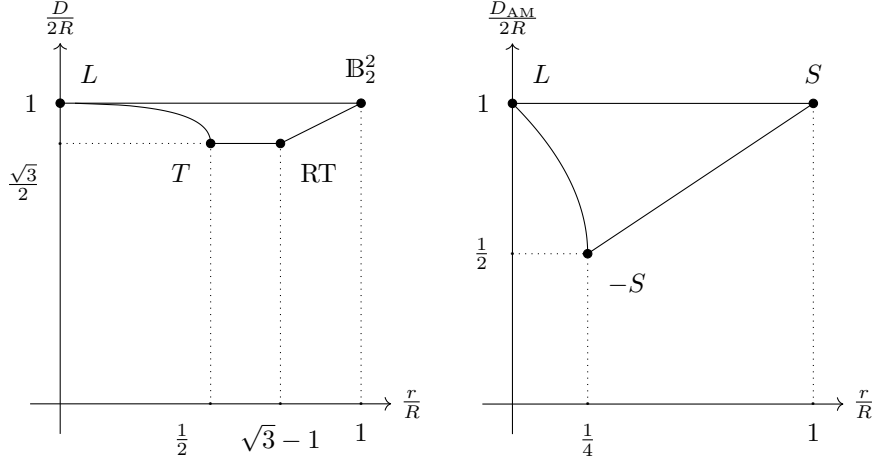

It is shown in \cite{ngons} that $f_{\AM}(\bar{\CC}^n,C)$ is star-shaped with respect to the vertex $f_{\AM}(C,C) = (1,1)$. This means that these diagrams can be fully described by characterizing the boundaries of the set. In the following, we prove similar (slightly weaker, but sufficient for our purposes) results for the other diameters. One may note that all diagrams with respect to triangles that are described in the following chapters are still star-shaped w.r.t.~$f_{\m}(C,C)$.


\begin{lemma}
    \label{simpleconn}
 The diagram $f_{\m}(\bar{\CC}^n,C)$ is closed and if there is a continous description of the outer boundary, it is simply connected.
    \end{lemma}
    \begin{proof}
      Assume there is a sequence $(K_n)_{n\in\N}\subset \CC^n$ such that $K_n \optc C$ for all $n\in\N$ and $r(K_n,C)\to r^{*}$ and $D_{\m}(K_n,C) \to D^{*}$ for $n \to \infty$. The sequence $(K_n)_{n\in\N}$ is bounded as all sets are contained in $C$. Thus, by the Blaschke-Selection-Theorem there exists a converging subsequence $K_{n_{k}}\to K^{*}$ for $k\to\infty$. The inradius and diameter are continuous and therefore $r(K^{*},C)=r^{*}$ and $D_{\m}(K^{*},C)=D^{*}$. Hence, $f_{\m}(\bar{\CC}^n,C)$ is closed. \\
      As a consequence, we know that $f_{\m}(\bar{\CC}^n,C)$ can only have open holes and therefore only fulldimensional holes.
    For $K\in \CC^n$ such that $K\optc C$, define $K_t:=(1-t)K+tC$ for $t\in[0,1]$. \\
    Then by \Cref{linearityrrRD},
    \begin{equation*}
        r(K_t,C)=(1-t)r(K,C)+t,
    \end{equation*}
    \begin{equation*}
        R(K_t,C)=(1-t)R(K,C)+t=1
    \end{equation*}
    and
    \begin{equation*}
        D_{\m}(K_t,C) \leq (1-t)D_{\m}(K,C)+tD_{\m}(C,C).
    \end{equation*}
    In the case $\m = \AM $, we also have equality for the diameter, but this does not necessarily hold for the other diameters.
    Since $R(\cdot,C),r(\cdot,C)$ and $D_{\m}(\cdot,C)$ are continuous with respect to the Hausdorff distance and $t\in [0,1] \mapsto (1-t)K+tC$ is continuous in $t$, the composition $\Gamma_K: [0,1]\to \R^2,\: t\mapsto \left(r(K_t,C),\frac{D_{\m}(K_t,C)}{2}\right)$ is continuous as well.
    Thus, for every such $K$ there is a continuous curve $\Gamma_K$ in the diagram from $f_{\m}(K,C)$ to $f_{\m}(C,C)$. \\
    Let $(K^n)_{n\in \N}$ be a sequence of bodies on the boundary converging to $K$ on the boundary. We show that the functions $\Gamma_{K^n}$ converge uniformly to $\Gamma_{K}$.
    We can consider the components separately.
    For the inradius, we know
    \begin{equation*}
        \begin{split}
            |r(K_t,C)-r(K^n_t,C)|&=|{(1-t)r(K,C)+t-(1-t)r(K^n,C) -t}|\\
            &=(1-t)|r(K,C)-r(K^n,C)|\\
            &\leq |r(K,C)-r(K^n,C)|.
        \end{split}
    \end{equation*}
    Let $\epsilon>0$. Since $|r(K,C)-r(K^n,C)|\to 0$ for $n \to \infty$, there exists an $N$ such that  for all $n\geq N$, $ |r(K_t,C)-r(K^n_t,C)|< \epsilon$ for all  $t\in[0,1]$.
    It is known that when convex, continuous functions $f_n :[a,b]\to \R$ converge pointwise to a convex and continuous function $f$, the convergence is uniform \cite[Lemma 21]{Uniform}.
    The functions $g_n: [0,1] \to \R$, $t \mapsto D_{\m}(K^n_t,C)$ are convex and continuous in $t$ and they converge pointwise to the convex and continuous function $g: [0,1] \to \R$, $t \mapsto D_{\m}(K_t,C)$. Thus, this convergence is also uniform and the curves converge uniformly.

    Now, assume the diagram has a hole. For $K$ on the boundary of the diagram we say that the curve $\Gamma_K$ lies \textit{above} the hole, if the set enclosed by $[f_{\m}(L_D,C),f_{\m}(C,C)]$, $\Gamma_K$ and the boundary between $f_{\m}(K,C)$ and $f_{\m}(L_D,C)$ which does not contain the segment $[f_{\m}(L_D,C),f_{\m}(C,C)]$ does not contain
    the hole. Analogously, we say that $\Gamma_K$ lies \textit{below} the hole if the set contains the hole. Thus, $\Gamma_{L_D}$ lies above the hole and $\Gamma_{C}$ below. Then, there exists a converging sequence $(K_n)_{n\in\N}$ with $K_n\to K$ of bodies on the boundary such that all $\Gamma_{K^n}$ lie above the hole and $\Gamma_{K}$ below or vice versa. This contradicts the fact that the curves converge uniformly.
    \begin{figure}[ht]
        \centering
         \begin{tikzpicture}[scale=5]

         \draw[->] (-0.1,0) -- (1.1,0) node[right] {$\frac{r}{R}$};
         \draw[->] (0,-0.1) -- (0,1.2) node[above] {$\frac{D_{\m}}{2R}$}
         ;

         \draw[ domain=0:1, smooth, variable=\x] plot ({\x}, {1})
         ;
         \draw[ domain=0.6:1, smooth, variable=\x] plot ({\x}, {1.25*\x-0.25})
         ;
         \draw[ domain=0.25:0.6, smooth, variable=\x] plot ({\x}, {0.5})
         ;
         \draw[ domain=0.5:0.75, smooth, variable=\y] plot ({(2*\y-0.5)*(1.5-2*\y)},{\y} ) 
         ;
         \node[label=below left:$K$] (K) at (77/324,5/9) {.};
            \draw[fill=black] (K) circle[radius=0.2pt];
            \node[label=below:$K_n$] (Kn) at (5/36,2/3) {.};
            \draw[fill=black] (Kn) circle[radius=0.2pt];
            \tkzDefPoint(15/49,17/28){O}
            \tkzDefPoint(0.35,17/28){A}
            \tkzDrawCircle[smooth, black](O,A);
            \draw[ domain=5/36:1, dashed, variable=\x] plot ({\x}, {((12*36)/(31*31))*(\x-1)*(\x+13/18)+1});
            \draw[ domain=77/324:1, dashed, variable=\x] plot ({\x}, {0.764739*(\x-1)*(\x+85/162)+1});

         \node[label=above:$C$] (C) at (1,1) {.};
         \draw[fill=black] (1,1) circle[radius=0.4pt];

         \node[label=left:$L_D$] (LD) at (0,1) {.};
         \draw[fill=black] (0,1) circle[radius=0.4pt];

         \node[label=left:$L_w$] (Lw) at (0,0.75) {.};
         \draw[fill=black] (0,0.75) circle[radius=0.2pt];

     \end{tikzpicture}
     \caption{Proof of \Cref{simpleconn}: $K_n\to K$ but $\Gamma_K$ lies below the hole and $\Gamma_{K_n}$ above.}

     \end{figure}
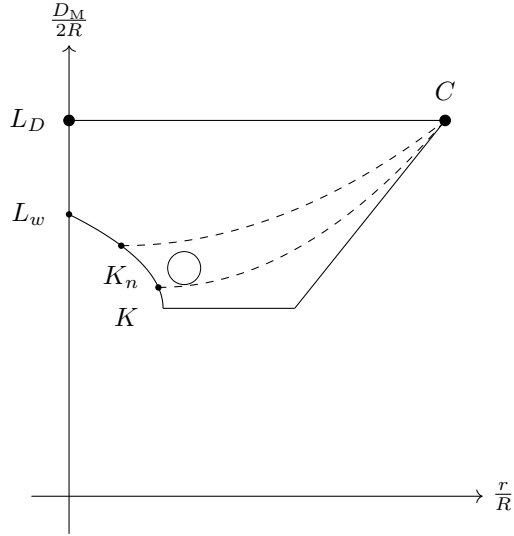
    \end{proof}

    In the standard diameter case it was sufficient to describe the Blaschke-Santaló diagram \wrt a triangle to obtain $f_{\AM}(\bar{\CC}^2,\CC^2_0)$. Thus,
    it seems reasonable to look at the diagrams for the three other diameters $D_{\MIN}$, $D_{\MAX}$ and $D_{\HM}$ in terms of triangular gauges first, which we do in the remaining sections.

\section{The diameter $D_{\MAX}$}
When we use the notions \enquote{equilateral}, \enquote{regular}, and \enquote{isosceles} in the following, then it is meant in the euclidean sense. Unless otherwise specified, we fix the \cemph{red}{equilateral triangle} to be $T:=\conv(\{p^1,p^2,p^3\})\subset \R^2$ with $p^1=(0,1)^T$, $p^2=(-\sqrt{3}/2,-1/2)^T$ and $p^3=(\sqrt{3}/2,-1/2)^T$. It is Minkowski-centered with $s(T)=2$. Moreover, in this case $T_{\MAX}$ is the regular hexagon $\conv(\{p^1,p^2,p^3,-p^1,-p^2,-p^3\})$.

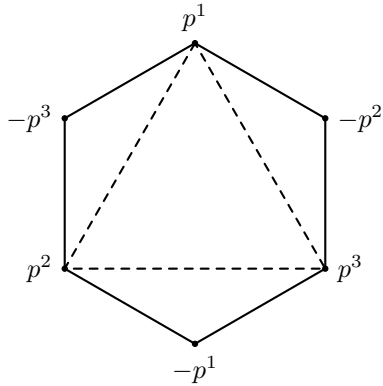
\begin{figure}[ht]
    \centering
\begin{tikzpicture}[scale=2]
    \tkzDefPoint(0,1){p1}
    \tkzDrawPoint[radius=0.4pt,label=above:$p^1$](p1)
    \tkzDefPoint(-0.86602,-0.5){p2}
    \tkzDrawPoint[radius=0.4pt,label=left:$p^2$](p2)
    \tkzDefPoint(0.86602,-0.5){p3}
    \tkzDrawPoint[radius=0.4pt,label=right:$p^3$](p3)
    \tkzDrawPolygon[black, thick,dashed](p1,p2,p3);
        \tkzDefPoint(0,-1){p1m}
     \tkzDrawPoint[radius=0.4pt,label=below:$-p^1$](p1m)
    \tkzDefPoint(0.86602,0.5){p2m}
    \tkzDrawPoint[radius=0.4pt,label=right:$-p^2$](p2m)
    \tkzDefPoint(-0.86602,0.5){p3m}
    \tkzDrawPoint[radius=0.4pt,label=left:$-p^3$](p3m)
     \tkzDrawPolygon[black, thick](p1,p2m,p3,p1m,p2,p3m);
\end{tikzpicture}
    \caption{The equilateral triangle $T$ and its maximum $T_{\MAX}$}
    \label{regtrianglemax}
  \end{figure}

In the case of the maximum both factors $\rho_{\MAX}$ and $\delta_{\MAX}$ are known \cite{reversing} and depend at most on $s(C)$:
\begin{proposition}
  \begin{equation*}
    C_{\AM}{\subset}^{\opt}C_{\MAX}\subset^{\opt}\frac{2s(C)}{s(C)+1}C_{\AM},
  \end{equation*}
  \ie~$\rho_{\MAX} = \frac{s(C)+1}{2s(C)}$ and $\delta_{\MAX} = 1$.
\end{proposition}
Taking $K=C_{\MAX}$, we have $R(C_{\MAX},C)=s(C)$, $r(C_{\MAX},C)=1$ and $D_{\MAX}(C_{\MAX},C)=2$ \cite{reversing}.

The inequalities from \Cref{lem:basicineq} have the form:

\begin{equation}
  \label{upMaxT}
  \frac{D_{\MAX}(K,C)}{2}\leq R(K,C),
\end{equation}

\begin{equation}
  \label{rightMaxT}
  r(K,C) \leq \frac{D_{\MAX}(K,C)}{2},
\end{equation}

\begin{equation}
  \label{rRDMAX}
  s(C)r(K,C)+R(K,C)\leq s(C)D_{\MAX}(K,C),
\end{equation}
  and
\begin{equation}
  \label{eqrRzero}
  0 \leq r(K,C).
\end{equation}
In case of $K=C$ \eqref{upMaxT} and \eqref{rightMaxT} become tight while \eqref{rightMaxT} and \eqref{rRDMAX} become tight for
  $K=C_{\MAX}$.


Asymmetric gauges are not complete, but a completion is easy to find.
\begin{definition}
  Let $A_C^{\out} := \bd(C\pol) \cap \bd(-C\pol)$.

  We define the \cemph{red}{outer symmetric support}:
    \begin{equation*}
    C^{\out}:=\bigcap_{a\in A_C^{\out}} H_{(a,1)}^{\leq}
    \end{equation*}
\end{definition}

\begin{lemma} \label{thm:max-completion}
    \begin{enumerate}[i)]
        \item $C_{\MAX}$ is always a completion of $C$ with maximal circumradius $R(C_{\MAX},C) = s(C)$,
        \item $C_{\MAX} \subset C^{\sup} \subset C^{\out}$,
        \item $C_{\MAX}= C^{\sup}$ if and only if $C_{\MAX}= C^{\out}$, and

        \item $C_{\MAX}$ is always the unique 0-symmetric completion of $C$.
    \end{enumerate}
  \end{lemma}

One should recognize that we need a scaling factor of up to $n$ to cover the completion $C_{\MAX}$ by $C$ here, while with the arithmetic diameter $C$ is always already complete itself.

\begin{proof}
  $C_{\MAX}$ is a completion of $C$ since $C \subset C_{\MAX}$ and $D_{\MAX}(C,C) =2\delta_{\MAX}= 2 = D_{\MAX}(C_{\MAX},C)$, while for all $K \supset C_{\MAX}$ we have $D_{\MAX}(K,C) = D_{\MAX}(K,C_{\MAX}) = 2R(K,C_{\MAX}) > 2$.

  The maximality of the circumradius can be seen as follows: Let $K$ be any completion of $C$. Then, from \eqref{rRDMAX} and \Cref{CCompletion} we obtain
  \begin{align*}
    R(K,C) & \le s(C)(D_{\MAX}(K,C) - r(K,C)) = s(C)(D_{\MAX}(C_{\MAX},C) - r(C_{\MAX},C)) \\
           &= s(C)=R(C_{\MAX},C).
  \end{align*}
  Next, we show that  $C_{\MAX} \subset C^{\sup} \subset C^{\out}$. The first containment follows from the fact that $C^{\sup}$ is the union of all completions of $C$. For the second containment it suffices to show that $h_{C^{\sup}}(a)\leq h_{C^{\out}}(a)$ for all $a\in A^{\out}_C$. Let $a\in A^{\out}_C$. Then, there exists $p\in -C\cap H_{(a,1)}$. Since $-p\in C$ we obtain $C^{\sup} \subset -p + 2 C_{\MAX}$ and therefore
  \begin{equation*}
    h_{C^{\sup}}(a)\leq h_{-p + 2 C_{\MAX}}(a)= -p^Ta + 2 h_{C_{\MAX}}(a)= 1 = h_{C^{\out}}(a).
  \end{equation*}

  The containment chain directly shows the backward direction of part $iii)$.

  To show the forward direction, assume $C_{\MAX}\neq C^{\out}$. Since $C_{\MAX}$ is the intersection of its regular slabs, there must exist some $a\in \R^n\setminus \set{0}$ which defines a regular slab of $C_{\MAX}$ but $h_C(a)>h_C(-a)$. Assuming $a\in \bd (C\pol)$, there exists a smooth boundary point $x$ of $C_{\MAX}$ supported by the hyperplane $H_{(a,1)}$. Because of $h_C(a)>h_C(-a)$, the point $x$ must belong to  $\bd(C) \cap H_{(a,1)}$.

    Now, assume $x\in \bd(C^{\sup})$ as well.
    Then, there exist an outer normal $a_x \in \bd (C_{\MAX}\pol)$ and a point $p_x \in C$ such that $x^Ta_x=h_{C^{\sup}}(a_x)$ and $(a_x)^T(x-p_x)=2$.
    Then, $H_{(a_x,h_{C^{\sup}}(a_x))}$ also supports $C_{\MAX}$, which implies $a_X=a$ since $x$ is a smooth boundary point of $C_{\MAX}$.
    But since $h_C(-a)<h_C(a)$ we cannot choose $p_x\in C$. Thus, $x$ is not contained in the boundary of $C^{\sup}$ and therefore, $C_{\MAX}\neq C^{\sup}$.

    Finally, any 0-symmetric completion of $C$ must contain $C$ and $-C$ and therefore $C_{\max}$.
\end{proof}

\begin{example}
    Trapezoids within the following family have completions besides their maximum:
    \begin{equation*}
        Z_{\lambda}:=\conv\left((\sqrt{3}/2,-1/2)^T,(-\sqrt{3}/2,-1/2)^T,(\lambda \sqrt{3}/2, 1-\lambda/2)^T, (-\lambda \sqrt{3}/2, 1-\lambda/2)^T\right)
    \end{equation*}
    with $\lambda\in(0,1)$. $Z_{\lambda}$ is Minkowski-centered with Minkowski asymmetry $2-\lambda$ and $(Z_{\lambda})^{\out}\neq {(Z_{\lambda})}_{\MAX}$, which because of \Cref{thm:max-completion} means that $(Z_{\lambda})_{\MAX}$ is not the unique completion of $Z_{\lambda}$. In the extreme cases $\lambda\in\set{0,1}$, $Z_{\lambda}$ is a triangle or a rectangle and $(Z_{\lambda})_{\MAX}=(Z_{\lambda})^{\sup}=(Z_{\lambda})^{\out}$.
    \begin{figure}[ht]
        \centering
    \begin{tikzpicture}[scale=2]
        \tkzDefPoint(0,1){p1}
        \tkzDefPoint(-0.86602,-0.5){p2}
        \tkzDefPoint(0.86602,-0.5){p3}
        \tkzDefPoint(0,-1){p1m}
        \tkzDefPoint(0.86602,0.5){p2m}
        \tkzDefPoint(-0.86602,0.5){p3m}

         \tkzDefPointOnLine[pos=0.4](p3m,p1)\tkzGetPoint{q3}
         \tkzDefPointOnLine[pos=0.4](p2m,p1)\tkzGetPoint{q2}

         \tkzDefPointOnLine[pos=0.4](p3,p1m)\tkzGetPoint{q3m}
         \tkzDefPointOnLine[pos=0.4](p2,p1m)\tkzGetPoint{q2m}
         \tkzDrawPolygon[fill=teal!30, opacity=0.7](p1,p3m,p2,q2m,q3m,p3,p2m);
         \tkzDrawPolygon[black, dashed](p2m,q3m,q2m,p3m);
         \tkzDrawPolygon[blue, thick](p1,p3m,p2,p1m,p3,p2m);
         \tkzDrawPolygon[black, very thick](p2,q3,q2,p3);
         \tkzDrawPoint[radius=0.4pt](p1)
         \tkzDrawPoint[radius=0.4pt](p2)
         \tkzDrawPoint[radius=0.4pt](p3)
          \tkzDrawPoint[radius=0.4pt](p3m)
         \tkzDrawPoint[radius=0.4pt](p2m)
         \tkzDrawPoint[radius=0.4pt](p1m)
         \tkzDrawPoint[radius=0.4pt](q2)
         \tkzDrawPoint[radius=0.4pt](q3)
         \tkzDrawPoint[radius=0.4pt](q2m)
         \tkzDrawPoint[radius=0.4pt](q3m)
         \tkzDefPoint(0,0){O}

         \tkzDrawPoint[radius=0.4pt](O)
         \tkzLabelSegment[right,blue](p3,p2m){$(Z_{\lambda})^{\out}$}
         \tkzLabelSegment[above,teal](p2,p3){$(Z_{\lambda})^{\sup}$}
         \tkzLabelSegment[right,black](p2,q3){$Z_{\lambda}$}

    \end{tikzpicture}
        \caption{For trapezoids $Z_{\lambda}$ the maximum $(Z_{\lambda})_{\MAX}$ is not their unique completion since $(Z_{\lambda})_{\MAX}\neq (Z_{\lambda})^{\out}$.}
        \label{fig:ex:notuniquecomp}
    \end{figure}
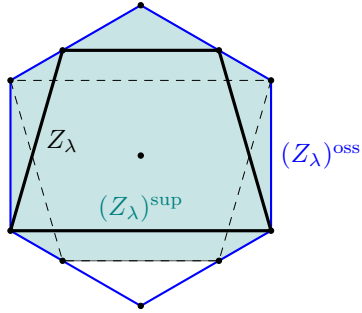
\end{example}


The first new inequality we provide is a lower bound for the diameter-circumradius ratio, a so called Jung-type inequality, which stays true independently of the gauge $C$.

\begin{theorem}
    \label{thm:jungbound_Dmax_mink}
    Let $K,C\in\CC^2$, \st~$C$ is Minkowski-centered. Then
    \begin{equation*}
        D_{\MAX}(K,C)\geq R(K,C)
    \end{equation*}
\end{theorem}

\begin{proof}
  If $K$ is a single point, $R(K,C)=D_{\MAX}(K,C)=0$. Thus, we can assume $K\in \bar{\CC}^2$ and  $K\subset^{\opt}C$, which implies $R(K,C)=1$. Then, there exist touching points $q^1,\dots,q^k\in \bd(K) \cap \bd(C)$ with $k\in \set{2,3}$ and corresponding outer normals $a^i$ as described in \Cref{opt}. If $k=2$ is possible, there exists a segment $L\subset K$ with the same circumradius as $K$ and by \Cref{CDiam} $D_{\MAX}(K,C)\geq D_{\MAX}(L,C)\geq 2\rho_{\MAX}(C)=\frac{s(C)+1}{s(C)} > 1$. For $k=3$, the triangle $\conv(\set{q^1,q^2,q^3})$ has the same circumradius as $K$ and $D_{\MAX}(K,C)\geq D_{\MAX}(\conv(\set{q^1,q^2,q^3}),C)$. Thus, it suffices to prove the claim for the case that $K$ is a proper triangle $K=\conv(\set{q^1,q^2,q^3})$.

  Let $S:=\bigcap_{i=1}^3 H_{(a^i,1)}^{\leq}$ be the intersection of the three supporting halfspaces of $C$ \st~$q^i\in H_{(a^i,1)}$. Denote the vertex opposing the edge defined by $a^i$ by $\tilde{p}^i$. Then, $R(K,S)=R(K,C)=1$ and $D_{\MAX}(K,C)\geq D_{\MAX}(K,S)$. By invariance under linear transformations we can assume that $S=c+T$ where $T$ is the Minkowski-centered equilateral triangle as described before: $\tilde{p}^i=c+p^i$, $i=1,2,3$. In the following indices are to be understood modulo 3. Let $\alpha_i\in[0,1]$ be, \st~$q^i=\alpha_i \tilde{p}^{i+1} + (1-\alpha) \tilde{p}^{i+2}$. Since $C$ is Minkowski-centered, $0$ lies in the interior of $C$.

  We split the proof into two parts. First, we consider the case where the origin is close to the center $c$, \ie~$0\in \inte (\conv(\set{c-\frac{1}{2}p^i,\quad i=1,2,3}))$. Afterwards, we care about the case, where $c$ is further apart from the origin.

  Let us start with the case where $0\in \inte (\conv(\set{c-\frac{1}{2}p^i,\quad i=1,2,3}))$, which is equivalent to $c\in \inte (\conv(\set{\frac{1}{2}p^i,\quad i=1,2,3 }))$. Define $\lambda_1,\lambda_2,\lambda_3 > 0$ with $\sum_{i=1}^3 \lambda_i = \frac{1}{2}$ such that $c=\sum_{i=1}^3 \lambda_i p^i$. Let $z^i \in\R^2$ be the direction such that
  \begin{equation*}
    (z^i)^T \tilde{p}^{i+1}=-1 \quad \text{and} \quad
    (z^i)^T \tilde{p}^{i+2}=1.
  \end{equation*}
  This is possible since $0\in\inte S$.
  Since $c=\frac{1}{3}\sum_{i=1}^3 \tilde{p}^i$ we have
  \begin{equation*}
    (z^i)^T \tilde{p}^{i}= 3(z^i)^Tc.
  \end{equation*}
  Inserting $c=\sum_{i=1}^3 \lambda_i p^i$ yields
  \begin{align*}
    (z^i)^Tc &= (z^i)^T \sum_{i=1}^3 \lambda_i (\tilde{p}^i -c)= -\frac{1}{2}  (z^i)^Tc + 3\lambda_i  (z^i)^Tc -\lambda_{i+1}+\lambda_{i+2}
  \end{align*}
  which implies
  \begin{align*}
    3 (z^i)^Tc = \frac{\lambda_{i+2}-\lambda_{i+1}}{\frac{1}{2}-\lambda_i}= \frac{\lambda_{i+2}-\lambda_{i+1}}{\lambda_{i+2}+\lambda_{i+1}}.
  \end{align*}
  Thus,
  \begin{equation*}
    1+3(z^i)^Tc = \frac{2\lambda_{i+2}}{\lambda_{i+2}+\lambda_{i+1}}\geq 0 \quad \text{and} \quad
    1-3(z^i)^Tc = \frac{2\lambda_{i+1}}{\lambda_{i+2}+\lambda_{i+1}} \geq 0,
  \end{equation*}
It follows that
  \begin{equation*}
    +/- (z^i)^T \tilde{p}^i \leq 1,
  \end{equation*}
   which shows that $z^i$ is an outer normal of $\conv(S\cup (-S))$ with $h_{\conv(S\cup (-S))}(z^i)=1$.

  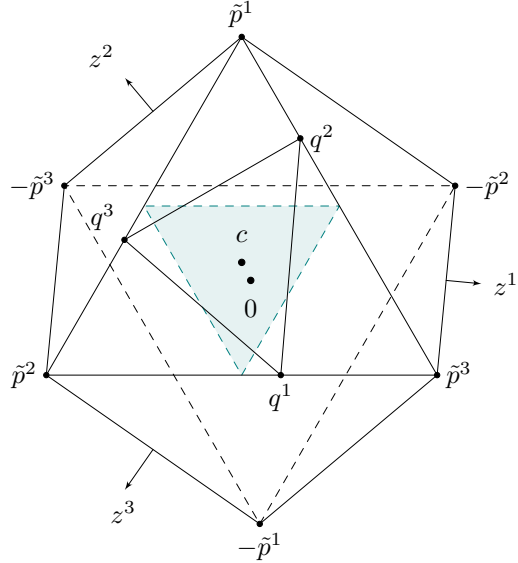
\begin{figure}[ht]
    \centering
    \begin{tikzpicture}[scale=3,vect/.style={->,
        shorten >=1pt,>=latex'}]
      \tkzDefPoint(0,1){p1}
      \tkzDefPoint(-0.86602,-0.5){p2}
      \tkzDefPoint(0.86602,-0.5){p3}
      \tkzDrawPoints[fill=black](p1,p2,p3)
      \tkzDefMidPoint(p1,p2)
      \tkzGetPoint{m3}
      \tkzDefMidPoint(p2,p3)
      \tkzGetPoint{m1}
      \tkzDefMidPoint(p3,p1)
      \tkzGetPoint{m2}
      \tkzDrawPolygon[fill=teal!10, draw=teal, dashed](m1,m2,m3)

      \node[label=below:$0$] (O) at (0.04,-0.08) {.};
      \draw[fill=black] (O) circle[radius=0.4pt];
      \node[label=above:$c$] (c) at (0,0) {.};
      \draw[fill=black] (c) circle[radius=0.4pt];

      \tkzDefPointBy[symmetry=center O](p1)
      \tkzGetPoint{p1m}
      \tkzDefPointBy[symmetry=center O](p2)
      \tkzGetPoint{p2m}
      \tkzDefPointBy[symmetry=center O](p3)
      \tkzGetPoint{p3m}
      \tkzDrawPoints[fill=black](p1m,p2m,p3m)
      \tkzDrawPolygon(p1,p2,p3)
      \tkzDrawPolygon[dashed](p1m,p2m,p3m)
      \tkzDrawPolygon(p1,p3m,p2,p1m,p3,p2m)
      \tkzLabelPoint[above](p1){$\tilde{p}^1$}
      \tkzLabelPoint[left](p2){$\tilde{p}^2$}
      \tkzLabelPoint[right](p3){$\tilde{p}^3$}
      \tkzLabelPoint[below](p1m){$-\tilde{p}^1$}
      \tkzLabelPoint[right](p2m){$-\tilde{p}^2$}
      \tkzLabelPoint[left](p3m){$-\tilde{p}^3$}

      \tkzDefPointOnLine[pos=.6](p2,p3)
      \tkzGetPoint{q1}
      \tkzLabelPoint[below](q1){$q^1$}
      \tkzDefPointOnLine[pos=.7](p3,p1)
      \tkzGetPoint{q2}
      \tkzLabelPoint[right](q2){$q^2$}
      \tkzDefPointOnLine[pos=.6](p1,p2)
      \tkzGetPoint{q3}
      \tkzLabelPoint[above left](q3){$q^3$}
      \tkzDrawPoints[fill=black](q1,q2,q3)
      \tkzDrawPolygon(q1,q2,q3)

      \tkzDefMidPoint(p1,p3m)
      \tkzGetPoint{b2b}
      \tkzDefMidPoint(p2,p1m)
      \tkzGetPoint{b3b}
      \tkzDefMidPoint(p3,p2m)
      \tkzGetPoint{b1b}
      \tkzDefPointWith[orthogonal,K=0.4](b3b,p2)
      \tkzGetPoint{b3}
      \tkzDefPointWith[orthogonal,K=0.4](b2b,p1)
      \tkzGetPoint{b2}
      \tkzDefPointWith[orthogonal,K=0.4](b1b,p3)
      \tkzGetPoint{b1}
      \tkzDrawSegments[vect](b3b,b3 b2b,b2 b1b,b1)
      \tkzLabelPoint[right](b1){$z^1$}
      \tkzLabelPoint[above left](b2){$z^2$}
      \tkzLabelPoint[below](b3){$z^3$}
    \end{tikzpicture}
    \caption{Proof of \Cref{thm:jungbound_Dmax_mink}. If $0\in \inte (\conv(\set{c-\frac{1}{2}p^i,\quad i=1,2,3}))$, then $\max_{i=1,2,3} b_{z^i}(K,S\cup (-S))\geq 1 $.}
    \label{proofjungTmax1}
  \end{figure}

  We know
  \begin{equation}
    \label{eq:breadthsjung}
    \begin{split}
      D_{\MAX}(K,S)&\geq \max_{i=1,2,3} b_{z^i}(K,\conv(S\cup (-S)))\\
                   &\geq \max_{i=1,2,3}\frac{(z^i)^Tq^{i+1}-(z^i)^Tq^{i+2}}{h_{\conv(S\cup (-S))}(z^i)}\\
                   &= \max_{i=1,2,3}(z^i)^T(\alpha_{i+1}\tilde{p}^{i+2}+(1-\alpha_{i+1})\tilde{p}^{i})-(z^i)^T(\alpha_{i+2}\tilde{p}^{i}+(1-\alpha_{i+2})\tilde{p}^{i+1})\\
                   &=\max_{i=1,2,3} \alpha_{i+1}(1-3(z^i)^Tc) +(1-\alpha_{i+2}) (1+3(z^i)^Tc)\\
                   &=\max_{i=1,2,3}  \frac{2\alpha_{i+1}\lambda_{i+1}+ 2(1-\alpha_{i+2})\lambda_{i+2}}{\lambda_{i+2}+\lambda_{i+1}}
    \end{split}
  \end{equation}
  Next, we show that the last term is larger than or equal to 1. If $b_{z^i}(K,T_{\MAX})\geq 1$ for $i\in\set{1,2}$ there is nothing to show. Thus let us
  assume that those breadths are smaller than 1.
  We show that the latter implies $b_{z^3}(K,T_{\MAX})\geq 1$.
  \begin{align*}
    b_{z^1}(K,S_{\MAX})<1 \text{\quad implies \quad} 2\alpha_2\lambda_2 +2(1-\alpha_3)\lambda_3 &< \lambda_2+\lambda_3 \text{\quad and}\\
    b_{z^2}(K,S_{\MAX})<1 \text{\quad implies \quad}
    2\alpha_3\lambda_3 +2(1-\alpha_1)\lambda_1 &< \lambda_3+\lambda_1.
  \end{align*}
  By adding these inequalities we obtain
  \begin{align*}
    2 \lambda_3 + 2\alpha_2\lambda_2 +2(1-\alpha_1)\lambda_1 &< \lambda_2+2\lambda_3 + \lambda_1
                                                               \intertext{or equivalently}
                                                               2\alpha_1\lambda_1 +2(1-\alpha_2)\lambda_2 &> \lambda_1+\lambda_2,
  \end{align*}
  which proves $b_{z^3}(K,T_{\MAX})\geq 1$.
  Thus, $D_{\MAX}(K,C)\geq D_{\MAX}(K,S)\geq 1$.

  Now, consider the case that $0\notin \inte (\conv(\set{c-\frac{1}{2}p^i,\quad i=1,2,3}))$. Because of the symmetries of $T$ we can assume that  $0=\sum_{i=1}^3 \beta_i \tilde{p^i}$  with $\sum_{i=1}^3 \beta_i=1$, $\beta_3\geq \frac{1}{2}$ and $\beta_2\geq \beta_1> 0$ (\cf~\Cref{proofjungTmax2}).

  \begin{figure}[ht]
    \centering
    \begin{tikzpicture}[scale=3.5,vect/.style={->,
        shorten >=1pt,>=latex'}]
      \tkzDefPoint(0,1){p1}
      \tkzDefPoint(-0.86602,-0.5){p2}
      \tkzDefPoint(0.86602,-0.5){p3}
      \tkzDrawPoints[fill=black](p1,p2,p3)

      \tkzDefMidPoint(p2,p1)
      \tkzGetPoint{a3b}
      \tkzDefPointWith[orthogonal,K=0.35](a3b,p1)
      \tkzGetPoint{a3}
      \tkzDrawSegment[vect](a3b,a3)
      \tkzLabelPoint[above left](a3){$a^3$}
      \tkzDefMidPoint(p2,p3)
      \tkzGetPoint{a1b}
      \tkzDefPointWith[orthogonal,K=0.35](a1b,p2)
      \tkzGetPoint{a1}
      \tkzDrawSegment[vect](a1b,a1)
      \tkzLabelPoint[below](a1){$a^1$}
      \tkzDefMidPoint(p3,p1)
      \tkzGetPoint{a2b}
      \tkzDefPointWith[orthogonal,K=0.35](a2b,p3)
      \tkzGetPoint{a2}
      \tkzDrawSegment[vect](a2b,a2)
      \tkzLabelPoint[above right](a2){$a^2$}
      \tkzDefPointOnLine[pos=.5](p3,a3b)
      \tkzGetPoint{m3}
      \tkzDefPointOnLine[pos=.5](p3,p2)
      \tkzGetPoint{m1}
      \tkzDrawPolygon[fill=teal!10, draw=teal](m1,m3,p3)

      \def\mu{5/12}
      \tkzDefPointOnLine[pos=\mu](p3,p1)
      \tkzGetPoint{y1}
      \tkzDefPointOnLine[pos=\mu](p3,p2)
      \tkzGetPoint{y2}
      \tkzDefPointOnLine[pos=0.65](y1,y2)
      \tkzGetPoint{O}
      \tkzDrawSegment[dashed, blue](y1,y2)
      \tkzDrawPoint[fill=black](O)
      \tkzLabelPoint[below right](O){$0$}
      \node[label=above:$c$] (c) at (0,0) {.};
      \draw[fill=black] (c) circle[radius=0.4pt];
      \draw[black, dashed] (0,-0.5) -- (0.43301,0.25) -- (-0.43301,0.25) -- cycle;

      \tkzDrawPolygon(p1,p2,p3)
      \tkzLabelPoint[above](p1){$\tilde{p}^1$}
      \tkzLabelPoint[left](p2){$\tilde{p}^2$}
      \tkzLabelPoint[right](p3){$\tilde{p}^3$}

      \tkzDefPointOnLine[pos=.56](p2,p3)
      \tkzGetPoint{q1}
      \tkzLabelPoint[below](q1){$q^1$}
      \tkzDefPointOnLine[pos=.44](p3,p1)
      \tkzGetPoint{q2}
      \tkzLabelPoint[right](q2){$q^2$}
      \tkzDefPointOnLine[pos=.43](p1,p2)
      \tkzGetPoint{q3}
      \tkzLabelPoint[above left](q3){$q^3$}
      \tkzDrawPoints[fill=black](q1,q2,q3)
      \tkzDrawPolygon(q1,q2,q3)
    \end{tikzpicture}
    \caption{Proof of \Cref{thm:jungbound_Dmax_mink}.
      If $0\notin \inte (\conv(\set{c-\frac{1}{2}p^i,\quad i=1,2,3}))$ we may assume, \Wlog, that the origin is within the colored area. Then, $\max_{i=1,3} b_{a^i}(K,\conv(C\cup (-C)))\geq 1 $.}
    \label{proofjungTmax2}
  \end{figure}
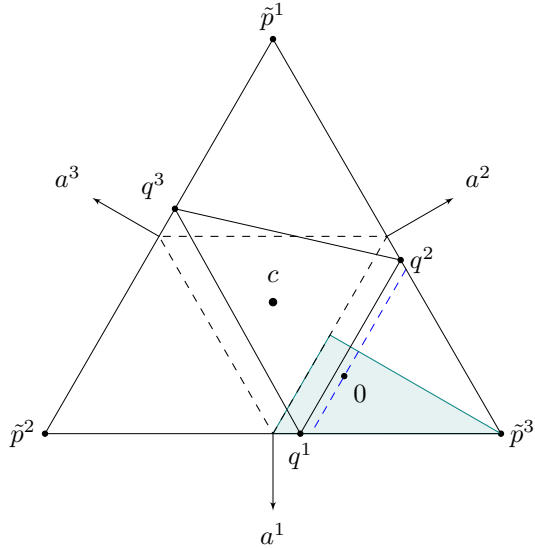

  By the above conditions we have
  \begin{equation*}
    h_C(-a^3)\leq h_S(-a^3)=(-a^3)^T\tilde{p}^3= \frac{\beta_1+\beta_2}{\beta_3}\leq 1=h_C(a^3).
  \end{equation*}
  Thus, using $\alpha_1,\alpha_2,\alpha_3$ as above, we know, if $\alpha_1< 1-\beta_3$ or $\alpha_2>\beta_3$, then
  \begin{align*}
    h_K(-a^3)&\geq\max\set{(-a^3)^Tq^1,(-a^3)^Tq^2}\\
             &=\max\set{(-a^3)^T(\alpha_1\tilde{p}^2+(1-\alpha_1)\tilde{p}^3), (-a^3)^T(\alpha_2\tilde{p}^3+(1-\alpha_2)\tilde{p}^1)}\\
             &=\max\set{-\alpha_1+(1-\alpha_1)\frac{1-\beta_3}{\beta_3}, \alpha_2\frac{1-\beta_3}{\beta_3}-(1-\alpha_2)}>0
  \end{align*}
  and therefore,
  \begin{equation*}
    D_{\MAX}(K,C) \geq b_{a^3}(K,\conv(C\cup(-C))) \geq \frac{h_K(a^3)+h_K(-a^3)}{h_C(a^3)} > 1.
  \end{equation*}

  Now, let $\alpha_1\geq 1-\beta_3$ and $\alpha_2\leq \beta_3$. Since $\beta_1\leq \beta_2$ we have $\beta_1\leq \frac{1-\beta_3}{2}$.
  Thus,
  \begin{align*}
    (-{a^1})^T q^2&= \alpha_2((-{a^1})^T\tilde{p}^3) +(1-\alpha_2)((-a^1)^T \tilde{p}^1)
                    = -\alpha_2 +(1-\alpha_2)\left(\frac{\beta_2}{\beta_1}+\frac{\beta_3}{\beta_1}\right)\\
                  &\geq -\beta_3 +(1-\beta_3)(1+\frac{2\beta_3}{1-\beta_3})
                    =1.
  \end{align*}

 Using that $C$ is Minkowski-centered, we know $h_C(-a^1)\leq s(C)h_C(a^1)=s(C)$ and we obtain
  \begin{align*}
    D_{\MAX}(K,C)&\geq b_{a^1}(K,\conv(C\cup(-C)))
                   \geq \frac{(a^1)^T q^1 -(a^1)^T q^2 }{\max(h_C(a^1),h_C(-a^1))}\\
                 &\geq \frac{1+1}{s(C)}=\frac{2}{s(C)}\geq 1.
  \end{align*}
\end{proof}

Since $D_{\MAX}$ is the smallest diameter, this provides a lower bound for all four diameter definitions.

\begin{corollary} \label{cor:jung-MC-all}
  Let $K,C\in\CC^2$, $C$ Minkowski-centered. Then
  \begin{equation*}
    D_{\m}(K,C)\geq R(K,C).
  \end{equation*}
\end{corollary}

If $\m\neq \MAX$ it follows directly from \Cref{prop:firey-chain} and the translation invariance in the arithmetic case that \Cref{cor:jung-MC-all} stays true for non-Minkowski-centered $C\in\CC^2$.
However, since $D_{\AM}(K,C)= R(K,C)$ is obtained if and only if $C$ is a triangle and $K$ is a homothet of $-C$ \cite{sharpening,CompleteSimpl}, we will see in the next section that this bound cannot be reached if $\m \in \set{\MIN,\HM}$.

If we omit the restriction of $C$ being Minkowski-centered, \Cref{thm:jungbound_Dmax_mink} is not necessarily true. However, for gauges that still contain the origin, we obtain the following Jung-type inequality.

\begin{theorem}
  \label{thm:jungbound_Dmax}
  Let $K,C\in\CC^2$, \st~$0\in C$ and $\dim(C)=2$. Then,
  \begin{equation*}
    D_{\MAX}(K,C)\geq \frac{2}{3}R(K,C).
  \end{equation*}
  Moreover, equality can be attained for some $K$ if and only if $C$ is a triangle with one vertex at the origin.
\end{theorem}

\begin{proof}
  We use the same notation as in the previous proof. As in \Cref{thm:jungbound_Dmax_mink}, we can assume that $K$ is a triangle and $K\optc C$. It suffices to show $D(K,\conv(S\cup(-S)))\geq \frac{2}{3}$ for $S$ as given in \Cref{thm:jungbound_Dmax_mink} since $D_{\MAX}(K,C)= D(K,\conv(C\cup(-C)))\geq D(K,\conv(S\cup(-S)))$.

  If there exists $\alpha_i \notin \left[\frac{1}{3},\frac{2}{3}\right]$, there is an $ a^j$ with $\frac{{a^j}^T q^j- {a^j}^T q^i}{h_C(a^j)+h_C(-a^j)}
  > \frac{2}{3}$ (\cf~\Cref{proofjungTmax3}) and therefore,
  \begin{align*}
    D(K,\conv(S\cup(-S)))&\geq b_{a^j}(K,\conv(S\cup(-S)))
                         \geq \frac{{a^j}^T q^j-{a^j}^T q^i }{\max(h_C(a^j),h_C(-a^j))}\\
                         &\geq \frac{{a^j}^T q^j- {a^j}^T q^i}{h_C(a^j)+h_C(-a^j)}
                         > \frac{2}{3}.
  \end{align*}

  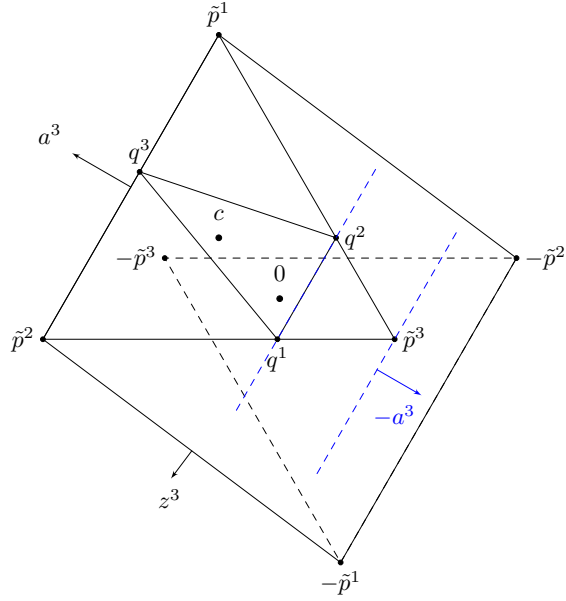
\begin{figure}[ht]
    \centering
    \scalebox{0.9}{
    \begin{tikzpicture}[scale=3,vect/.style={->,
        shorten >=1pt,>=latex'}]
      \tkzDefPoint(0,1){p1}
      \tkzDefPoint(-0.86602,-0.5){p2}
      \tkzDefPoint(0.86602,-0.5){p3}
      \tkzDrawPoints[fill=black](p1,p2,p3)

      \node[label=above:$0$] (O) at (0.3,-0.3) {.};
      \draw[fill=black] (O) circle[radius=0.4pt];
      \node[label=above:$c$] (c) at (0,0) {.};
      \draw[fill=black] (c) circle[radius=0.4pt];
      \tkzDefPointBy[symmetry=center O](p1)
      \tkzGetPoint{p1m}
      \tkzDefPointBy[symmetry=center O](p2)
      \tkzGetPoint{p2m}
      \tkzDefPointBy[symmetry=center O](p3)
      \tkzGetPoint{p3m}
      \tkzDrawPoints[fill=black](p1m,p2m,p3m)
      \tkzDrawPolygon(p1,p2,p3)
      \tkzDrawPolygon[dashed](p1m,p2m,p3m)
      \tkzDrawPolygon(p1,p2,p1m,p2m)
      \tkzLabelPoint[above](p1){$\tilde{p}^1$}
      \tkzLabelPoint[left](p2){$\tilde{p}^2$}
      \tkzLabelPoint[right](p3){$\tilde{p}^3$}
      \tkzLabelPoint[below](p1m){$-\tilde{p}^1$}
      \tkzLabelPoint[right](p2m){$-\tilde{p}^2$}
      \tkzLabelPoint[left](p3m){$-\tilde{p}^3$}

      \tkzDefPointOnLine[pos=2/3](p2,p3)
      \tkzGetPoint{q1}
      \tkzLabelPoint[below](q1){$q^1$}
      \tkzDefPointOnLine[pos=1/3](p3,p1)
      \tkzGetPoint{q2}
      \tkzLabelPoint[right](q2){$q^2$}
      \tkzDefPointOnLine[pos=.45](p1,p2)
      \tkzGetPoint{q3}
      \tkzLabelPoint[above ](q3){$q^3$}
      \tkzDrawPoints[fill=black](q1,q2,q3)
      \tkzDrawPolygon(q1,q2,q3)

      \tkzDefMidPoint(p2,p1m)
      \tkzGetPoint{b3b}

      \tkzDefPointWith[orthogonal,K=0.2](b3b,p2)
      \tkzGetPoint{b3}

      \tkzDrawSegment[vect](b3b,b3)
      \tkzLabelPoint[below](b3){$z^3$}
      \tkzDefMidPoint(p1,p2)
      \tkzGetPoint{a3b}

      \tkzDefPointWith[orthogonal,K=0.4](a3b,p1)
      \tkzGetPoint{a3}

      \tkzDrawSegment[vect](a3b,a3)
      \tkzLabelPoint[above left](a3){$a^3$}

      \tkzDrawLine[dashed,blue, add=0.7 and 0.7](q1,q2)
      \tkzDefLine[parallel=through p3](p1,p2) \tkzGetPoint{c'}
      \tkzDrawLine[dashed,blue, add=0.35 and -0.55](p3,c')
      \tkzDefPointOnLine[pos=0.1](p3,c')
      \tkzGetPoint{a3bm}

      \tkzDefPointWith[orthogonal,K=0.17](a3bm,c')
      \tkzGetPoint{a3m}

      \tkzDrawSegment[vect,blue](a3bm,a3m)
      \tkzLabelPoint[below left, blue](a3m){$-a^3$}

    \end{tikzpicture}}
    \caption{Proof of \Cref{thm:jungbound_Dmax}. If $\alpha_1<\frac{1}{3}$ or $\alpha_2>\frac{2}{3}$, we have $b_{a^3}(K,\conv(S\cup(-S)))>\frac{2}{3}$.}
    \label{proofjungTmax3}
  \end{figure}

  Now, consider the case where $\alpha_i \in \left[\frac{1}{3},\frac{2}{3}\right]$ for all $i\in\set{1,2,3}$.
  There exists $i\in\set{1,2,3}$ such that $z^i$ defined as in the previous proof defines a supporting hyperplane of  $\conv\left(S\cup (-S)\right)$ with $h_{\conv(S\cup(-S))}(z^i)=1$, \Wlog $i=3$. Then, we know $1=h_{\conv(S\cup(-S))}(z^3) \geq (z^3)^T (\pm p^3)= \pm 3(z^3)^Tc$ and using \eqref{eq:breadthsjung} we obtain
  \begin{equation}
    \label{eq:notminkchain}
    \begin{split}
      D(K,\conv(S\cup(-S)))&\geq b_{z^3}(K,\conv(S\cup(-S)))
                           \geq \alpha_1(1-3(z^3)^Tc) +(1-\alpha_2)(1+3(z^3)^Tc)\\
                           &\geq \frac{1}{3}(1-3(z^3)^Tc)+\frac{1}{3} (1+3(z^3)^Tc)
                           =\frac{2}{3}
    \end{split}
  \end{equation}

  Reaching equality in the last inequality chain we need $\alpha_1=(1-\alpha_2)=\frac{1}{3}$. In this case we need $0\in \set{\tilde{p}^3}\cup [\tilde{p}^1,\tilde{p}^2]$ for $b_{a^3}$ not to be larger than $\frac{2}{3}$. However, if $0 \in[\tilde{p}^1,\tilde{p}^2]$ we have that $z^3$ cannot define a supporting hyperplane as described above. Thus, $0$ remains to be a vertex of $S$ to reach equality.
  For $D_{\MAX}(K,C)=\frac{2}{3}$ to be true, we need
  \begin{equation*}
    b_{z^3}(K,\conv(C\cup(-C)))=  b_{z^3}(K,\conv(S\cup(-S))) = \frac{2}{3},
  \end{equation*}
  which is only possible if $\tilde{p}^1\in C$ or
  $\tilde{p}^2\in C$. If only one of these vertices is contained in $C$, $\frac{3}{2}(q^2-q^1)\notin \conv(C\cup(-C))$ which implies $D_{\MAX}(K,C)>\frac{2}{3}$. Thus, $\tilde{p}^1, \tilde{p}^2\in C$, which means  $C=S$.

  All in all, we see that equality can be obtained for $C=\conv(\set{0,\tilde{p}^1,\tilde{p}^2})$ and $\conv(\set{q^1,q^2,q^3})$ with $\alpha_1=(1-\alpha_2)=\frac{1}{3}$ and $\alpha_3\in\left[\frac{1}{3},\frac{2}{3}\right]$. In that case
  \[D_{\MAX}(K,C)=\max\set{b_{a^3}(K,\conv(C\cup(-C))),b_{z^3}(K,\conv(C\cup(-C)))}=\frac{2}{3}\]
  (see \Cref{proofjungTmax4}), which is achieved, e.g.~if we choose $K = \conv\left(\set{\frac{1}{3}\tilde{p}^1,\frac{1}{3}\tilde{p}^2, \frac{1}{2}(\tilde{p}^1+\tilde{p}^2)}\right)$.

  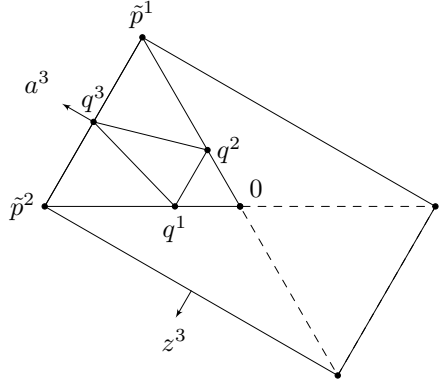
\begin{figure}[ht]
    \centering
    \begin{tikzpicture}[scale=1.5,vect/.style={->,
        shorten >=1pt,>=latex'}]
      \tkzDefPoint(0,1){p1}
      \tkzDefPoint(-0.86602,-0.5){p2}
      \tkzDefPoint(0.86602,-0.5){p3}
      \tkzDrawPoints[fill=black](p1,p2,p3)

      \node (O) at (0.86602,-0.5) {.};

      \tkzDefPointBy[symmetry=center O](p1)
      \tkzGetPoint{p1m}
      \tkzDefPointBy[symmetry=center O](p2)
      \tkzGetPoint{p2m}
      \tkzDefPointBy[symmetry=center O](p3)
      \tkzGetPoint{p3m}
      \tkzDrawPoints[fill=black](p1m,p2m,p3m)
      \tkzDrawPolygon(p1,p2,p3)
      \tkzDrawPolygon[dashed](p1m,p2m,p3m)
      \tkzDrawPolygon(p1,p2,p1m,p2m)
      \tkzLabelPoint[above](p1){$\tilde{p}^1$}
      \tkzLabelPoint[left](p2){$\tilde{p}^2$}
      \tkzLabelPoint[above right](p3){$0$}

      \tkzDefPointOnLine[pos=2/3](p2,p3)
      \tkzGetPoint{q1}
      \tkzLabelPoint[below](q1){$q^1$}
      \tkzDefPointOnLine[pos=1/3](p3,p1)
      \tkzGetPoint{q2}
      \tkzLabelPoint[right](q2){$q^2$}
      \tkzDefPointOnLine[pos=.5](p1,p2)
      \tkzGetPoint{q3}
      \tkzLabelPoint[above](q3){$q^3$}
      \tkzDrawPoints[fill=black](q1,q2,q3)
      \tkzDrawPolygon(q1,q2,q3)

      \tkzDefMidPoint(p2,p1m)
      \tkzGetPoint{b3b}

      \tkzDefPointWith[orthogonal,K=0.2](b3b,p2)
      \tkzGetPoint{b3}

      \tkzDrawSegment[vect](b3b,b3)
      \tkzLabelPoint[below](b3){$z^3$}
      \tkzDefPointWith[orthogonal,K=0.4](q3,p1)
      \tkzGetPoint{a3}

      \tkzDrawSegment[vect](q3,a3)
      \tkzLabelPoint[above left](a3){$a^3$}

    \end{tikzpicture}
    \caption{Equality case in \Cref{thm:jungbound_Dmax}: $C=\conv(\set{0,\tilde{p}^1,\tilde{p}^2})$ and $K=\conv\left(\set{\frac{1}{3}\tilde{p}^1,\frac{1}{3}\tilde{p}^2, \frac{1}{2}(\tilde{p}^1+\tilde{p}^2)}\right)$.
    }
    \label{proofjungTmax4}
  \end{figure}
\end{proof}
For fixed $C$, we refer to any $\tilde{K}$ that fulfills $\frac{D_{\m}(\tilde{K},C)}{R(\tilde{K},C)}=\min_{K\in\CC^n}\frac{D_{\m}(K,C)}{R(K,C)}$ as \cemph{red}{Jung-extremal}. If $C=T$, we have equality in \Cref{thm:jungbound_Dmax_mink} for the following family of triangles.

\begin{example}
  \label{example:jungboundmax}
  Let $T$ be the equilateral triangle as defined above and
  \begin{equation*}
    T_{\alpha}:=\conv(\set{\alpha p^1+ (1-\alpha)p^2, \alpha p^2+ (1-\alpha)p^3,\alpha p^3+ (1-\alpha)p^1})
  \end{equation*}
  for $\alpha\in\left[\frac{1}{3},\frac{2}{3}\right]$ (\cf~\Cref{examplejung}). Then,
  \begin{align*}
    D_{\MAX}(T_{\alpha},T)=R(T_{\alpha},T) \quad  \text{and} \quad r(T_{\alpha},T)&=(1-3\alpha+3\alpha^2)R(T_{\alpha},T)
  \end{align*}

  The triangles $T_{\alpha}$ are all equilaterals. Special cases are $\alpha=\frac{1}{2}$, where $T_\frac{1}{2} = -\frac{1}{2}T$, and $\alpha \in \set{\frac{1}{3},\frac{2}{3}}$, rotations of dilatated $T$ by $\pm \frac{\pi}{6}$.

  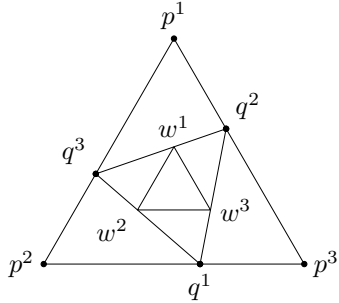
\begin{figure}[ht]
    \centering
    \begin{tikzpicture}[scale=2,vect/.style={->,
        shorten >=1pt,>=latex'}]
        \tkzDefPoint(0,1){p1}
        \tkzDefPoint(-0.86602,-0.5){p2}
        \tkzDefPoint(0.86602,-0.5){p3}
        \tkzDrawPoints[fill=black](p1,p2,p3)

    \tkzDrawPolygon(p1,p2,p3)
    \tkzLabelPoint[above](p1){$p^1$}
    \tkzLabelPoint[left](p2){$p^2$}
    \tkzLabelPoint[right](p3){$p^3$}

    \tkzDefPointOnLine[pos=.6](p2,p3)
    \tkzGetPoint{q1}
    \tkzLabelPoint[below](q1){$q^1$}
    \tkzDefPointOnLine[pos=.6](p3,p1)
    \tkzGetPoint{q2}
    \tkzLabelPoint[above right](q2){$q^2$}
    \tkzDefPointOnLine[pos=.6](p1,p2)
    \tkzGetPoint{q3}
    \tkzLabelPoint[above left](q3){$q^3$}
        \tkzDrawPoints[fill=black](q1,q2,q3)
        \tkzDrawPolygon(q1,q2,q3)
        \tkzDefPointOnLine[pos=.6](q3,q2)
        \tkzGetPoint{w1}
        \tkzLabelPoint[above](w1){$w^1$}
        \tkzDefPointOnLine[pos=.6](q1,q3)
        \tkzGetPoint{w2}
        \tkzLabelPoint[below left](w2){$w^2$}
        \tkzDefPointOnLine[pos=.6](q2,q1)
        \tkzGetPoint{w3}
        \tkzLabelPoint[right](w3){$w^3$}
            \tkzDrawPoints[fill=black](q1,q2,q3)
            \tkzDrawPolygon(w1,w2,w3)

    \end{tikzpicture}
    \caption{All the equilateral triangles $T_\alpha$ are Jung-extremal (\wrt~$T$).}
    \label{examplejung}
\end{figure}
\end{example}

\begin{proof}
  By construction $T_{\alpha}\optc T $. The diameter of $T_{\alpha}$ is attained between two of its vertices and since $T_{\MAX}$ is a regular hexagon which has rotational symmetry of order six, is does not matter which of the edges of $T_{\alpha}$ we consider. Let $q^i=\alpha p^{i+1} + (1-\alpha) p^{i+2}$. By using $p^3=-p^2-p^1$ we obtain
  \begin{align*}
    q^2-q^1&=\alpha p^3+(1-\alpha)p^1 -\alpha p^2-(1-\alpha)p^3\\
           &= (2-3\alpha)p^1 +(3\alpha-1)(-p^2) \\
           &\subset [p^1,-p^2] \subset \bd\left(T_{\MAX}\right)
  \end{align*}
  Hence, $D_{\MAX}(T_{\alpha},T)=\cnorm{q^2-q^1}{T_{\MAX}}=1$.

  The triangle $\conv(\set{w^1,w^2,w^3})$ with $w^i:=\alpha q^{i+2} +(1-\alpha)q^{i+1}$ is optimally contained in $T_{\alpha}$ by \Cref{opt}. Furthermore,
  \begin{align*}
    w^i &= \alpha q^{i+2} +(1-\alpha)q^{i+1} \\
        &= \alpha \left(\alpha p^i+ (1-\alpha)p^{i+1}\right) +(1-\alpha)\left(\alpha p^{i+2}+ (1-\alpha)p^{i}\right) \\
        &= (1-2\alpha+2\alpha^2)p^i +\alpha(1-\alpha)p^{i+1}+\alpha(1-\alpha)p^{i+2}
  \end{align*}
  Using this we compute
  \begin{align*}
    w^{i+1}-w^i &= (1-2\alpha+2\alpha^2)p^{i+1} +\alpha(1-\alpha)p^{i} - (1-2\alpha+2\alpha^2)p^i -\alpha(1-\alpha)p^{i+1}\\
                &= (1-3\alpha-3\alpha^2)(p^{i+1}-p^i).
  \end{align*}
  Thus, $\conv(\set{w^1,w^2,w^3})$ is a translate of $(1-3\alpha-3\alpha^2)T$ and the inradius is $r(T_{\alpha},T)=1-3\alpha-3\alpha^2$.
\end{proof}

$T_{\frac{2}{3}}$ is especially interesting: it
  is complete since its arithmetic mean is a dilatation of $T_{\MAX}$.
Since the symmetrization $T_{\MAX}$ of $T$ is also Jung-extremal, there are at least two complete Jung-extremal bodies, $T_{\MAX}$ and $T_{\frac{2}{3}}$ with different inradius-circumradius ratio.

Now, we compute the values of the functionals for a second family of triangles. We will see that these lie on the boundary of the diagram \wrt triangles.
\begin{lemma}
\label{leftMaxTequalcase}
Let $T$ be the equilateral triangle as defined above and
    $S_{\lambda}=\conv(\{q^1,q^2,q^3\})$
with $q^1=\frac{1}{2}(p^2+p^3)$, $q^2=\lambda p^1 + (1-\lambda)p^3$ and $q^3=\lambda p^1 + (1-\lambda)p^2$
for some $\lambda\in [\frac{1}{2},1]$.
Then, $R(S_{\lambda},T)=1, D_{\MAX}(S_{\lambda},T)=\lambda+\frac{1}{2}$ and $r(S_{\lambda},T)=\lambda(1-\lambda)$.
\end{lemma}
\begin{proof} By construction, $R(S_\lambda,T)=1$.
  Obviously,
  $\norm{q^2-q^3}_{T_{\MAX}}\leq 1$ and $\norm{q^2-q^1}_{T_{\MAX}}=\norm{q^3-q^1}_{T_{\MAX}}$. Using $p^3=-(p^1+p^2)$ we can compute:
\begin{equation*}
    \begin{split}
        q^2-q^1&= \lambda p^1 + (1-\lambda)p^3 - \frac{1}{2}(p^2+p^3) = \lambda p^1 + \frac{1}{2}(-p^2) + (\frac{1}{2}-\lambda)(-p^1 -p^2)\\
        &= (2\lambda -\frac{1}{2})p^1 +(1-\lambda)(-p^2) \in (\lambda+\frac{1}{2}) \conv(\{p^1,-p^2\})\\
        &\subset (\lambda+\frac{1}{2}) \bd(T_{\MAX}).
    \end{split}
\end{equation*}
Since $\lambda+\frac{1}{2}\geq 1$ for $\lambda \in [\frac{1}{2},1]$, it follows $D_{\MAX}(S_{\lambda},T)=\lambda+\frac{1}{2}$.

Now we show the formula for the inradius (\cf~\Cref{proofleftMaxTequalcase}). Let $\conv(\set{w^1,w^2,w^3})$ be the inner triangle. By axial symmetry of $T$ and $S_{\lambda}$ we know that $w^1=\frac{1}{2}(q^2+q^3)$.
Denote the euclidean edge length of the inner triangle by $a$. Since $T$ has edge length $\sqrt{3}$, we obtain $r(S_{\lambda},T)=\frac{a}{\sqrt{3}}$. The segments $[q^3,q^2]$ and $[p^2,p^3]$ are parallel and the corresponding edges of the inner and outer triangle are parallel as well. Thus, the triangles $\conv(\set{w^1,w^3,q^2})$ and $\conv(\set{p^3,q^2,q^1})$ are similar, implying $\frac{a}{\norm{q^2-p^3}_2}=\frac{q^2_1}{p^3_1}$. We obtain
\begin{equation*}
    a = \frac{q^2_1}{p^3_1} \cdot \norm{q^2-p^3}_2 = (1-\lambda)\norm{q^2-p^3}_2
      = (1-\lambda)\lambda \norm{p^1 - p^3}_2 = \sqrt{3}\lambda (1-\lambda),
\end{equation*}
and therefore $r(S_{\lambda},T)=\frac{a}{\sqrt{3}}=\lambda (1-\lambda)$.
\end{proof}

\begin{figure}[ht]
    \centering
    \begin{tikzpicture}[scale=3]
    \tkzDefPoint(0,1){p1}
    \tkzDrawPoint[radius=0.4pt,label=above:$p^1$](p1)
    \tkzDefPoint(-0.86602,-0.5){p2}
    \tkzDrawPoint[radius=0.4pt,label=left:$p^2$](p2)
    \tkzDefPoint(0.86602,-0.5){p3}
    \tkzDrawPoint[radius=0.4pt,label=right:$p^3$](p3)
    \tkzDrawPolygon[black, thick](p1,p2,p3);
      \tkzDefPoint(0,-0.5){q1}
    \tkzDrawPoint[radius=0.4pt,label=below:$q^1$](q1)
    \tkzDefPoint(0.28867,0.5){q2}
    \tkzDrawPoint[radius=0.4pt,label=right:$q^2$](q2)
    \tkzDefPoint(-0.28867,0.5){q3}
    \tkzDrawPoint[radius=0.4pt,label=left:$q^3$](q3)
    \tkzDrawPolygon[black, thick](q1,q2,q3);
        \tkzDefPoint(0,0.5){w1}
        \tkzDrawPoint[radius=0.4pt](w1)
        \tkzLabelPoint[above](w1){$w_1$}
        \tkzDefPoint(-0.19245,0.166666){w2}
        \tkzDrawPoint[radius=0.4pt](w2)
        \tkzLabelPoint[left](w2){$w^2$}
        \tkzDefPoint(0.19245,0.166666){w3}
        \tkzDrawPoint[radius=0.4pt](w3)
        \tkzLabelPoint[right](w3){$w^3$}
    \node[label=left:\textcolor{blue}{$a$}] (a) at (0.12,0.3) {};
    \tkzDrawSegments[black, thick](w1,w2 w2,w3);
    \tkzDrawSegment[blue,thick](w3,w1)
    \tkzFillAngle[fill=teal!25, opacity=0.5,size=0.2](p3,q1,q2)
    \tkzMarkAngle[size=0.2,mark=none](p3,q1,q2);
    \tkzFillAngle[fill=teal!25, opacity=0.5,size=0.1](w1,q2,w3)
    \tkzMarkAngle[size=0.1,mark=none](w1,q2,w3);

    \tkzFillAngle[fill=orange!25, opacity=0.5,size=0.2](q2,p3,q1)
    \tkzMarkAngle[size=0.2,mark=none](q2,p3,q1);

    \tkzFillAngle[fill=orange!25, opacity=0.5,size=0.1](w3,w1,q2)
    \tkzMarkAngle[size=0.1,mark=none](w3,w1,q2);

    \end{tikzpicture}
    \caption{Calculating the inradius of $S_\lambda$.}
    \label{proofleftMaxTequalcase}
\end{figure}
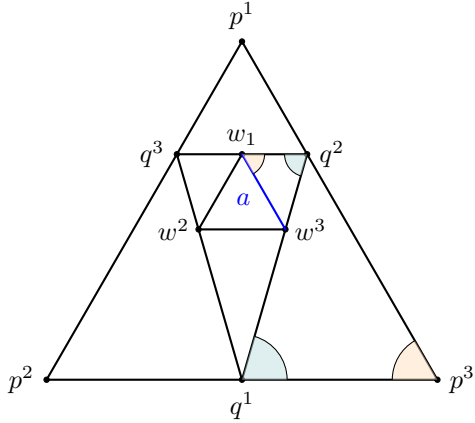

$S_{\frac 12} = -\frac{1}{2}T$ is Jung-extremal and $S_1$ is a segment $L_w$.

\begin{theorem} \label{leftMaxT}
 Let $K \in \bar{\CC}^n$ and $T$ an equilateral Minkowski-centered triangle. Then
  \begin{equation*}
    \left(\frac{D_{\MAX}(K,T)}{R(K,T)}-\frac{1}{2}\right)\left(\frac{3}{2}-\frac{D_{\MAX}(K,T)}{R(K,T)}\right) \leq \frac{r(K,T)}{R(K,T)}
  \end{equation*}
  with equality for the triangles $S_{\lambda}$, $\lambda \in [\frac12,1]$, as described in \Cref{leftMaxTequalcase}.
\end{theorem}

To prepare the proof of \Cref{leftMaxT} we need the following lemma.

\begin{lemma} \label{steep}
  Let $K=\conv\left(\{q^1,q^2,q^3\}\right)$ be a triangle that is optimally contained in $T$, \st~$q^i$ belongs to the edge of $T$ opposing $p^i$.
  We say that $(q^{1}, q^{3})$ is \emph{steep} if $\norm{q^3-p^1}_2\leq \norm{q^1-p^3}_2$. Let $\tilde{K}:=\conv(\{q^{1},q^{3},\tilde{q}^2\})\optc T$ be a triangle such that $\tilde{q}^2$ is on the same edge of $T$ as $q^2$ and  $\norm{q^2-p^1}_2\leq \norm{\tilde{q}^2-p^1}_2$. If $(q^{1}, q^{3})$ is steep, the inradius of $\tilde{K}$ \wrt~$T$ is greater or equal than the one of $K$.
\end{lemma}

We call this property "steep" since the segment $[q^1,q^3]$ is steeper than $[p^3,p^1]$. By symmetry of $T$ we can generalize this result to all choices $i,j\in\set{1,2,3}$: $(q^{i}, q^{j})$, $i,j\in\{1,2,3\}$, $i\neq j$ is steep if $\norm{q^j-p^i}_2\leq \norm{q^i-p^j}_2$. At least one of the ordered pairs $(q^i,q^j)$ or $(q^j,q^i)$ is always steep.

\begin{proof}[Proof of \Cref{steep}]
  Let $H_1$ and $H_2$ be the two lines parallel to $[p^1,p^3]$ supporting the inner triangle $\conv(\set{w^1,w^2,w^3})$ of $K$ (that necessarily touches all three edges of $K$), s.t.~$H_1$ contains the edge $[w^1,w^3]$ and $H_2$ the opposing vertex $w^2$ (\cf~\Cref{proofsteep}). Since $(q^{1}, q^{3})$ is steep the part of $H_2$ below $w^2$ intersects $\tilde{K}$. By the intercept theorem with the two parallel lines $H_1$ and $H_2$ and the points $q^2$ or $\tilde{q}^2$ the segment of $H_1$ contained in $\tilde{K}$ is greater or equal than the one contained in $K$. Thus, we can move the inner triangle of $K$ within the slab between $H_1$ and $H_2$ until it touches $[q^1,\tilde{q}^2]$. The resulting translations of $w^1,w^2,w^3$ are all contained in $\tilde{K}$, the translation of $w^2$ due to the steepness of $(q^1,q^3)$. Thus, $r(K,T)\leq r(\tilde{K},T)$.
\end{proof}

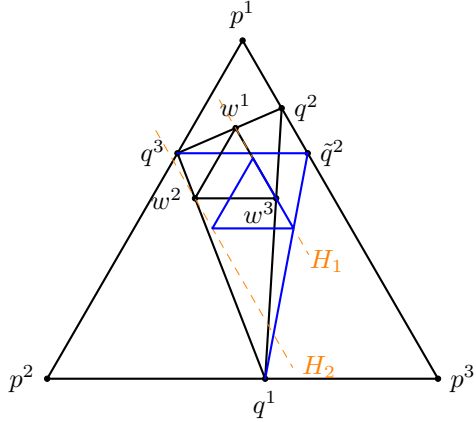
\begin{figure}[ht]
  \centering
  \begin{tikzpicture}[scale=3]
    \tkzDefPoint(0,1){p1}
    \tkzDrawPoint[radius=0.4pt,label=above:$p^1$](p1)
    \tkzDefPoint(-0.86602,-0.5){p2}
    \tkzDrawPoint[radius=0.4pt,label=left:$p^2$](p2)
    \tkzDefPoint(0.86602,-0.5){p3}
    \tkzDrawPoint[radius=0.4pt,label=right:$p^3$](p3)
    \tkzDrawPolygon[black, thick](p1,p2,p3);

    \tkzDefPoint(0.1,-0.5){q1}
    \tkzDrawPoint[radius=0.4pt,label=below:$q^1$](q1)
    \tkzDefPoint(0.173205,0.7){q2}
    \tkzDrawPoint[radius=0.4pt,label=right:$q^2$](q2)
    \tkzDefPoint(-0.28867,0.5){q3}
    \tkzDrawPoint[radius=0.4pt,label=left:$q^3$](q3)
    \tkzDrawPolygon[black, thick](q1,q2,q3);
    \tkzDefPoint(0.28867,0.5){q2t}
    \tkzDrawPoint[radius=0.4pt,label=right:$\tilde{q}^2$](q2t);
    \tkzDrawPolySeg[blue,thick](q1,q2t,q3);


    \tkzDefPoint(-0.0310653, 0.611546){w1}
    \tkzDrawPoint[radius=0.2pt,label=above:$w^1$](w1)
    \tkzDefPoint(-0.210935, 0.3){w2}
    \tkzDrawPoint[radius=0.2pt,label=left:$w^2$](w2)
    \tkzDefPoint(0.148805, 0.3){w3}
    \tkzDrawPoint[radius=0.2pt](w3)
    \tkzLabelSegment[left, pos=1.2](w1,w3){$w^3$}
    \tkzDrawPolygon[thick](w1,w2,w3);
    \tkzDefLine[parallel=through w2](p1,p3) \tkzGetPoint{c1};
    \tkzDrawLine[color=orange,dashed, add=0.2 and -0.5](w2,c1);
    \tkzDrawLine[color=orange,dashed, add=0.4 and 0.8](w1,w3);
    \tkzLabelLine[color=orange,below right, pos=1.6](w1,w3){$H_1$}
    \tkzLabelLine[color=orange,right](w2,c1){$H_2$}
    \tkzInterLL(w1,w3)(q1,q2t)
    \tkzGetPoint{w3t}
    \tkzDefPointsBy[translation= from w3 to w3t](w1,w2){w1t,w2t}
    \tkzDrawPolygon[thick,color=blue](w3t,w2t,w1t);
  \end{tikzpicture}
  \vspace*{-1.5cm}
  \caption{Proof of \Cref{steep}. Since $(q^1,q^3)$ is steep, the inner triangle can be translated along the orange hyperplanes.}
  \label{proofsteep}
\end{figure}

\begin{proof}[Proof of \Cref{leftMaxT}]
  The equality case follows directly from \Cref{leftMaxTequalcase}. So we only need to prove the correctness of the inequality.

  Let $K \optc T$. As shown in the proof of \Cref{thm:jungbound_Dmax_mink}, we either have $D_{\MAX}(K,T)\geq \frac{3}{2}$ or three touching points of $K$ to the boundary of $T$, each situated on a different edge of $T$.

  In the first case, the left side of the inequality in \Cref{leftMaxT} is non-positive while the right side is always non-negative. Hence, in this case the inequality is fulfilled.

  In the other case we consider the triangle $S:=\conv(\{q^1,q^2,q^3\})$, where $q^i\in K $ belongs to the edge of $T$ opposing $p^i$, $i = 1, 2, 3$. Assume \Wlog that the diameter of $S$ is attained between $q^1$ and one of the other points and that $\norm{q^1-p^3}_2\leq\norm{q^1-p^2}_2$. Our goal is to show that there exists a triangle $S_{\lambda}\optc T$, $\lambda\in[\frac{1}{2},1]$, as defined in \Cref{leftMaxTequalcase} with at most the same diameter and inradius as $S$. Using the fact that $\left(D_{\MAX}-\frac{1}{2}\right)\left(\frac{3}{2}-D_{\MAX}\right)$ is decreasing in $D_{\MAX}$ if $D_{\MAX}\geq 1$, we may then conclude
\begin{equation*}
\begin{split}
     \left(\frac{D_{\MAX}(K,T)}{R(K,T)}-\frac{1}{2}\right)\left(\frac{3}{2}-\frac{D_{\MAX}(K,T)}{R(K,T)}\right) &\leq \left(\frac{D_{\MAX}(S,T)}{R(S,T)}-\frac{1}{2}\right)\left(\frac{3}{2}-\frac{D_{\MAX}(S,T)}{R(S,T)}\right) \\
     &\leq \left(\frac{D_{\MAX}(S_{\lambda},T)}{R(S_{\lambda},T)}-\frac{1}{2}\right)\left(\frac{3}{2}-\frac{D_{\MAX}(S_{\lambda},T)}{R(S_{\lambda},T)}\right) \\
     &= \frac{r(S_{\lambda},T)}{R(S_{\lambda},T)} \leq \frac{r(S,T)}{R(S,T)}
     \leq \frac{r(K,T)}{R(K,T)}.
\end{split}
\end{equation*}
We distinguish between the two cases if the diameter is attained by $[q^1,q^2]$ or $[q^1,q^3]$ and show that we may always assume that both segments are diametral. 

\begin{enumerate}[{Case} 1]
\item
  This case is depicted in \Cref{proofleftsidefig1a}. If $D_{\MAX}(S,T)=D_{\MAX}([q^1,q^2],T)$ our assumption $\norm{q^1-p^3}_2\leq\norm{q^1-p^2}_2$ implies that $\norm{q^2-p^1}_2\leq\norm{q^2-p^3}_2$. Let us assume otherwise.
  Then we would have  $[q^1,q^2] \subset \conv(\set{p^3,\frac{1}{2}(p^2+p^3),\frac{1}{2}(p^1+p^3)}=\frac{1}{2}(p^3+T)$ with $[q^1,q^2]$ not being an edge of $\frac12(T+p^3)$, which implies $D_{\MAX}([q^1,q^2],T)=2R([q^1,q^2],T_{\MAX})<2R(\frac12 T,T_{\MAX}) = 1 = R(K,T)$, contradicting \Cref{thm:jungbound_Dmax_mink}.

  Hence, $\norm{q^2-p^1}_2\leq \sqrt{3}/2 \leq \norm{q^1-p^2}_2$, which implies the steepness of $(q^1,q^2)$. Since $[q^1,q^2]$ is diametral, $q^3$ lies inside $q^1+D_{\MAX}(S,T)T_{\MAX}$. However, $\norm{q^1-p^3}_2\leq\norm{q^1-p^2}_2$ now implies the existence of an intersection point between  $[p^1,p^2]$ and the boundary of $q^1+D_{\MAX}(S,T)T_{\MAX}$  which is not further from $p^1$ than $q^3$. Choosing this point as our new $q^3$ neither increases the diameter nor the inradius (the latter because of \Cref{steep}). Doing so, $[q^1,q^3]$ becomes diametral, too.

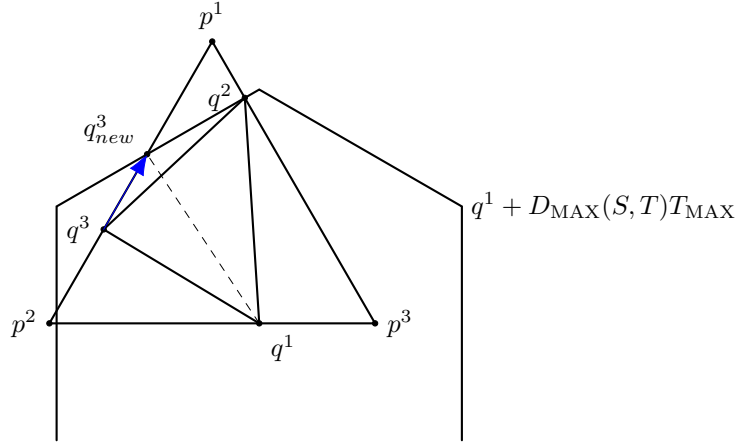
\begin{figure}[ht]
    \centering
    \begin{tikzpicture}[scale=2.5]
    \tkzDefPoint(0,1){p1}
    \tkzDrawPoint[radius=0.4pt,label=above:$p^1$](p1)
    \tkzDefPoint(-0.86602,-0.5){p2}
    \tkzDrawPoint[radius=0.4pt,label=left:$p^2$](p2)
    \tkzDefPoint(0.86602,-0.5){p3}
    \tkzDrawPoint[radius=0.4pt,label=right:$p^3$](p3)
    \tkzDrawPolygon[black, thick](p1,p2,p3);
    \tkzDefPoint(0,-1){p1m}
    \tkzDefPoint(0.86602,0.5){p2m}
    \tkzDefPoint(-0.86602,0.5){p3m}
    \tkzDefPoint(0.173205,-0.5){q1t}
    \tkzDefPoint(0.173205,0.7){q2}
    \tkzDrawPoint[radius=0.4pt,label=left:$q^2$](q2);
    \tkzDefPoint(0.25,-0.5){q1}
    \tkzDrawPoint[radius=0.4pt,label=below right:$q^1$](q1);
    \tkzDefLine[orthogonal=through q2](p3,p1) \tkzGetPoint{c1};
    \tkzInterLL(q2,c1)(p1,p2) \tkzGetPoint{q3}
    \tkzDrawPoint[radius=0.4pt,label=above left:$q^3_{new}$](q3);
    \tkzDefPoint(-0.5773,0){q3t}
    \tkzDrawPoint[radius=0.4pt,label=left:${q}^3$](q3t);
    \tkzDefLine[orthogonal=through q2](p1,p2) \tkzGetPoint{c2};
    \tkzDefLine[orthogonal=through q1](p2,p3)
    \tkzGetPoint{c3};
    \tkzInterLL(q1,c3)(q2,q3)
    \tkzGetPoint{lp}
    \tkzDefLine[parallel=through lp](q2,c2)
    \tkzGetPoint{c4};
    \tkzDrawPolygon[black, thick](q1,q2,q3t);
    \tkzDrawPolySeg[black,dashed](q1,q3)

\tkzDrawSegment[-{Latex[scale=2]},blue](q3t,q3)

\tkzDefPoint(0,0){z}
\tkzDefPointsBy[translation= from z to q1](p2,p3m,p1,p2m,p3){a1,a2,a3,a4,a5}
\tkzDefPointsBy[homothety=center q1 ratio 1.24434](a1,a2,a3,a4,a5){b1,b2,b3,b4,b5}
\tkzDrawPolySeg[black, thick](b1,b2,b3,b4,b5);
\tkzLabelPoint[right](b4){$q^1+D_{\MAX}(S,T)T_{\MAX}$}

    \end{tikzpicture}
    \caption{Proof of Case 1 of \Cref{leftMaxT}. The diameter is attained between $q^1$ and $q^2$. We can replace $q^3$ such that it is attained between $q^1$ and $q^3$ as well since $(q^1,q^2)$ is steep.}
    \label{proofleftsidefig1a}
\end{figure}
\item If $D_{\MAX}(S,T)=D_{\MAX}([q^1,q^3],T)$, we need to consider three subcases:
  \begin{enumerate}[a)]
  \item If $\norm{q^3-p^2}_2\leq\norm{q^3-p^1}_2$ this corresponds to Case 1 with $q^3$ in the role of $q^1$ and $q^2$ being the vertex that is moved.
  \item If $\norm{q^3-p^2}_2\geq\norm{q^3-p^1}_2$ and $(q^1,q^3)$ is steep one can move $q^2$ in the direction of $p^1$ such that $[q^1,q^2]$ becomes diametral, too.
  \item If $\norm{q^3-p^2}_2\geq\norm{q^3-p^1}_2$ and $(q^3,q^1)$ is steep this corresponds to Case 2b) with roles of $q^1$ and $q^3$ interchanged, which means that we may move $q^2$ towards $p^3$.
  \end{enumerate}
\end{enumerate}

Altogether we see that assuming $\norm{q^2-q^1}_{T_{\MAX}}=\norm{q^3-q^1}_{T_{\MAX}}$ is possible and doing so the points $q^2$ and $q^3$ do not only lie on the boundary of $q^1+ D_{\MAX}(S,T)T_{\MAX}$, they essentialy lie on the (translated and dilatated) edges $[p^1,-p^3]$ or $[p^1,-p^2]$ of $T_{\MAX}$. For $q^2$ this follows from the fact that it has to lie closer to $p^1$ than to $p^3$. 
As described in the proof of Case 1, the boundary of  $q^1+ D_{\MAX}(S,T)T_{\MAX}$ intersects $[p^1,p^2]$ once or twice, but it is not possible that it only intersects with the segment $q^1 + D_{\MAX}(S,T)[p^2,-p^3]$
as this would contradict our assumption $D_{\MAX}(S,T) = D_{\MAX}([q^1,q^2],T)$. If we have two intersection points, we can replace $q^3$ if necessary by the upper one without increasing the inradius since $(q^1,q^2)$ is steep. Thus, we can assume $q^3 \in q^1+D_{\MAX}(S,T)[-p^3,p^1]$.

For the next part of the proof we now assume that $q^3 \in q^1+D_{\MAX}(S,T)[p^1,-p^3]$ and $q^2 \in q^1+D_{\MAX}(S,T)[p^1,-p^2]$ is true. This is not the case if  $q^1_1>q^2_1$ (see \Cref{proofleftsidefig1}). But then, we know $q^3_2<q^2_2$ and $(q^2,q^3)$ is steep. Thus, replacing $q^1$ by $\tilde{q}^1$ such that $\tilde{q}^1_1=q^2_1$ does not increase the diameter or the inradius and we still have $\cnorm{q^3-q^1}{T_{\MAX}}=\cnorm{q^2-q^1}{T_{\MAX}}$. This shows that we can assume $q^1_1\leq q^2_1$ and that $q^2 \in q^1+D_{\MAX}(S,T)[-p^2,p^1]$.

\begin{figure}[ht]
    \centering
    \begin{tikzpicture}[scale=2.5]
    \tkzDefPoint(0,1){p1}
    \tkzDrawPoint[radius=0.4pt,label=above:$p^1$](p1)
    \tkzDefPoint(-0.86602,-0.5){p2}
    \tkzDrawPoint[radius=0.4pt,label=left:$p^2$](p2)
    \tkzDefPoint(0.86602,-0.5){p3}
    \tkzDrawPoint[radius=0.4pt,label=right:$p^3$](p3)
    \tkzDrawPolygon[black, thick](p1,p2,p3);
    \tkzDefPoint(0,-1){p1m}
    \tkzDefPoint(0.86602,0.5){p2m}
    \tkzDefPoint(-0.86602,0.5){p3m}
    \tkzDefPoint(0.173205,-0.5){q1t}
    \tkzDrawPoint[radius=0.4pt,label=below left:$\tilde{q}^1$](q1t);
    \tkzDefPoint(0.173205,0.7){q2}
    \tkzDrawPoint[radius=0.4pt,label=left:$q^2$](q2);
    \tkzDefPoint(0.25,-0.5){q1}
    \tkzDrawPoint[radius=0.4pt,label=below right:$q^1$](q1);
    \tkzDefLine[orthogonal=through q2](p3,p1) \tkzGetPoint{c1};
    \tkzInterLL(q2,c1)(p1,p2) \tkzGetPoint{q3}
    \tkzDrawPoint[radius=0.4pt,label=left:$q^3$](q3);
    \tkzDefLine[orthogonal=through q2](p1,p2) \tkzGetPoint{c2};
    \tkzDefLine[orthogonal=through q1](p2,p3)
    \tkzGetPoint{c3};
    \tkzInterLL(q1,c3)(q2,q3)
    \tkzGetPoint{lp}
    \tkzDefLine[parallel=through lp](q2,c2)
    \tkzGetPoint{c4};
    \tkzDrawPoint[radius=0.4pt](lp)
    \tkzDrawPolygon[black, thick](q1,q2,q3);
    \tkzDrawPolySeg[blue, thick](q3,q1t,q2);
    \tkzDrawPolySeg[black,dashed](q1,lp)
    \tkzMarkRightAngle[blue, size=0.05](p2,q1t,q2)
    \tkzMarkRightAngle[black, size=0.05](p3,q1,lp)
\tkzDefPoint(0,0){z}
\tkzDefPointsBy[translation= from z to q1t](p2,p3m,p1,p2m,p3){v1,v2,v3,v4,v5}
\tkzDefPointsBy[homothety=center q1t ratio 1.2](v1,v2,v3,v4,v5){w1,w2,w3,w4,w5}
\tkzDrawPolySeg[blue, thick](w1,w2,w3,w4,w5);
\tkzDefPointsBy[translation= from z to q1](p2,p3m,p1,p2m,p3){a1,a2,a3,a4,a5}
\tkzDefPointsBy[homothety=center q1 ratio 1.24434](a1,a2,a3,a4,a5){b1,b2,b3,b4,b5}
\tkzDrawPolySeg[black, thick](b1,b2,b3,b4,b5);
\tkzLabelSegment[above right](lp,b4){$q^1+D_{\MAX}(S,T)T_{\MAX}$}
\tkzLabelPoint[left,blue](w2){$\tilde{q}^1+D_{\MAX}(\tilde S,T)T_{\MAX}$};
    \end{tikzpicture}
    \caption{Proof of \Cref{leftMaxT}. We can assume that $q^3$ lies on $q^1+D_{\MAX}(S,T)[p^1,-p^3]$ and $q^2$ lies on $q^1+D_{\MAX}(S,T)[p^1,-p^2]$. Otherwise we can consider the triangle $\tilde S = \conv(\set{\tilde{q}^1,q^2,q^3})$ which has smaller or equal inradius and diameter.}
    \label{proofleftsidefig1}
\end{figure}
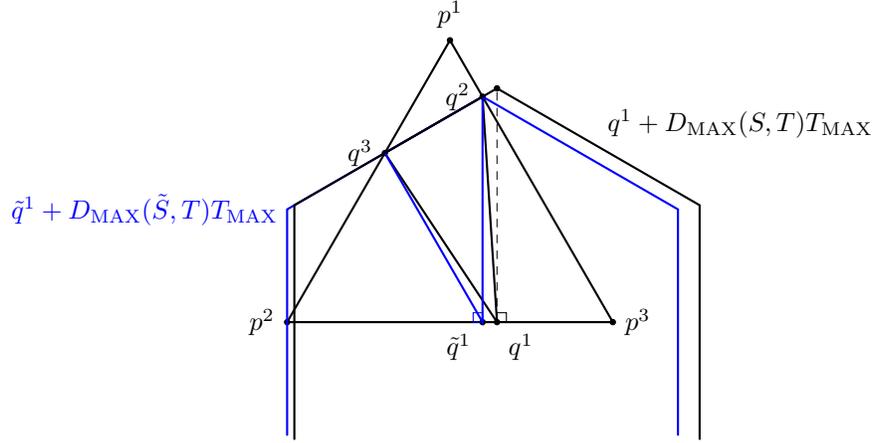
Now, we consider the triangle $S_{\lambda}=\conv(\{\tilde{q}^1,\tilde{q}^2,\tilde{q}^3\})$ with $\lambda$ as in \Cref{leftMaxTequalcase} such that it has the same diameter. Due to symmetry reasons and our assumptions about the positions of $q^2$ and $q^3$, $\norm{q^1-\tilde{q}^1}_2=\norm{q^2-\tilde{q}^2}_2=\norm{q^3-\tilde{q}^3}_2=:\kappa$ (see \Cref{proofleftsideisosceles}).

\begin{figure}[ht]
    \centering
    \begin{tikzpicture}[scale=3]
    \tkzDefPoint(0,1){p1}
    \tkzDrawPoint[radius=0.4pt,label=above:$p^1$](p1)
    \tkzDefPoint(-0.86602,-0.5){p2}
    \tkzDrawPoint[radius=0.4pt,label=left:$p^2$](p2)
    \tkzDefPoint(0.86602,-0.5){p3}
    \tkzDrawPoint[radius=0.4pt,label=right:$p^3$](p3)
      \tkzDefPoint(0,-0.5){q1t}
    \tkzDrawPoint[radius=0.4pt](q1t)

    \tkzDefPoint(0.28867,0.5){q2t}
    \tkzDrawPoint[radius=0.4pt](q2t)

    \tkzDefPoint(-0.28867,0.5){q3t}
    \tkzDrawPoint[radius=0.4pt](q3t)
      \tkzDefPoint(0.2,-0.5){q1}
    \tkzDrawPoint[radius=0.4pt](q1)

    \tkzDefPoint(0.18867,0.6732){q2}
    \tkzDrawPoint[radius=0.4pt](q2)

    \tkzDefPoint(-0.38867,0.3268){q3}
    \tkzDrawPoint[radius=0.4pt](q3)

    \tkzDrawLine[dashed,add=0.3 and 0.4](q3,q2)
    \tkzDefLine[parallel=through q3t](q3,q2)
    \tkzGetPoint{c1}
    \tkzDrawLine[dashed,add=0.3 and 0.2](q3t,c1);
     \tkzDefLine[parallel=through q1t](q3,q2)
    \tkzGetPoint{c2}
    \tkzDrawLine[dashed](q1t,c2);;
     \tkzDefLine[parallel=through q1](q3,q2)
    \tkzGetPoint{c3}
    \tkzDrawLine[dashed](q1,c3);
     \tkzDefLine[orthogonal=through q1](p2,p1)
    \tkzGetPoint{c4}
    \tkzDrawLine[dashed, add=0.2 and -0.4](q1,c4);
      \tkzDefLine[orthogonal=through q1t](p2,p1)
    \tkzGetPoint{c5}
    \tkzDrawLine[dashed, add=0.2 and -0.4](q1t,c5);
      \tkzDefLine[orthogonal=through q2](p2,p1)
    \tkzGetPoint{c6}
    \tkzDrawLine[dashed, add=0.2 and -0.6](q2,c6);
      \tkzDefLine[orthogonal=through q2t](p2,p1)
    \tkzGetPoint{c7}
    \tkzDrawLine[dashed, add=0.2 and -0.5](q2t,c7);

    \tkzLabelPoint[above right, fill=white ](q2){$q^2$}
    \tkzLabelPoint[below left, fill=white](q1t){$\tilde{q}^1$}
    \tkzLabelPoint[below right, fill=white](q1){$q^1$}
    \tkzLabelSegment[left, pos=0.5, fill=white](p2,p1){$q^3$}
    \tkzLabelSegment[right, pos=0.6, fill=white](p3,p1){$\tilde{q}^2$}
    \tkzLabelPoint[above left, fill=white](q3t){$\tilde{q}^3$}
    \tkzLabelSegment[below,color=blue, fill=white](q1t,q1){$\kappa$}
    \tkzLabelSegment[right,color=blue, fill=white](q2t,q2){$\kappa$}
    \tkzLabelSegment[left,color=blue, fill=white](q3t,q3){$\kappa$};
    \tkzDrawPolygon[black, thick](q1,q2,q3);
    \tkzDrawPolygon[black, thick](q1t,q2t,q3t);

    \tkzDrawPolygon[black, thick](p1,p2,p3);
    \tkzDrawSegment[blue,thick](q1t,q1)
    \tkzDrawSegment[blue,thick](q2t,q2)
    \tkzDrawSegment[blue,thick](q3t,q3)
    \end{tikzpicture}
    \caption{Proof of \Cref{leftMaxT}. Transformation of $S$ into $S_{\lambda}$ with the same diameter. The distances $\cnorm{q^i-\tilde{q}^i}{2}$ are equal.}
    \label{proofleftsideisosceles}
\end{figure}
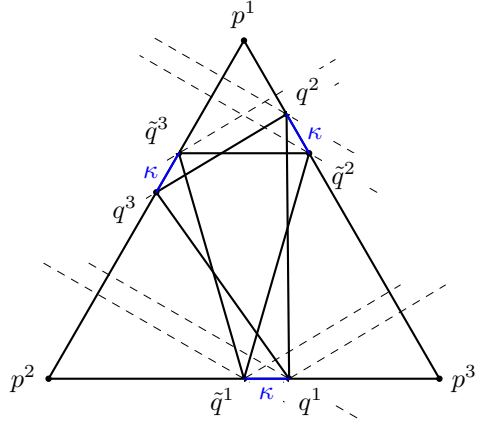
By using the intercept theorem and the law of sines we will show that, possibly after a suitable translation, all vertices of the inner triangle of $S_{\lambda}$ are contained in $S$. This way we see that the inradius of $S_{\lambda}$ is at most the one of $S$. We denote the vertices of the inner triangle of $S_\lambda$ by $w^i$ (see \Cref{proofleftsidelabel_l23}) and the euclidean distance in the horizontal direction of $w^i$ to the segment $[q^j,q^k]$, $\set{i,j,k}=\set{1,2,3}$, by $l_i$. If $\kappa \neq0$, every side of $S$ intersects the corresponding side of $S_{\lambda}$ exactly once. Let us denote these intersection points by $v^i$ (\cf~\Cref{proofleftsidelabel_l1}).
We will show that if we shift the inner triangle of $S_\lambda$ by $l_3$ to the left it is completely contained in $S$. To do so we compute all the values $l_i$, $i=1,2,3$ and prove that we have $l_3 \le l_j$, $j=1,2$.

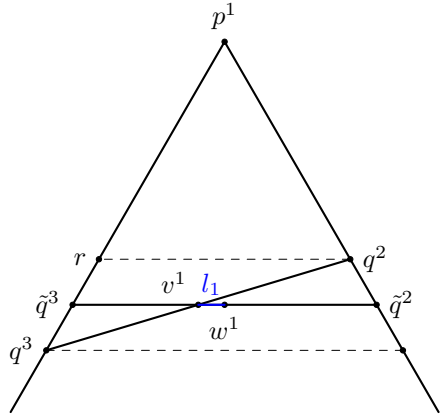
\begin{figure}[ht]
    \centering
    \begin{tikzpicture}[scale=7]
      \tkzDefPoint(0,1){p1}
    \tkzDrawPoint[radius=0.4pt,label=above:$p^1$](p1)
    \tkzDefPoint(0.28867,0.5){q2t}
    \tkzDrawPoint[radius=0.4pt,label=right:$\tilde{q}^2$](q2t)
    \tkzDefPoint(-0.28867,0.5){q3t}
    \tkzDrawPoint[radius=0.4pt,label=left:$\tilde{q}^3$](q3t)
    \tkzDrawPolySeg[black, thick](q2t,q3t);
    \tkzDefPoint(0.23867,0.5866){q2}
    \tkzDrawPoint[radius=0.4pt,label=right:$q^2$](q2)
    \tkzDefPoint(-0.33867,0.4134){q3}
    \tkzDrawPoint[radius=0.4pt,label=left:$q^3$](q3)
    \tkzDrawPolySeg[black, thick](q3,q2);
    \tkzDefPoint(0,0.5){w1}
    \tkzDrawPoint[radius=0.4pt,label=below:$w^1$](w1)
    \tkzInterLL(q3,q2)(q3t,q2t)
    \tkzGetPoint{f1}
    \tkzDrawPoint[radius=0.4pt,label=above left:$v^1$](f1)
    \tkzDrawSegment[blue,thick](f1,w1)
    \tkzLabelSegment[above,color=blue](f1,w1){$l_1$};
     \tkzDefPoint(-0.23867,0.5866){q2p}
    \tkzDrawPoint[radius=0.4pt,label=left:$r$](q2p)
    \tkzDefPoint(0.33867,0.4134){q3p}
    \tkzDrawPoint[radius=0.4pt](q3p)
    \tkzDrawSegment[dashed](q2,q2p)
    \tkzDrawSegment[dashed](q3,q3p);
    \tkzDrawLine[black,thick,add=0 and 0.2](p1,q3p)
    \tkzDrawLine[black,thick,add=0 and 0.2](p1,q3);
    \end{tikzpicture}
    \caption{Proof of \Cref{leftMaxT}. Computation of $l_1$. The triangles defined by $p^1$ and the intersection points of the parallel lines are all three equilateral.}
    \label{proofleftsidelabel_l1}
    \end{figure}

  Computation of $l_1$ (\cf~\Cref{proofleftsidelabel_l1}): We use the intercept theorem for $[\tilde{q}^2,\tilde{q}^3]$ and the two lines parallel to this segment through $q^2$ and $q^3$, respectively. Since $T$ is an equilateral triangle, $\conv(\set{p^1,\tilde{q}^3,\tilde{q}^2})$ is also equilateral with an edge length of $(1-\lambda)\sqrt{3}$. Furthermore,
  \[\normeuc{\tilde{q}^3-v^1} = \frac{1}{2}\normeuc{r-q^2} = \frac{1}{2}\normeuc{p^1-q^2} = \frac{1}{2}(\normeuc{p^1 - \tilde{q}^2}-\kappa).\]
  It follows that $w^1$ is always contained in $S$ and
  \begin{equation*}
      l_1 = \frac{1}{2}\normeuc{p^1 - \tilde{q}^2} - \normeuc{\tilde{q}^3-v^1} = \frac{1}{2}\kappa.
  \end{equation*}

  Computation of $l_2$ (\cf~\Cref{proofleftsidelabel_l23}): Let $\alpha_2=\angle v^2q^3\tilde{q}^3$, $\beta_2=\angle v^2q^1p^2$ and $\gamma_2=\angle \tilde{q}^1v^2q^1$. We compute $\frac{\normeuc{\tilde{q}^3-v^2}}{\normeuc{\tilde{q}^1-v^2}}$ using the law of sines, first for the triangles $\conv(\set{v^2,q^1,\tilde{q}^1})$ and $\conv(\set{v^2,q^3,\tilde{q}^3})$, and then for $\conv(\set{p^2,q^1,q^3})$.

  \begin{equation*}
    \begin{split}   \frac{\normeuc{\tilde{q}^3-v^2}}{\normeuc{\tilde{q}^1-v^2}}&=\frac{\sin{(\gamma_2)}}{\kappa\sin{(\beta_2)}}\cdot\frac{\kappa\sin{(\alpha_2)}}{\sin{(\gamma_2)}} = \frac{\sin{(\alpha_2)}}{\sin{(\beta_2)}}\\
        &=\frac{\sin{(\pi-\alpha_2)}}{\sin{(\beta_2)}} = \frac{\frac{\sqrt{3}}{2}+\kappa}{\sqrt{3}\lambda-\kappa}
    \end{split}
\end{equation*}
Furthermore, we know from \Cref{leftMaxTequalcase} that $r(S_{\lambda},T)=\lambda(1-\lambda)$. Together with the intercept theorem we obtain
\begin{align*}
  \frac{\normeuc{\tilde{q}^3-w^2}}{\normeuc{\tilde{q}^1-w^2}} = \frac{w^1_2-w^2_2}{(\tilde{q}^3_2-\tilde{q}^1_2)-(w^1_2-w^2_2)} = \frac{1-\lambda}{\lambda}.
\end{align*}

Since $\frac{\frac{\sqrt{3}}{2}+\kappa}{\sqrt{3}\lambda-\kappa}\geq \frac{1-\lambda}{\lambda} $ for $\lambda\in[\frac{1}{2},1]$, $v^2$ is closer to $\tilde{q}^1$ than $w^2$ and therefore also $w^2\in S$.

The distance of $w^2$ to $[q^1,q^3]$ in the direction of $(-1,0)^T$ is by the intercept theorem
\begin{equation*}
    \begin{split}
        l_2&=\kappa\cdot \frac{\normeuc{v^2-w^2}}{\normeuc{\tilde{q}^1-v^2}}
        = \kappa \cdot \frac{\normeuc{\tilde{q}^3-v^2}-\normeuc{\tilde{q}^3-w^2}}{\normeuc{\tilde{q}^1-v^2}} 
    \end{split}
  \end{equation*}

  Computation of $l_3$ (\cf~\Cref{proofleftsidelabel_l23}): If $w^3$ is also contained in $S$, we have that the complete inner triangle of $S_\lambda$ is contained in $S$ and we are done.

  Otherwise, $\normeuc{\tilde{q}^1-v^3}\leq \normeuc{\tilde{q}^1-w^3} $ and we need to show that we can translate the inner triangle to be in $S$. To do so, we compute the distance $l_3$ of $w^3$ to $[q^1,q^2]$ in the direction of $(-1,0)^T$, which can be done completely analogously to $l_2$: 
  Let $ \alpha_3=\angle q^1q^2p^3$, $\beta_3=\angle v^3q^1p^2$ and $\gamma_3=\angle \tilde{q}^1v^3q^1$. Then,
  \begin{equation*}
    \frac{\normeuc{\tilde{q}^2-v^3}}{\normeuc{\tilde{q}^1-v^3}} = \frac{\sin{(\gamma_3)}}{\kappa\sin{(\beta_3)}}\cdot \frac{\kappa\sin{(\alpha_3)}}{\sin{(\gamma_3)}} = \frac{\sin{(\alpha_3)}}{\sin{(\pi-\beta_3)}} = \frac{\frac{\sqrt{3}}{2}-\kappa}{\sqrt{3}\lambda+\kappa}.
  \end{equation*}
  and $\frac{\normeuc{\tilde{q}^2-w^3}}{\normeuc{\tilde{q}^1-w^3}}=\frac{1-\lambda}{\lambda}$.
  Thus,
\begin{equation*}
  \frac{\normeuc{\tilde{q}^2-v^3}}{\normeuc{\tilde{q}^1-v^3}} \leq \frac{\normeuc{\tilde{q}^3-v^2}}{\normeuc{\tilde{q}^1-v^2}}.
\end{equation*}
Together with $\normeuc{\tilde{q}^2-\tilde{q}^1}=\normeuc{\tilde{q}^3-\tilde{q}^1}$ we obtain
$\normeuc{\tilde{q}^1-v^2}\leq\normeuc{\tilde{q}^1-v^3}$ and $\normeuc{\tilde{q}^2-v^3}\leq\normeuc{\tilde{q}^3-v^2}$
Hence,
  \begin{align*}
    l_3 &= \kappa\cdot \frac{\normeuc{v^3-w^3}}{\normeuc{\tilde{q}^1-v^3}}
    =\kappa \cdot \left(\frac{\normeuc{\tilde{q}^2-v^3}-\normeuc{\tilde{q}^2-w^3}}{\normeuc{\tilde{q}^1-v^3}} \right) \\
    &\leq \kappa \cdot \left(\frac{\normeuc{\tilde{q}^3-v^2}-\normeuc{\tilde{q}^3-w^2}}{\normeuc{\tilde{q}^1-v^2}} \right) \\
    &= l_2.
  \end{align*}

Moreover, since $\kappa \geq 0$ and $\lambda\in[\frac{1}{2},1]$ it follows

\begin{align*}
  l_3 &= \kappa \cdot \left(\frac{\normeuc{\tilde{q}^2-v^3}-\normeuc{\tilde{q}^2-w^3}}{\normeuc{\tilde{q}^1-v^3}} \right) \\
  &\leq \kappa \cdot \left(\frac{\normeuc{\tilde{q}^2-v^3}}{\normeuc{\tilde{q}^1-v^3}}-\frac{\normeuc{\tilde{q}^2-w^3}}{\normeuc{\tilde{q}^1-w^3}} \right) \\
  &= \kappa \cdot\left( \frac{\frac{\sqrt{3}}{2}-\kappa}{\sqrt{3}\lambda+\kappa} - \frac{1-\lambda}{\lambda}\right) \\
  &\leq \kappa \cdot \left( \frac{1}{2\lambda} - \frac{1-\lambda}{\lambda} \right) = \kappa \cdot \left(1 - \frac{1}{2\lambda} \right)\\
  &\leq \frac{1}{2}\cdot \kappa = l_1.
\end{align*}
 Hence, if we translate the inner triangle of $S_\lambda$ by $(l_3,0)^T$, it is contained in $S$, which proves $r(S_{\lambda},T)\leq r(S,T)$.
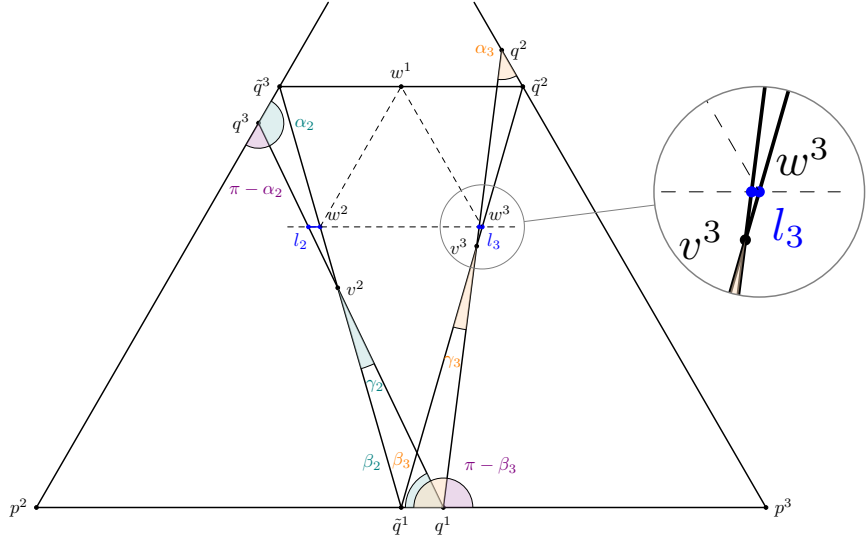
\begin{figure}[h]
  \centering
  \scalebox{0.7}{
  \begin{tikzpicture}[scale=8, spy using outlines={circle, magnification=2.5, size=3cm, connect spies}]
    \tkzDefPoint(-0.86602,-0.5){p2}
    \tkzDrawPoint[radius=0.4pt,label=left:$p^2$](p2)
    \tkzDefPoint(0.86602,-0.5){p3}
    \tkzDrawPoint[radius=0.4pt,label=right:$p^3$](p3)
    \tkzDrawSegment[black, thick](p2,p3);
    \tkzDefPoint(0,-0.5){q1t}
    \tkzDrawPoint[radius=0.4pt,label=below:$\tilde{q}^1$](q1t)
    \tkzDefPoint(0.28867,0.5){q2t}
    \tkzDrawPoint[radius=0.4pt,label=right:$\tilde{q}^2$](q2t)
    \tkzDefPoint(-0.28867,0.5){q3t}
    \tkzDrawPoint[radius=0.4pt,label=left:$\tilde{q}^3$](q3t)
    \tkzDrawPolygon[black, thick](q1t,q2t,q3t);
    \tkzDefPoint(0.1,-0.5){q1}
    \tkzDrawPoint[radius=0.4pt,label=below:$q^1$](q1)
    \tkzDefPoint(0.23867,0.5866){q2}
    \tkzDrawPoint[radius=0.4pt,label=right:$q^2$](q2)
    \tkzDefPoint(-0.33867,0.4134){q3}
    \tkzDrawPoint[radius=0.4pt,label=left:$q^3$](q3)
    \tkzDrawPolySeg[black, thick](q3,q1,q2);
    \tkzDefPoint(0,0.5){w1}
    \tkzDrawPoint[radius=0.4pt,label=above:$w^1$](w1)
    \tkzDefPoint(-0.19245,0.166666){w2}
    \tkzDrawPoint[radius=0.4pt,color=blue](w2)
    \tkzLabelPoint[above right](w2){$w^2$}
    \tkzDefPoint(0.19245,0.166666){w3}
    \tkzDrawPoint[radius=0.05pt,color=blue](w3)
    \tkzLabelPoint[above right](w3){$w^3$}
    \tkzDrawPolySeg[dashed](w2,w1,w3)
    \tkzDrawLine[black,dashed](w2,w3);
    \tkzDrawLine[black,thick,add=0 and 0.2](p2,q3t);
    \tkzDrawLine[black,thick,add=0 and 0.2](p3,q2t);
    \tkzInterLL(w2,w3)(q1,q2)
    \tkzGetPoint{f2}
    \tkzDrawPoint[radius=0.05pt,color=blue](f2);
    \tkzInterLL(w2,w3)(q3,q1)
    \tkzGetPoint{f3}
    \tkzDrawPoint[radius=0.4pt,color=blue](f3);
    \tkzDrawSegment[blue,thick](f2,w3)
    \tkzLabelSegment[below right,color=blue](f2,w3){$l_3$};
    \tkzDrawSegment[blue,thick](f3,w2)
    \tkzLabelSegment[below left,color=blue](f3,w2){$l_2$};
    \tkzInterLL(q1,q2)(q1t,q2t)
    \tkzGetPoint{v3}
    \tkzDrawPoint[radius=0.4pt,label=left:$v^3$](v3);
    \tkzInterLL(q1,q3)(q1t,q3t)
    \tkzGetPoint{v2}
    \tkzDrawPoint[radius=0.4pt,label=right:$v^2$](v2);
    \tkzFillAngle[fill=teal!25, opacity=0.5,size=0.2](q1t,v2,q1)
    \tkzLabelAngle[pos=.25,color=teal](q1t,v2,q1){$\gamma_2$}
    \tkzMarkAngle[size=0.2,mark=none](q1t,v2,q1);
    \tkzFillAngle[fill=teal!25, opacity=0.5,size=0.09](v2,q1,p2)
    \tkzLabelAngle[pos=.2,color=teal](v2,q1,p2){$\beta_2$}
    \tkzMarkAngle[size=0.09,mark=none](v2,q1,p2)

    \tkzFillAngle[fill=teal!25, opacity=0.5,size=0.06](v2,q3,q3t)
    \tkzLabelAngle[pos=.11,color=teal](v2,q3,q3t){$\alpha_2$}
    \tkzMarkAngle[size=0.06,mark=none](v2,q3,q3t);

    \tkzFillAngle[fill=violet!50, opacity=0.3,size=0.06](p2,q3,q1)
    \tkzLabelAngle[pos=.16,color=violet](p2,q3,q1){$\pi-\alpha_2$}
    \tkzMarkAngle[size=0.06,mark=none](p2,q3,q1);
    \tkzFillAngle[fill=orange!25, opacity=0.5,size=0.2](q1t,v3,q1)
    \tkzLabelAngle[pos=.28,color=orange](q1t,v3,q1){$\gamma_3$}
    \tkzMarkAngle[size=0.2,mark=none](q1t,v3,q1);

    \tkzFillAngle[fill=orange!25, opacity=0.5,size=0.07](v3,q1,q1t)
    \tkzLabelAngle[pos=.15,color=orange](v3,q1,q1t){$\beta_3$}
    \tkzMarkAngle[size=0.07,mark=none](v3,q1,q1t);

    \tkzFillAngle[fill=violet!50, opacity=0.3,size=0.07](p3,q1,q2)
    \tkzLabelAngle[pos=.15,color=violet](p3,q1,q2){$\pi-\beta_3$}
    \tkzMarkAngle[size=0.07,mark=none](p3,q1,q2);

    \tkzFillAngle[fill=orange!25, opacity=0.5,size=0.07](q1,q2,p3)
    \tkzLabelPoint[left,color=orange](q2){$\alpha_3$}
    \tkzMarkAngle[size=0.07,mark=none](q1,q2,p3);
    \spy [gray , size=4cm] on (w3)
    in node [right] at (0.6,0.25);
  \end{tikzpicture} }
  \vspace*{-1.5cm}
  \caption{Proof of \Cref{leftMaxT}. Computation of $l_2$ and $l_3$. }
  \label{proofleftsidelabel_l23}
\end{figure}
\end{proof}

The collected inequalities are now sufficient to provide a full description of the Blaschke-Santaló-diagram $f_{\MAX}(C^2,T)$ (\cf~\Cref{figdiagramTmaxwithineqs}).

\begin{theorem}
    \label{fulldiamaxT}
    For every Minkowski-centered triangle $S$ the diagram $f_{\MAX}(\bar{\CC}^2,S)$ is fully described by the inequalities
    \begin{align*}
        \frac{D_{\MAX}(K,S)}{2}&\leq R(K,S)\\
        r(K,S)&\leq \frac{D_{\MAX}(K,S)}{2}\\
        0&\leq r(K,S)\\
        R(K,S)&\leq D_{\MAX}(K,S)\\
        \left(\frac{D_{\MAX}(K,S)}{R(K,S)}-\frac{1}{2}\right)&\left(\frac{3}{2}-\frac{D_{\MAX}(K,S)}{R(K,S)}\right)\leq \frac{r(K,S)}{R(K,S)}.
    \end{align*}
  \end{theorem}

\begin{proof}
    Since all Minkowski-centered triangles can be linearly transformed into the equilateral triangle $T$ it suffices to show the claim for $T$.  We give a continuous description of the boundaries described by the inequalities \eqref{upMaxT}, \eqref{rightMaxT}, and \eqref{eqrRzero}, as well as those given by \Cref{thm:jungbound_Dmax_mink} and \Cref{leftMaxT}.
     First, \eqref{upMaxT} is attained with equality for $(1-\lambda)L_D+\lambda T$, $\lambda\in[0,1]$ where $L_D$ and $T$ are the extreme cases. Second, \eqref{rightMaxT} is attained with equality for $(1-\lambda)T+\lambda T_{\MAX}$, $\lambda\in[0,1]$ by \Cref{linearityrrRD} where $T$ and $T_{\MAX}$ are the extreme cases. The boundary induced by \eqref{eqrRzero} is filled by segments from $L_w$ to $L_D$.
     Next, since $T_{\MAX}$ is the completion of $-T$, the boundary  from \Cref{thm:jungbound_Dmax_mink} is filled by the sets $-T_+$ with $-T \subset -T_+ \subset T_{\MAX}$. 
     Finally, the inequality in \Cref{leftMaxT} is fulfilled with equality by the triangles $S_{\lambda}$ as introduced in \Cref{leftMaxTequalcase}.
     Since we have presented a continuous description of the boundary we can apply \Cref{simpleconn} and follow that the diagram is simply connected.
\end{proof}

 While the inequalities \eqref{upMaxT}, \eqref{rightMaxT}, \eqref{eqrRzero} and the one given by \Cref{thm:jungbound_Dmax_mink} are valid for all choices of planar, Minkowski-centered gauges, the inequality from \Cref{leftMaxT} is only proven for triangles. Using the result from \Cref{prop:diagramAM} and $D_{\AM}(K,C)\leq \frac{4}{3}D_{\MAX}(K,C)$ which follows from $\frac{s(C)+1}{2s(C)}C_{\MAX}\optc C_{\AM}$ \cite{reversing}, we are able to give another general inequality
 \begin{equation}
   \label{eq:upperboundunion}
   \frac{2}{3}\frac{D_{\MAX}(K,C)}{R(K,C)}\left(1-\frac{2}{3}\frac{D_{\MAX}(K,C)}{R(K,C)}\right)\leq \frac{r(K,C)}{R(K,C)},
 \end{equation}
  which enables us to give a bound for the union of the diagrams $f_{\MAX}(\bar{\CC}^2,C)$ with $C$ Minkowski-centered (depicted in red within \Cref{figdiagramTmaxwithineqs}).
\begin{conjecture}
  The diagram of any triangle is dominating, i.e.~for every Minkowski-centered $C$ we have $f_{\MAX}(\bar{\CC}^2,C) \subset f_{\MAX}(\bar{\CC}^2,S)$.
\end{conjecture}

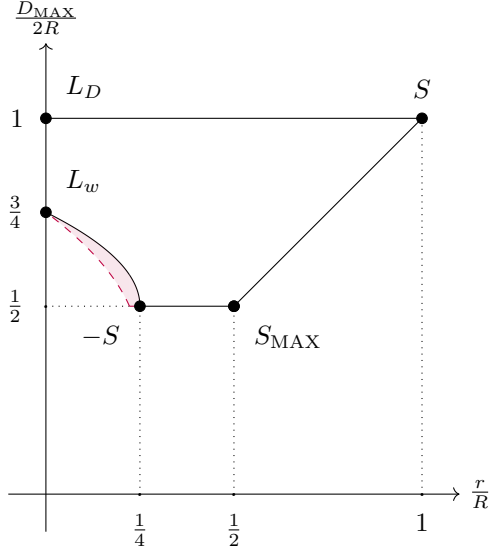
\begin{figure}[ht]
    \centering
     \begin{tikzpicture}[scale=5]

     \draw[->] (-0.1,0) -- (1.1,0) node[right] {$\frac{r}{R}$};
     \draw[->] (0,-0.1) -- (0,1.2) node[above] {$\frac{D_{\MAX}}{2R}$}
     ;

     \draw[ domain=0:1, smooth, variable=\x, color=black] plot ({\x}, {1})
     ;
     \draw[ domain=0.5:1, smooth, variable=\x] plot ({\x}, {\x})
     ;
     \draw[ domain=0.25:0.5, smooth, variable=\x] plot ({\x}, {0.5})
     ;
     \draw[name path=F1, domain=0.5:7/12, smooth, variable=\y] plot ({(2*\y-0.5)*(1.5-2*\y)},{\y} ) 
     ;
     \draw[name path=F2, domain=7/12:0.75, smooth, variable=\y] plot ({(2*\y-0.5)*(1.5-2*\y)},{\y} ) 
     ;
     \draw[ domain=0:2/9, dotted, variable=\x] plot ({\x}, {0.5})
     ;
      \node[label=left:$\frac{1}{2}$] (jung) at (0,0.5) {.};
       \draw[ domain=0:1, dotted, variable=\y] plot ({1}, {\y})
     ;
      \node[label=below:$1$] (one) at (1,0) {.};
     \draw[ domain=0:0.5, dotted, variable=\y] plot ({0.25}, {\y})
     ;
      \node[label=below:$\frac{1}{4}$] (minus) at (0.25,0) {.};
      \draw[ domain=0:0.5, dotted, variable=\y] plot ({0.5}, {\y})
     ;
      \node[label=below:$\frac{1}{2}$] (max) at (0.5,0) {.};
      \node[label=left:$1$] (delta) at (0,1) {.};
      \node[label=left:$\frac{3}{4}$] (rho) at (0,0.75) {.};

     \node[label=above:$S$] ($T$) at (1,1) {.};
     \draw[fill=black] (1,1) circle[radius=0.4pt];

     \node[label=above right:$L_D$] ($L_D$) at (0,1) {.};
     \draw[fill=black] (0,1) circle[radius=0.4pt];

     \node[label=above right:$L_w$] (Lw) at (0,0.75) {.};

     \node[label=below right:$S_{\MAX}$] (Tmax) at (0.5,0.5) {.};
     \draw[fill=black] (0.5,0.5) circle[radius=0.4pt];

     \node[label=below left:$-S$] (T) at (0.25,0.5) {.};

     \draw[ name path=F3,domain=2/9:0.25, dashed, smooth, variable=\x, color=purple] plot ({\x}, {0.5});
     \draw[ name path=F4, domain=0.5:0.75, dashed, smooth, variable=\y, color=purple] plot ({4/3*\y*(1-4/3*\y)}, {\y});
    \tikzfillbetween[of=F1 and F3]{purple, opacity=0.1};
     \tikzfillbetween[of=F2 and F4]{purple, opacity=0.1};
     \draw[fill=black] (0.5,0.5) circle[radius=0.4pt];
     \draw[fill=black] (0,0.75) circle[radius=0.4pt];
     \draw[fill=black] (0.25,0.5) circle[radius=0.4pt];
 \end{tikzpicture}
   \caption{The diagram $f_{\MAX}(\bar{\CC}^2,S)$ \wrt~a Minkowski-centered triangle $S$ (black) and an upper bound for the union over all Minkowski-centered gauges bounded by \eqref{eq:upperboundunion} (additional purple region).}
  \label{figdiagramTmaxwithineqs}
\end{figure}

\section{The diameter $D_{\HM}$}
In the case of the harmonic mean, the factors $\rho_{\HM}$ and $\delta_{\HM}$ are not bound to the Minkowski asymmetry $s(C)$. However, in \cite{reversing} the following bounds are proven:
\begin{proposition}
    \begin{equation*}
        1 \leq \rho_{\HM} \leq \frac{(s(C)+1)^2}{4s(C)} \leq \delta_{\HM} \leq \frac{s(C)+1}{2}.
     \end{equation*}
\end{proposition}
Taking $K=C_{\HM}$, we know $R(C_{\HM},C)=\frac{2s(C)}{s(C)+1}$, $r(C_{\HM},C)=\frac{2}{s(C)+1}$ and $D_{\HM}(C_{\HM},C)=2$ \cite{reversing}.
Furthermore, the inequalities from \Cref{lem:basicineq} have the form:
\begin{align} \label{eq:HMdeltaDR}
    \frac{D_{\HM}(K,C)}{2} & \leq \delta_{\HM}R(K,C),
\intertext{}  \label{eq:HMdeltaDr}
  \delta_{\HM}r(K,C) & \leq \frac{D_{\HM}(K,C)}{2},
\intertext{} \label{rRDfromAM}
     s(C)r(K,C)+R(K,C) &\leq \frac{s(C)+1}{2\rho_{\HM}} D_{\HM}(K,C)\leq \frac{s(C)+1}{2} D_{\HM}(K,C),
     \intertext{} \label{rRDfromAM2}
     r(K,C)+R(K,C)&\leq \frac{2s(C)}{s(C)+1}D_{\HM}(K,C)
     \intertext{and} \label{eq:zerohm}
     0 & \leq r(K,C)
\end{align}

  Thus, if $\delta_{\HM}=\frac{s(C)+1}{2}$ then $C_{\HM}$ fulfills \eqref{eq:HMdeltaDr} with equality and if  $\rho_{\HM}=\frac{(s(C)+1)^2}{4s(C)}$ then $C_{\HM}$ fulfills the first inequality in \eqref{rRDfromAM} with equality, which
  in this case can be rewritten as
  \begin{equation}
    \label{rRDfromAM3}
    s(C)r(K,C)+R(K,C)\leq \frac{2s(C)}{s(C)+1}D_{\HM}(K,C).
  \end{equation}
Unlike with $D_{\MAX}$, (a dilatate of) the symmetrization $C_{\HM}$ is not always a completion of the gauge.
\begin{lemma}
  \label{omegaequivalences}
  The following are equivalent:
  \begin{enumerate}[i)]
  \item $\delta_{\HM}=\frac{s(C)+1}{2}$,
  \item $\frac{s(C)+1}{2}C_{\HM}$ is a completion of $C$ \wrt~$C$, and
  \item $D(C,C_{\HM})=2R(C,C_{\HM})$.
  \end{enumerate}
\end{lemma}

\begin{proof}
We know from \cite{reversing} that $C\optc \frac{s(C)+1}{2}C_{\HM}$ and therefore $\frac{s(C)+1}{2}C_{\HM}$ is a complete set containing $C$ with $R(C,C_{\HM}) = \frac{s(C)+1}{2}$.

\begin{enumerate}[align=left]
\item[$i)\Rightarrow ii)$:]
  If $\delta_{\HM}=\frac{s(C)+1}{2}$ then
  $D_{\HM}\left(\frac{s(C)+1}{2}C_{\HM},C\right)=s(C)+1=2\delta_{\HM} = D_{\HM}(C,C)$,
  implying that $\frac{s(C)+1}{2}C_{\HM}$ is a completion of $C$.
    \item[$ii)\Rightarrow iii)$:] If $\frac{s(C)+1}{2}C_{\HM}$ is a completion of $C$, $D(C,C_{\HM})=D_{\HM}(C,C)=s(C)+1=2R(C,C_{\HM})$.
    \item[$iii)\Rightarrow i)$:] Assuming (iii) it follows $2\delta_{\HM}=2R(C_{\AM},C_{\HM})=D(C,C_{\HM})=2R(C,C_{\HM})=s(C)+1$.
\end{enumerate}
\end{proof}

If $\frac{s(C)+1}{2}C_{\HM}$ is a completion of the gauge $C$, then it is also a completion of $-C$ and since $R(-C,C)=s(C)=R(\frac{s(C)+1}{2}C_{\HM},C)$ for $-C$ it is even a Scott-completion.

\begin{example}
    The Reuleaux triangle $\rt$ is the completion of the equilateral triangle $T$ in the euclidean case:
\begin{equation*}
    \rt:= \bigcap_{i=1}^3 p^i + \sqrt{3}\B^2_2,
\end{equation*}
where the $p_i$ are the vertices of $T$. For the Reuleaux triangle one obtains (omitting the detailed calculations) $s(\rt)=\frac{1}{\sqrt{3}-1}\approx 1.366$, $\delta_{\HM}=\frac{\sqrt{3}}{\sqrt{11}-\sqrt{3}} \approx 1.093 <\frac{s(\rt)+1}{2}$ and $\rho_{\HM}=\frac{(s(\rt)+1)^2}{4s(\rt)}=\frac{3(\sqrt{3}+1)}{8}\approx 1.025$.
Thus, by \Cref{omegaequivalences} this is a case where the (dilatated) harmonic mean $\rt_{\HM}$ is not a completion of the gauge $\rt$.

Since $T\subset \rt$ is diametric we obtain from \Cref{diamtriagle2dim} that the unique completion of $\rt$ is
\begin{equation*}
  \rt^{*}:=\bigcap_{i=1}^3 p_i + 2\delta_{\HM}\rt_{\HM} \quad \text{(\cf~\Cref{reultrianglehm})}.
\end{equation*}

\begin{figure}[ht]
  \centering

  \begin{minipage}[c]{.4\linewidth}
    \includegraphics[scale=0.3]{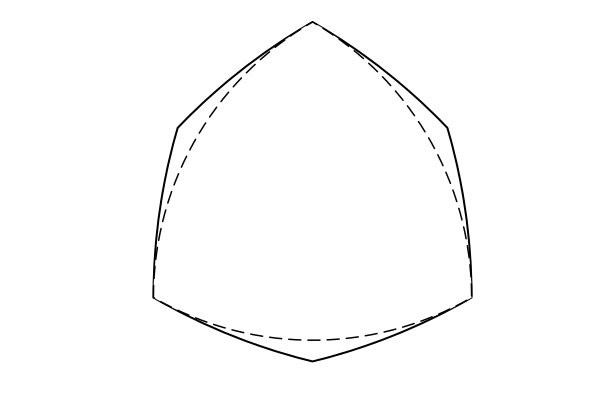}
  \end{minipage}
  \noindent
  \begin{minipage}[c]{.4\linewidth}
    \includegraphics[scale=0.3]{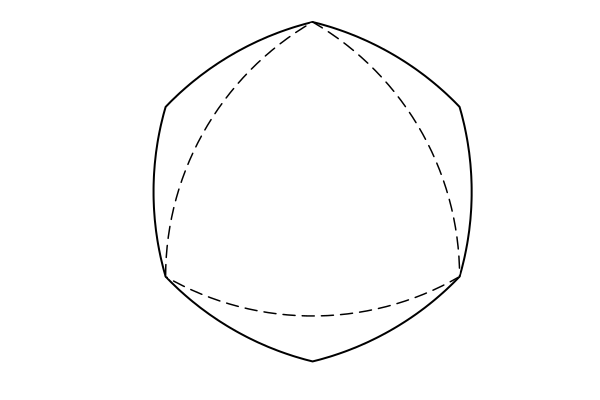}
  \end{minipage}

  \caption{Reuleaux Triangle $\rt$ (dashed), its completion $\rt^{*}$ (left) and its harmonic symmetrization $\rt_{\HM}$ (right).}
  \label{reultrianglehm}
\end{figure}

Using the symmetries of the Reuleaux triangle we obtain $R(\rt^{*},\rt)=(s(\rt)+1)\cdot\frac{\delta_{\HM}}{\rho_{\HM}}-s(\rt)$ and thus $\rt^{*}$ fulfills \eqref{rRDfromAM3} with equality, which is also true for $\rt_{\HM}$.

The diagram $f_{\HM}(\bar{\CC}^2,\rt)$ is given in \Cref{diagramreul}. In addition to known inequalities \eqref{eq:HMdeltaDR}, \eqref{eq:HMdeltaDr}, \eqref{rRDfromAM}, \eqref{eq:zerohm} and points, two conjectured inequalities are shown.

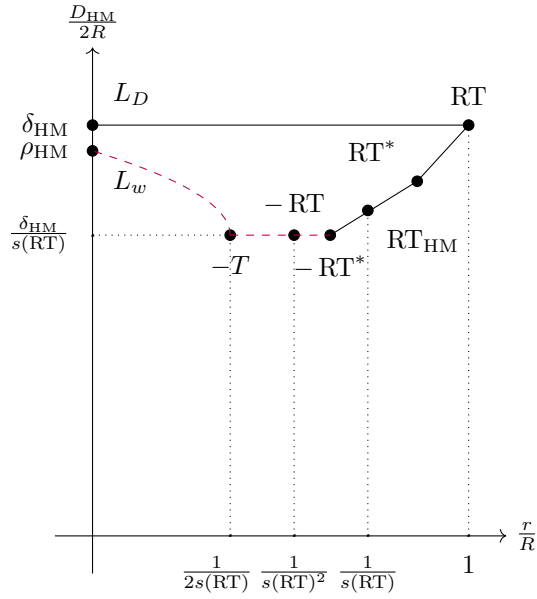
\begin{figure}[ht]
  \centering
  \begin{tikzpicture}[scale=5]
    \def\sC{1/(sqrt(3)-1)}
    \def\deltaC{(sqrt(11)-sqrt(3))/(sqrt(3))}
    \def\rhoC{0.976067}
    \def\f{(\deltaC)/((\sC+1)*(\rhoC)-(\deltaC)*(\sC))}
    \def\a{(-8*\sC)/(4/(\deltaC)-(\sC+1)^2)^2}
    \draw[->] (-0.1,0) -- (1.1,0) node[right] {$\frac{r}{R}$};
    \draw[->] (0,-0.1) -- (0,1.3) node[above] {$\frac{D_{\HM}}{2R}$};
    \draw[ domain=0:1, smooth, variable=\x] plot ({\x}, {1/(\deltaC)});
    \draw[ domain={\f}:1, smooth, variable=\x] plot ({\x}, {(1/(\deltaC))*(\x)});
    \node[label=above:$ \rt$] (rt) at (1,{1/(\deltaC)}) {.};
    \draw[fill=black] (1,{1/(\deltaC)}) circle[radius= 0.4pt];

    \node[label=above right:$L_D$] ($L_D$) at (0,{1/(\deltaC)}) {.};
    \draw[fill=black] (0,{1/(\deltaC)}) circle[radius= 0.4pt];

    \node[label=below right:$L_w$] ($L_w$) at (0,{1/(\rhoC)}) {.};
    \draw[fill=black] (0,{1/(\rhoC)}) circle[radius= 0.4pt];

    \node[label= above left:$ \rt^{*}$] (Tco) at ({\f},{(\f)/(\deltaC)}) {.};
    \draw[fill=black] ({\f},{(\f)/(\deltaC)}) circle[radius= 0.4pt];

    \node[label=below right:$ \rt_{\HM}$] (Thm) at ({1/(\sC)},{(\sC+1)/(2*(\sC))}) {};
    \draw[fill=black] ({1/(\sC)},{(\sC+1)/(2*(\sC))}) circle[radius= 0.4pt];

    \node[label=below:$-T$] (T) at ({1/(2*(\sC))},{1/((\deltaC)*(\sC))}) {.};
    \draw[fill=black] ({1/(2*(\sC))},{1/((\deltaC)*(\sC))}) circle[radius= 0.4pt];
    \node[label=above:$- \rt$] (RTminus) at ({1/((\sC)*(\sC))},{1/((\deltaC)*(\sC))}) {.};
    \draw[fill=black] ({1/((\sC)*(\sC))},{1/((\deltaC)*(\sC))}) circle[radius= 0.4pt];
    \node[label=below:$- \rt^{*}$] (Tcon)  at ({(\f)/(\sC)},{1/((\deltaC)*(\sC))}) {.};
    \draw[fill=black] ({(\f)/(\sC)},{1/((\deltaC)*(\sC))}) circle[radius= 0.4pt];

    \draw plot[smooth] ({(\f)/(\sC)},{1/((\deltaC)*(\sC))}) -- ({\f},{(\f)/(\deltaC)}) ;
    \draw[ domain={1/(2*(\sC))}:{(\f)/(\sC)},purple, dashed, variable=\x] plot ({\x}, {1/((\deltaC)*(\sC))}) ;
    \draw[ domain=0:{1/(2*(\sC))}, dotted, variable=\x] plot ({\x}, {1/((\deltaC)*(\sC))}) ;
    \node[label=left:$\frac{\delta_{\HM}}{s(\rt)}$] (jung) at (0,{1/((\deltaC)*(\sC))}) {.}
    ;
    \draw[ domain=0:{1/(\deltaC)}, dotted, variable=\y] plot ({1}, {\y}) ;
    \node[label=below:$1$] (one) at (1,0) {.}
    ;
     \draw[ domain=0:{1/((\deltaC)*(\sC))}, dotted, variable=\y] plot ({1/(2*(\sC))}, {\y}) ;
    \node[label=below:$\frac{1}{2s(\rt)}\quad$] (third) at ({1/(2*(\sC))},0) {.}
    ;
    \draw[ domain=0:{(\sC+1)/(2*(\sC))}, dotted, variable=\y] plot ({1/(\sC)}, {\y}) ;
    \node[label=below:$\frac{1}{s(\rt)}$] (hm) at ({1/(\sC)},0) {.};
    \draw[ domain=0:{1/((\deltaC)*(\sC))}, dotted, variable=\y] plot ({1/((\sC)*(\sC))}, {\y}) ;
    \draw[ domain={1/((\deltaC)*(\sC))}:{1/(\rhoC)},purple, dashed, variable=\y] plot ({\a*(\y-1/(\rhoC))*(\y-(2/((\deltaC)*(\sC)) -1/(\rhoC)))}, {\y}) ;
    \node[label=below:$\frac{1}{s(\rt)^2}$] (hm) at ({1/((\sC)*(\sC))},0) {.};
     \node[label=left:$\delta_{\HM}$] (delta) at (0,{1/(\deltaC)}) {.};
     \node[label=left:$\rho_{\HM}$] (rho) at (0,{1/(\rhoC)}) {.};
   \end{tikzpicture}
   \caption{The Blaschke-Santaló diagram $f_{\HM}(\bar{\CC}^2, \rt)$ with conjectured inequalities (purple). Keep in mind that the (dilatated) harmonic symmetrization of the Reuleaux triangle is not its completion.}
   \label{diagramreul}
 \end{figure}
\end{example}

For the equilateral triangle, we have $\delta_{\HM}=\frac{s(T)+1}{2}=\frac{3}{2}$, $\rho_{\HM}=\frac{(s(T)+1)^2}{4s(T)}=\frac{9}{8}$ and $T_{\HM}=\frac{2}{3}T_{\MAX}$ \cite{BDG}. Thus, $D_{\HM}(K,T)=\frac{3}{2}D_{\MAX}(K,T)$ for all $K\in\CC^2$ and we can transfer all inequalities concerning triangles from the previous section.

\begin{corollary}
  \label{cor:diagramHM}
  For every Minkowski-centered triangle $S$ the diagram $f_{\HM}(\bar{\CC}^2,S)$ is fully described by the inequalities

  \begin{align*}
    D_{\HM}(K,S)&\leq 3R(K,S)\\
    3r(K,S)&\leq D_{\HM}(K,S)\\
    0&\leq r(K,S)\\
    R(K,S)&\leq \frac{2}{3}D_{\HM}(K,S)\\
    \frac{4}{9}\left(\frac{D_{\HM}(K,S)}{R(K,S)}-\frac{3}{4}\right)&\left(\frac{9}{4}-\frac{D_{\HM}(K,S)}{R(K,S)}\right)\leq \frac{r(K,S)}{R(K,S)}
  \end{align*}
(\cf~\Cref{fig:diagramHM}).
\end{corollary}

\begin{remark}
  \label{jungineqDhm}
  The bound $D_{\HM}(K,C)\geq R(K,C)$ derived from \Cref{cor:jung-MC-all} for the planar case cannot be reached in case of the harmonic diameter. If $D_{\HM}(K,C)=2$, we obtain from the containment chain
  \begin{equation*}
    \frac{1}{2}\left(1+\frac{1}{s(K)}\right) K\subset \frac{K-K}{2} \subset C_{\HM}\subset \frac{2s(C)}{s(C)+1}C,
   \end{equation*}
  that \begin{equation}
    \label{eq:hmjung}
    R(K,C)\leq \frac{4s(K)s(C)}{(s(K)+1)(s(C)+1)}\leq \frac{4n^2}{(n+1)^2}.
   \end{equation}
 Thus,
 \begin{equation*}
   D_{\HM}(K,C)\geq \frac{(n+1)^2}{2n^2} R(K,C),
 \end{equation*} which for $n=2$ gives $D_{\HM}(K,C)\geq \frac{9}{8} R(K,C)>R(K,C)$.
  Furthermore, equailty in \eqref{eq:hmjung} can only be reached if $K$ and $C$ are simplices, but we have shown that for Minkowski-centered triangles $D_{\HM}(K,S)\geq \frac{3}{2} R(K,S)$. Finally, one may recognize that if we consider 0-symmetric planar gauges, \eqref{eq:hmjung} also yields $D_{\HM}(K,C)\geq \frac{3}{2} R(K,C)$.
\end{remark}

  The following system of inequalities provides an upper bound for the union of the diagrams $f_{\HM}(\bar{\CC}^2,C)$ over all Minkowski-centered gauges $C\in\CC^2_0$ (\cf~\Cref{fig:diagramHM}).
  \begin{align*}
      0&\leq r(K,C)\\
      r(K,C)&\leq R(K,C)\\
      D_{\HM}(K,C)&\leq 3 R(K,C)\\
      2r(K,C)&\leq D_{\HM}(K,C)\\
       9R(K,C) &\leq 8D_{\HM}(K,C)\\
       \frac{D_{\HM}(K,C)}{2R(K,C)}\left(1-\frac{D_{\HM}(K,C)}{2R(K,C)}\right) &\leq \frac{r(K,C)}{R(K,C)}.
  \end{align*}
  Moreover, the following parts of the boundary described by the above inequalities are reached:
  \begin{enumerate}[i)]
  \item $r(K,C)=0$ for segments $K=L_D$ for gauges $C$ with $s(C)\in[1,2]$
  \item $r(K,C)= R(K,C)$ for $K=C$ for gauges $C$ with $s(C)\in[1,2]$
  \item $D_{\HM}(K,C)= 3 R(K,C)$ with $C$ being a triangle as in $f_{\HM}(\bar{\CC}^2,S)$
  \end{enumerate}
  The first two inequalities are trivial. The third and forth follow from \Cref{lem:basicineq} and the fifth from \Cref{jungineqDhm}.
  The last inequality follows from \Cref{prop:diagramAM} and \Cref{contandfactor}.
  The equality cases follow from \Cref{linearityrrRD}, \Cref{CDiam} and \Cref{cor:diagramHM}.

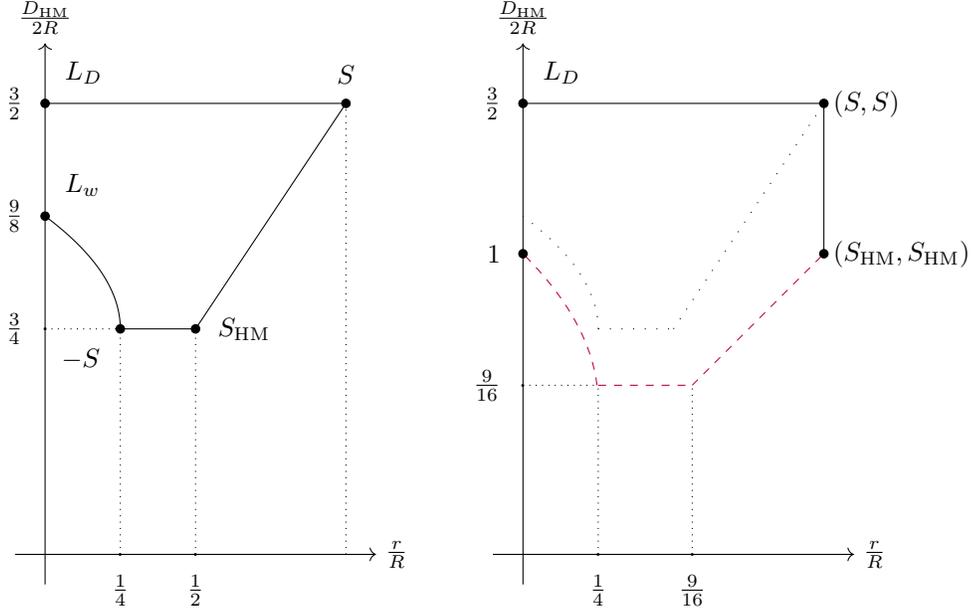
\begin{figure}[ht]
  \centering
  \noindent\begin{minipage}[c]{.4\linewidth}
    \begin{tikzpicture}[scale=4]

      \draw[->] (-0.1,0) -- (1.1,0) node[right] {$\frac{r}{R}$};
      \draw[->] (0,-0.1) -- (0,1.7) node[above] {$\frac{D_{\HM}}{2R}$}
      ;
      \draw[ domain=0:1, smooth, variable=\x] plot ({\x}, {1.5})
      ;
      \draw[ domain=0.5:1, smooth, variable=\x] plot ({\x}, {3/2*\x})
      ;
      \draw[ domain=0.25:0.5, smooth, variable=\x] plot ({\x}, {0.75})
      ;
      \draw[ domain=0.75:1.125, smooth, variable=\y] plot ({(4/3*\y-0.5)*(1.5-4/3*\y)},{\y} )
      ;
      \draw[ domain=0:0.25, dotted, variable=\x] plot ({\x}, {0.75})
      ;
       \node[label=left:$\frac{3}{4}$] (jung) at (0,0.75) {.};
        \draw[ domain=0:1.5, dotted, variable=\y] plot ({1}, {\y})
      ;
      \draw[ domain=0:0.75, dotted, variable=\y] plot ({0.25}, {\y})
      ;
       \node[label=below:$\frac{1}{4}$] (minus) at (0.25,0) {.};
       \draw[ domain=0:0.75, dotted, variable=\y] plot ({0.5}, {\y})
      ;
       \node[label=below:$\frac{1}{2}$] (hm) at (0.5,0) {.};
       \node[label=left:$\frac{3}{2}$] (delta) at (0,1.5) {.};
       \node[label=left:$\frac{9}{8}$] (rho) at (0,1.125) {.};

      \node[label=above:$S$] ($T$) at (1,1.5) {.};
      \draw[fill=black] (1,1.5) circle[radius=0.4pt];

      \node[label=above right:$L_D$] ($L_D$) at (0,1.5) {.};
      \draw[fill=black] (0,1.5) circle[radius=0.4pt];

      \node[label=above right:$L_w$] ($L_w$) at (0,1.125) {.};
      \draw[fill=black] (0,1.125) circle[radius=0.4pt];

      \node[label=right:$S_{\HM}$] (Thm) at (0.5,0.75) {.};
      \draw[fill=black] (0.5,0.75) circle[radius=0.4pt];

      \node[label=below left:$-S$] (Tminus) at (0.25,0.75) {.};
      \draw[fill=black] (0.25,0.75) circle[radius=0.4pt];
  \end{tikzpicture}
  \end{minipage}
  \noindent\begin{minipage}[c]{.4\linewidth}
    \begin{tikzpicture}[scale=4]

      \draw[->] (-0.1,0) -- (1.1,0) node[right] {$\frac{r}{R}$};
      \draw[->] (0,-0.1) -- (0,1.7) node[above] {$\frac{D_{\HM}}{2R}$}
      ;
      \draw[ domain=0:1, smooth, variable=\x] plot ({\x}, {1.5})
      ;
      \draw[ domain=0.5:1, loosely dotted, variable=\x] plot ({\x}, {3/2*\x})
      ;
      \draw[ domain=0.25:0.5, loosely dotted, variable=\x] plot ({\x}, {0.75})
      ;
      ;
      \draw[ domain=0.75:1.125,loosely dotted, variable=\y] plot ({(4/3*\y-0.5)*(1.5-4/3*\y)},{\y} )
      ;
      \draw[ domain=0.5625:1, dashed,purple, variable=\y] plot ({(\y)*(1-\y)},{\y} )
      ;
      \draw[ domain=1:1.5, smooth,black, variable=\y] plot ({1},{\y} )
      ;
      \draw[ domain=0.25:0.5625,dashed,purple, variable=\x] plot ({\x},{0.5625} )
      ;
     \draw[ domain=0.5625:1, smooth,dashed, purple,variable=\x] plot ({\x},{\x} )
      ;
      ;
      \node[label=left:$\frac{9}{16}$] (jung) at (0,0.5625) {.};
      ;
       \node[label=left:$1$] (one) at (0,1) {.};
      \draw[ domain=0:0.5625, dotted, variable=\y] plot ({0.25}, {\y})
      ;
      \draw[ domain=0:0.25, dotted, variable=\x] plot ({\x}, {0.5625})
      ;
       \node[label=below:$\frac{1}{4}$] (minus) at (0.25,0) {.};
       \draw[ domain=0:0.5625, dotted, variable=\y] plot ({0.5625}, {\y})
      ;
       \node[label=below:$\frac{9}{16}$] (minus) at (0.5625,0) {.};
      ;
       \node[label=left:$\frac{3}{2}$] (delta) at (0,1.5) {.};

      \node(T) at (1,1.5) {.};
      \tkzLabelPoint[right](T){$(S,S$)}
      \draw[fill=black] (1,1.5) circle[radius=0.4pt];
      \node(Ts) at (1,1) {.};
      \tkzLabelPoint[right](Ts){$(S_{\HM},S_{\HM}$)}
      \draw[fill=black] (1,1) circle[radius=0.4pt];

      \node[label=above right:$L_D$] ($L_D$) at (0,1.5) {.};
      \draw[fill=black] (0,1.5) circle[radius=0.4pt];
      \draw[fill=black] (0,1) circle[radius=0.4pt];



   \end{tikzpicture}
 \end{minipage}
 \caption{The diagram $f_{\HM}(C^2,S)$ \wrt~a Minkowski-centered triangle $S$ (left) and an upper bound for the union of the diagrams over all Minkowski-centered gauges $C\in\CC^2_0$ (right).}
 \label{fig:diagramHM}
\end{figure}

\begin{remark}
    Since $\delta_{\HM}$ is not the same for every gauge as it is in the arithmetic case, the diagram for the equilateral triangle cannot be dominating. For every $C$, $f_{\HM}(C,C)=(1,\delta_{\HM})$ and this is always the only combination where inradius and circumradius coincide. Thus, we cannot find a single gauge $C$ which defines the union of the diagrams. \end{remark}

\section{The diameter $D_{\MIN}$}
In the case of the minimum $\delta_{\MIN}$ depends solely on the asymmetry of $C$ but $\rho_{\MIN}$ can only be bounded in terms of the asymmetry.
 \cite{reversing}.

\begin{proposition} \label{prop:rhodeltamin}
   \begin{align*}
       1\leq &\rho_{\MIN}\leq \frac{s(C)+1}{2} = \delta_{\MIN}
   \end{align*}
\end{proposition}
Taking $K=C_{\MIN}$, we know $R(C_{\MIN},C)=1$, $r(C_{\MIN},C)=\frac{1}{s(C)}$ and $D_{\MIN}(C_{\MIN},C)=2$ \cite{reversing}.
The inequalities from \Cref{lem:basicineq} have the form:
\begin{equation}
  D_{\MIN}(K,C)\leq (s(C)+1)R(K,C),
\end{equation}

\begin{equation}
  (s(C)+1)r(K,C) \leq D_{\MIN}(K,C),
\end{equation}

\begin{equation}
  s(C)r(K,C)+R(K,C)\leq\rho_{\m}(s(C)r(K,C)+R(K,C)) \leq (s(C)+1)\frac{D_{\m}(K,C)}{2},
\end{equation}

\begin{equation}
 r(K,C)+R(K,C)\leq D_{\MIN}(K,C),
\end{equation}
\begin{equation}
  \label{eqrRzeromin}
  0 \leq r(K,C).
\end{equation}

Contrary to the other diameters the (dilatated) symmetrization only yields a completion of the gauge $C$ in the trivial case.

\begin{lemma}
  $\rho C_{\MIN}$ is  a completion of $C$ if and only if $\rho=s(C)=1$, which means $C$ is $0$-symmetric.
\end{lemma}

\begin{proof}
  We have
$C \optc s(C) C_{\MIN}$, which shows that $\rho$ must equal $s(C)$. However,
    \[D_{\MIN}(s(C)C_{\MIN},C)=2s(C) \geq s(C)+1=2\delta_{\MIN}=D_{\MIN}(C,C),\]
    with equality if and only if $s(C)=1$.
\end{proof}

Next, we prove a Jung-bound. Bohnenblust \cite{bohnenblust1938convex} shows for symmetric gauges
\begin{proposition}
  \label{prop:bohnenblust}
  \begin{equation}
    \frac{n+1}{2n}\leq\frac{D(K,C)}{2R(K,C)}.
  \end{equation}
\end{proposition}
For the minimum diameter we obtain the same bound.
\begin{corollary}
  \label{jungineqDmin}
  \begin{equation}
    \frac{n+1}{2n}\leq\frac{D_{\MIN}(K,C)}{2R(K,C)}
  \end{equation}
  Moreover, equality can be obtained only if $K$ is a simplex. For every $s_C \in [1, n]$, there exists $C\in\CC^n_0$, Minkowski-centered, with $s(C) = s_C$ and
a simplex $K$ such that the inequality is tight.
\end{corollary}

\begin{proof}
 Using \Cref{prop:bohnenblust} for the gauge $C_{\MIN}$ and \Cref{contandfactor}, it follows
\begin{equation}
  \frac{n+1}{2n}\leq\frac{D(K,C_{\MIN})}{2R(K,C_{\MIN})} \leq \frac{D(K,C_{\MIN})}{2R(K,C)}= \frac{D_{\MIN}(K,C)}{2R(K,C)}.
\end{equation}
  By \cite[Theorem 4.1]{sharpening} equailty in Bohnenblust's inequality can only be attained if $K$ is a simplex.
   Let $S$ be a Minkowski-centered simplex. Then, $C=S\cap s_C(-S)$ with $s_C\in[1,n]$ has Minkowski asymmetry $s(C)=s_C$, and for $K=-S$, we have $R(K,C)=n$ and
   $D_{\MIN}(K,C)=D(K,C_{\MIN})=D(S,S\cap(-S))=2R(\frac{S-S}{2},S\cap(-S))=n+1$.
\end{proof}

Since for Minkowski-centered triangles $S_{\MIN}=\frac{2}{3}S_{\AM}$, we know $D_{\MIN}(K,S)=\frac{3}{2}D_{\AM}(K,S)$ and the inequalities from \Cref{prop:diagramAM} can be transferred to describe $f_{\MIN}(\bar{\CC}^2,S)$.
\begin{corollary}
   \label{cor:diagramMIN}
   For every Minkowski-centered triangle $S\in\CC^2$, the diagram $f_{\MIN}(\bar{\CC}^2,S)$ is fully described by the  inequalities
   \begin{align*}
       D_{\MIN}(K,S)&\leq 3R(K,C)\\
       2r(K,S)+R(K,S)&\leq  D_{\MIN}(K,S)\\
       \frac{D_{\MIN}(K, C)}{3R(K,C)} \left(1 -\frac{D_{\MIN}(K, C)}{3R(K,C)}\right) &\leq \frac{r(K, C)}{R(K,C)}.
   \end{align*}
\end{corollary}

  As in the case of the harmonic diameter, since $\delta_{\MIN}$ is not the same for every gauge, the diagram for the equilateral triangle cannot be dominating. There does not exist a single gauge which defines the union of the diagrams.
    The following system of inequalities provides an upper bound for the union of the diagrams $f_{\MIN}(\bar{\CC}^2,C)$ over all Minkowski-centered gauges $C\in\CC^2_0$ (\cf~\Cref{fig:diagramMIN}).
    \begin{align*}
        0&\leq r(K,C)\\
        r(K,C)&\leq R(K,C)\\
        D_{\MIN}(K,C)&\leq 3 R(K,C)\\
        r(K,C)+R(K,C)&\leq D_{\MIN}(K,C)\\
         3 R(K,C) &\leq 2D_{\MIN}(K,C)\\
         \frac{D_{\MIN}(K,C)}{2R(K,C)}\left(1-\frac{D_{\MIN}(K,C)}{2R(K,C)}\right) &\leq \frac{r(K,C)}{R(K,C)}.
    \end{align*}
    Moreover, the following parts of the boundary described by the above inequalities are reached:
    \begin{enumerate}[i)]
    \item $r(K,C)=0$ for segments $K=L_D$ for gauges $C$ with $s(C)\in[1,2]$
    \item $r(K,C)= R(K,C)$ for $K=C$ for gauges $C$ with $s(C)\in[1,2]$
    \item $D_{\MIN}(K,C)= 3 R(K,C)$ with $C$ being a triangle as in $f_{\MIN}(\bar{\CC}^2,S)$
    \item $3 R(K,C) = 2D_{\MIN}(R,C)$ for $K=-S$ and $C=S \cap s(-S)$ with $s\in[1,2]$
    \item $r(K,C)+R(K,C)= D_{\MIN}(K,C)$ for $K=\lambda (-S) +(1-\lambda)S_{\MIN}$ with $\lambda\in[0,1]$ and $C=S_{\MIN}$
    \end{enumerate}
    The first two inequalities are trivial. The third and the fourth follow from \Cref{lem:basicineq} and the fifth from \Cref{jungineqDmin}.
    The last inequality follows from \Cref{prop:diagramAM},\Cref{contandfactor} and the fact that $R(K,C)\leq D_{\MIN}(K,C)$.
    The equality cases follow from \Cref{linearityrrRD}, \Cref{CDiam}, \Cref{cor:diagramMIN}, and the fact that $R(-S,S_{\MIN})=2$.

  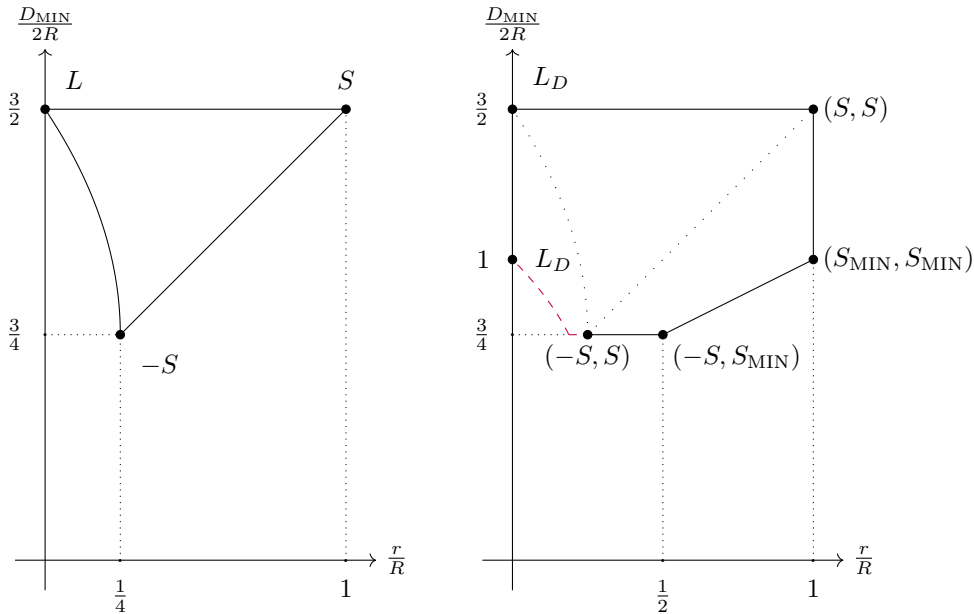
\begin{figure}[ht]
    \centering
    \noindent\begin{minipage}[c]{.4\linewidth}
      \begin{tikzpicture}[scale=4]
        \draw[->] (-0.1,0) -- (1.1,0) node[right] {$\frac{r}{R}$};
        \draw[->] (0,-0.1) -- (0,1.7) node[above] {$\frac{D_{\MIN}}{2R}$} ;
        \draw[ domain=0:1, smooth, variable=\x] plot ({\x}, {1.5}) ;
        \draw[ domain=0.25:1, smooth, variable=\x] plot ({\x}, {\x+0.5});
        \draw[ domain=0.75:1.5, smooth, variable=\y] plot ({(4/9)*(\y)*(1.5-\y)},{\y} );
        \draw[ domain=0:0.25, dotted, variable=\x] plot ({\x}, {0.75}) ;
        \node[label=left:$\frac{3}{4}$] (jung) at (0,0.75) {.};
        \draw[ domain=0:1.5, dotted, variable=\y] plot ({1}, {\y}) ;
        \node[label=below:$1$] (one) at (1,0) {.};
        \draw[ domain=0:0.75, dotted, variable=\y] plot ({0.25}, {\y});
        \node[label=below:$\frac{1}{4}$] (minus) at (0.25,0) {.};
        \node[label=left:$\frac{3}{2}$] (delta) at (0,1.5) {.};
        \node[label=above:$S$] (b2) at (1,1.5) {.};
        \draw[fill=black] (1,1.5) circle[radius=0.4pt];
        \node[label=above right:$L$] ($L$) at (0,1.5) {.};
        \draw[fill=black] (0,1.5) circle[radius=0.4pt];
        \node[label=below right:$-S$] (sm) at (0.25,0.75) {.};
        \draw[fill=black] (0.25,0.75) circle[radius=0.4pt];
      \end{tikzpicture}
    \end{minipage}
    \noindent\begin{minipage}[c]{.4\linewidth}
      \begin{tikzpicture}[scale=4]
       \draw[->] (-0.1,0) -- (1.1,0) node[right] {$\frac{r}{R}$};
       \draw[->] (0,-0.1) -- (0,1.7) node[above] {$\frac{D_{\MIN}}{2R}$} ;
       \draw[ domain=0:1, smooth, variable=\x] plot ({\x}, {1.5}) ;
       \draw[ domain=0.25:1, loosely dotted, variable=\x] plot ({\x}, {\x+0.5});
       \draw[ domain=0.75:1.5, loosely dotted, variable=\y] plot ({(4/9)*(\y)*(1.5-\y)},{\y} );
       \draw[ domain=0.25:0.5, smooth, variable=\x] plot ({\x}, {0.75}) ;
       \draw[ domain=0.5:1, smooth, variable=\x] plot ({\x}, {0.5 * \x +0.5}) ;
       \draw[ domain=1:1.5, smooth, variable=\y] plot ({1},{\y} );
       \draw[ domain=0:0.1875, dotted, variable=\x] plot ({\x}, {0.75}) ;
       \draw[ domain=0.75:1, dashed, purple, variable=\y] plot ({(\y)*(1-\y)},{\y} );
       \draw[ domain=0.1875:0.25, dashed, purple, variable=\x] plot ({\x},{0.75} );
       \node[label=left:$\frac{3}{4}$] (jung) at (0,0.75) {.};
       \node[label=left:$1$] (jung) at (0,1) {.};
       \draw[ domain=0:1, dotted, variable=\y] plot ({1}, {\y}) ;
       \draw[ domain=1:1.5,smooth, variable=\y] plot ({0}, {\y}) ;
       \node[label=below:$1$] (one) at (1,0) {.};
       \draw[ domain=0:0.75, dotted, variable=\y] plot ({0.5}, {\y});
       \node[label=below:$\frac{1}{2}$] (minus1) at (0.5,0) {.};
       \node[label=left:$\frac{3}{2}$] (delta) at (0,1.5) {.};
       \node (b2) at (1,1.5) {.};
       \tkzLabelPoint[right](b2){$(S,S)$}
       \draw[fill=black] (1,1.5) circle[radius=0.4pt];
       \node[label=above right:$L_D$] ($L_D$) at (0,1.5) {.};
       \draw[fill=black] (0,1) circle[radius=0.4pt];
       \node[label=right:$L_D$] ($L$) at (0,1) {.};
       \draw[fill=black] (0,1.5) circle[radius=0.4pt];
       \node (sm) at (0.25,0.75) {.};
       \tkzLabelPoint[below](sm){$(-S,S$)}
       \draw[fill=black] (0.25,0.75) circle[radius=0.4pt];
       \node (sm1) at (0.5,0.75) {.};
       \tkzLabelPoint[below right](sm1){$(-S,S_{\MIN}$)}
       \draw[fill=black] (sm1) circle[radius=0.4pt];
       \node (sm2) at (1,1) {.};
       \tkzLabelPoint[right](sm2){$(S_{\MIN},S_{\MIN})$}
       \draw[fill=black] (sm2) circle[radius=0.4pt];
     \end{tikzpicture}
   \end{minipage}
   \caption{The diagram $f_{\MIN}(C^2,S)$ \wrt~a Minkowski-centered triangle $S$ (left) and an upper bound for the union of the diagrams over all Minkowski-centered gauges $C\in\CC^2_0$ (right).}
   \label{fig:diagramMIN}
 \end{figure}

\printbibliography

\bigskip

René Brandenberg -- 
Technical University of Munich, Department of Mathematics, Germany. \\
\textbf{rene.brandenberg@tum.de}

Mia Runge -- 
Technical University of Munich, Department of Mathematics, Germany. \\
\textbf{mia.runge@tum.de}

\vfill\eject
\end{document}